\documentclass{fstyle}

\usepackage[T1]{fontenc}
\usepackage[latin2]{inputenc}

\title{\vspace{-1.8cm} \fontfamily{lmss} \selectfont Infinite-dimensional Teichmüller spaces}

\author{F{\i}rat Ya\c{s}ar  \\ {\small Université de Strasbourg} \\
	{  \href{mailto:yasar.math@gmail.com}{\vspace{.2cm}\large \texttt{yasar.math@gmail.com}}}}
\date{}

\begin{document}

\maketitle

\begin{abstract}
	In this paper, the Teichmüller spaces of surfaces appear from two points of views: the conformal category and the hyperbolic category. In contrast to the case of surfaces of topologically finite type, the Teichmüller spaces associated to surfaces of topologically infinite type depend on the choice of a base structure. In the setting of surfaces of infinite type, the Teichmüller spaces can be endowed with different distance functions such as the length-spectrum distance, the bi-Lipschitz distance, the Fenchel-Nielsen distance, the Teichmüller distance and there are other distance functions. Unlike the case of surfaces of topologically finite type, these distance functions are not equivalent.
		
	We introduce the finitely supported Teichmüller space $ \mathcal{T}^{fs} \big(H_0\big) $  associated to a base hyperbolic structure $ H_0 $ on a surface $ \Sigma $, provide its characterization by Fenchel-Nielsen coordinates and study its relation to the other Teichmüller spaces. This paper also involves a study of the Teichmüller space $ \mathcal{T}_{ls}^{0} \big(H_0\big) $ of  asymptotically isometric hyperbolic structures and its Fenchel-Nielsen parameterization. We show that $ \mathcal{T}^{fs} \big(H_0\big) $ is dense in $ \mathcal{T}_{ls}^{0} \big(H_0\big) $, where both spaces are considered to be subspaces of the length-spectrum Teichmüller space $ \mathcal{T}_{ls} \big(H_0\big) $. Another result we present here is that asymptotically length-spectrum bounded Teichmüller space $ \mathcal{AT}_{ls} \big(H_0\big) $ is contractible. We also prove that if the base surface admits short curves then the orbit of every finitely supported hyperbolic surface is non-discrete under the action of the finitely supported mapping class group $ \mathsf{MCG}^{fs}\big(\Sigma \big) $.
\end{abstract}


{\fontfamily{lmss}
	\selectfont
	\small
	\centering
	\vspace{-.2cm}
	\tableofcontents
}
\thispagestyle{empty}

\section{Introduction}

The Teichmüller space of a connected oriented surface $ \Sigma $ is the space of hyperbolic structures up to homotopy. It also parameterizes other structures such as complex structures on $ \Sigma $ up to homotopy and conformal classes of Riemannian metrics on $ \Sigma $ up to homotopy. If the surface in the spotlight is of finite conformal type, these deformation spaces are essentially equivalent. In this case, the distinction between the theory of conformal deformations and the theory of hyperbolic deformations is a question of perspective. However, in contrast to finite type surfaces, there are several Teichmüller spaces associated to an infinite type surface, and in addition to the distinction of points of views, the theory also differs in some other aspects.

A hyperbolic surface is said to be of \textit{finite topological type} if its fundamental group is finitely generated, otherwise it is of \textit{infinite topological type}. A Riemann surface is of \textit{finite conformal type} if it is obtained from a compact Riemann surface by removing a finite number of points, and otherwise it is said to be of \textit{infinite conformal type}. A surface of finite topological type need not be of finite conformal type; for instance, the complex plane is of finite conformal type since it is obtained from the Riemann sphere by removing a point (A conformal equivalence between the complex plane and a sphere with one point deleted is provided by a stereographic projection.); however, the open unit disk is not a surface of finite conformal type. A surface of infinite topological type is also of infinite conformal type. 

We focus on the Teichmüller spaces associated to surfaces of infinite topological type. The ways of associating Teichmüller spaces to a surface $ \Sigma $ of infinite topological type require the use of a suitable notion of marked surface and a suitable notion of topology on the set of marked surfaces. By a marking of $ \Sigma $ we mean a homeomorphism 
$ f : \Sigma \rightarrow S $
where $ S $ is a surface endowed with a hyperbolic structure (or a conformal structure). Depending on the point of view adapted, the requirement on the markings of  $ \Sigma $ is to be quasiconformal, length-spectrum bounded, Fenchel-Nielsen bounded or bi-Lipschitz bounded with respect to a chosen basepoint, and there are other possibilities. Now we will briefly recall these deformation spaces and we shall be more precise within the details in the following sections. 

\subsubsection*{The quasiconformal Teichmüller space.} In order to compare two conformal structures given together with a homeomorphism between them, one way is to measure the deviation of the mapping from conformality. Let $ \Omega $ be a domain in $ \C $ and $ f : \Omega \rightarrow \C $ be an orientation preserving $ C^{1} $ homeomorphism onto $ f(\Omega)$. The mapping $ f$ takes the infinitesimal circles centered at a point $ z $ to  ellipses centered at $ f(z) $. The dilatation of $ f $ at the point $ z $ is the quantity
\begin{align*}
	\K(z) = \dfrac{\big| f_{z}(z) \big| + \big| f_{\overline{z}}(z) \big|}{\big| f_{z}(z) \big| - \big| f_{\overline{z}}(z) \big| } 
\end{align*}
and one can show that this is equal to the ratio of the major axis of the corresponding ellipse to its minor axis. Such a mapping is said to be quasiconformal if $ \K (z) $ is uniformly bounded on $ \Omega $, that is, if the dilatation $ \K(f)  $ of $ f $ satisfies
\begin{align*}
	\K(f) = \sup_{z\in \Omega} \K(z) < \infty
\end{align*}

Since quasiconformality is a local property and a conformal invariant, the notion of quasiconformal mappings can be naturally defined for Riemann surfaces. From the point of view of the conformal category, the quasiconformal Teichmüller space $ \teich_{qc}\big(R_0\big) $ associated to a surface $ \Sigma $ of infinite topological type equipped with a base conformal structure $ R_0 $ is defined to be the space of marked Riemann surfaces up to homotopy where the marking has finite complex dilatation. In other words, the marking $ f : \Sigma \rightarrow R $ is required to be a quasiconformal map. We do not need such a requirement for the case of finite conformal type surfaces since in the homotopy class of any homeomorphism there is always a homeomorphism which is quasiconformal. But for the surfaces of infinite type there are homeomorphisms which are not homotopic to quasiconformal ones (see for instance, \cite{MR2792982} or \cite{article_1}). In order to compare conformal structures the natural way is to search for a quasiconformal homeomorphism with minimal complex dilatation. In this setting, the Teichmüller distance between any two points $ \big[ f_1,R_1 \big] $ and $ \big[ f_2,R_2 \big] $ in $ \teich_{qc}\big(R_0\big) $ is defined to be
\begin{align*}
	d_{qc}\bigg( \big[ f_1,R_1 \big], \big[ f_2,R_2 \big] \bigg) = \dfrac{1}{2} \inf_{f} \text{ } \log \K(f) 
\end{align*}
where $ f $ ranges over the set of all quasiconformal homeomorphisms from $ R_1 $ to $ R_2 $ homotopic to $ f_{2} \circ f^{-1}_{1} $.

\subsubsection*{The length-spectrum Teichmüller space.}  If $ \Sigma $ is a surface of topologically infinite type, now endowed with a base hyperbolic structure $ H_0 $, in order to define its Teichmüller space we consider the hyperbolic surfaces homeomorphic to $ \Sigma $ together with a marking which is a homeomorphism from $ \Sigma $ to the hyperbolic surface. The length-spectrum constant of a marking $ f : \Sigma \rightarrow H $ is the quantity
\begin{align*}
	\L(f) = 
	\sup_{\alpha \in \S (\Sigma)} 
	\Bigg\{ 
	\dfrac{\ell_{H}(f(\alpha))}{\ell_{H_{0}}(\alpha)}
	\text{ , }
	\dfrac{\ell_{H_{0}}(\alpha)}{\ell_{H}(f(\alpha))}
	\Bigg\} 
\end{align*}
where $ \S (\Sigma) $ denotes the set of homotopy classes of essential simple closed curves on $ \Sigma $ and where $ \ell_{H_0} (\alpha) $ denotes the length of the unique simple closed geodesic in a homotopy class $ \alpha $ on a hyperbolic surface $ H_0 $. Such a homeomorphism $ f $ is said to  be length-spectrum bounded if $ \L(f) < \infty $. In general, a homeomorphism from the surface $ \Sigma $ equipped with a base hyperbolic structure $ H_0 $ to a hyperbolic surface need not be length-spectrum bounded. One example of this can be found in  \cite{MR2792982}. In order to endow the space of marked hyperbolic structures with a distance function which compares the spectrum of lengths of simple closed curves, we require the markings to be length-spectrum bounded with respect to a base point. In this setting, we compare two marked hyperbolic surfaces $ \big[f_1, H_{1}\big] $ and $ \big[f_2, H_{2}\big] $ by the following quantity
\begin{align*}
	d_{ls}\bigg(\big[f_{1}, H_{1}\big], \big[f_{2}, H_{2}\big] \bigg) = 
	\dfrac{1}{2} 
	\log
	\sup_{\alpha \in \S(\Sigma)} 
	\Bigg\{ 
	\dfrac{\ell_{H_{1}}(f_1(\alpha))}{\ell_{H_{2}}(f_2(\alpha))}
	\text{ , }
	\dfrac{\ell_{H_{2}}(f_2(\alpha))}{\ell_{H_{1}}(f_1(\alpha))}
	\Bigg\} 
\end{align*}
The resulting Teichmüller space consists of the length-spectrum bounded marked hyperbolic surfaces up to homotopy. It is called the length-spectrum Teichmüller space $ \teich_{ls}\entre{H_0}$ of the hyperbolic surface $ H_0 $ and it is endowed with the length-spectrum metric $ d_{ls} $.

\subsubsection*{Fenchel-Nielsen Teichmüller space.} If $ \Sigma $ carries a hyperbolic structure with a geodesic pair of pants decomposition we can define the Fenchel-Nielsen parameters which consist of the \textit{length parameters} $ \ell_x (C_i) $ and \textit{twist parameters} $ \tau_{x}(C_i) $ for each marked hyperbolic surface $ x = \big[f,H\big] $. In \cite{MR2865518} the authors give a necessary and sufficient condition called \textit{Nielsen convexity} that guarantees that all simple closed curves in a topological pants decomposition can be straightened simultaneously so that the resulting collection of simple closed geodesics forms a \textit{geodesic} pair of pants decomposition of the surface. Fixing a geodesic pair of pants decomposition $ \P = \{ C_i \} $ we can define Fenchel-Nielsen parameters of each marked hyperbolic structure $ \big[f,H\big] $ (or marked conformal structure) by straightening the curves in the pair of pants decomposition with respect to the pullback metric induced by the marking homeomorphism $ f $. We fix a base point $ x_0 = \big[ \id , H_0 \big] $ and consider all marked hyperbolic surfaces which are Fenchel-Nielsen bounded, that is,
\begin{align*}
	\sup_{i \in \N} \max
	\Bigg\{
	\Bigg| \dfrac{\log\ell_x (C_i)}{\log \ell_{x_0} (C_i)} \Bigg|
	\text{ , }
	\big| \ell_x(C_i)\tau_{x}(C_i) - \ell_{x_0}(C_i)\tau_{x_0}(C_i) \big|
	\Bigg\}
	< \infty
\end{align*}
This allows us to define the Fenchel-Nielsen Teichmüller space $ \teich_{FN,\P} \entre{H_0} $ of the hyperbolic surface $ H_0 $, which is endowed with the distance function
\begin{align*}
	d_{FN}(x,y) =
	\sup_{i \in \N}
	\Bigg\{
	\Bigg| \dfrac{\log\ell_x (C_i)}{\log \ell_{y} (C_i)} \Bigg|
	\text{ , }
	\big| \ell_x(C_i)\tau_{x}(C_i) - \ell_{y}(C_i)\tau_{y}(C_i) \big|
	\Bigg\}
\end{align*}
where $ x $ and $ y $ are Fenchel-Nielsen bounded marked hyperbolic structures up to homotopy. As it depends on the base points $ H_0 $, the Fenchel-Nielsen Teichmüller space strongly depends on the choice of pair of pants decomposition $ \P $ as well. When the pair of pants decomposition $ \P $ is understood from the context we omit the subscript in the notation and simply denote the Fenchel-Nielsen Teichmüller space of $ H_0 $ by $ \teich_{FN} \entre{H_0}$. With Fenchel-Nielsen coordinates, it is possible to parameterize the quasiconformal Teichmüller space $ \teich_{qc}\entre{H_0}$ and the length-spectrum Teichmüller spaces $ \teich_{ls}\entre{H_0}$ under certain conditions on the lengths of simple closed geodesics on $ \Sigma $, see \cite{MR2865518} and \cite{MR3449399} respectively.

From these points of view, the comparability of two Riemann (or hyperbolic) surfaces in the case of surfaces of infinite topological type is possible under these requirements. As we will see in the next section, the definition of finitely supported Teichmüller spaces does not stipulate such a requirement. 

We shall also note that the three spaces $ \teich_{qc} \big(\Sigma\big) $,  $ \teich_{ls} \big(\Sigma\big) $ and $ \teich_{FN} \big(\Sigma\big) $ are closely related where the surface $ \Sigma $ is endowed with a Riemann surface structure $ R $. The complex structure on the surface $ \Sigma $ gives rise to a natural hyperbolic metric called the \textit{intrinsic metric}  \cite{MR2865518}. If $ \alpha $ is an essential simple closed curve on $ \Sigma $ we can consider the length $ \ell_{R} (\alpha) $ of the unique simple closed geodesic for the intrinsic metric on $ R $ in the homotopy class of $ \alpha $. We note that the quantity $\ell_{R} (\alpha)  $ does not depend on the choice of conformal structure homotopic to $ R $, in other words, it depends only on the conformal class of $ R $. \textit{Wolpert's lemma} \cite{MR624836} states that if $ f : R_1 \rightarrow R_2 $ is a $K$-quasiconformal map and $ \alpha $ is a simple closed curve on $ R_1 $ then we have 
\begin{align*}
	\dfrac{1}{K} \leq \dfrac{\ell_{R_2}(f(\alpha))}{\ell_{R_1}(\alpha)} \leq K 
\end{align*}
from which it follows that there is a natural 1-Lipschitz inclusion map
\begin{align}\label{inclusion_map}
	\bigg( \teich_{qc}\big(\Sigma\big) , d_{qc} \bigg) \hookrightarrow \bigg( \teich_{ls}\big(\Sigma\big) , d_{ls} \bigg)
\end{align}
that is, $ \teich_{qc}\big(\Sigma\big) \subset \teich_{ls}\big(\Sigma\big) $. In general, the converse inclusion does not hold, see \cite{MR2792982}. We know that under certain conditions on the the base structure, the inclusion map (\ref{inclusion_map}) is surjective. A hyperbolic structure $ H $ on a surface $ \Sigma $ is said to satisfy \textit{Shiga's condition} if $ \Sigma $ admits a geodesic pair of pants decomposition $ \P = \{ C_{i}\}_{i \in I}$ such that there exists a constant $ M > 0 $, for all $ i \in I $
\begin{align}\label{shiga-condition}
	\frac{1}{M} \leq \ell_{H}(C_{i}) \leq M 
\end{align}
In \cite[Theorem 8.10]{MR2865518} and \cite[Theorem 4.14]{MR2792982} it is proven that if $ R $ satisfies Shiga condition, then we have the set-theoretical equality $ \teich_{qc} \entre{\Sigma} = \teich_{ls} \entre{\Sigma} = \teich_{FN} \entre{\Sigma} $, and moreover the identity map between these three spaces is locally bi-Lipschitz homeomorphism.

\subsubsection*{Finitely supported Teichmüller space.} A finitely supported marked hyperbolic surface on $ \Sigma $ consists of a homeomorphism $ f : H_0 \rightarrow H $ and a hyperbolic structure $ H $ for which the pullback $ f^{*}H $ coincides with the base structure $ H_0 $ outside a subsurface $ E $ of $ \Sigma $. Up to homotopy, the collection of equivalence classes of marked hyperbolic surfaces that are finitely supported with respect to the base surface $ H_0 $ forms a space which we call the finitely supported Teichmüller space and denote by $ \teich^{fs}\entre{H_0} $.  

Now we briefly outline the results of the paper.

In  \autoref{Section:cs} we show that $ \teich^{fs}\entre{H_0} $ can be embedded into $ \teich_{ls} \entre{H_0} $ in a natural way, in the sense that the inclusion map $ \teich^{fs}\entre{H_0} \hookrightarrow \teich_{ls}\entre{H_0} $ is continuous where the topology that we consider on $ \teich^{fs}\entre{H_0} $ is the one defined by geodesic length functions. The proof we present here uses only the tools of hyperbolic geometry (\autoref{theorem:finitely_supported_is_ls}). We also prove the analogous statement for the conformal context (\autoref{theorem:inclusion_qc}).

A closely related object to $ \teich^{fs} \entre{H_0} $ is the space $ \teich_{ls}^{0}\entre{H_0} $ of marked hyperbolic structures which are asymptotically isometric to $ H_0 $. This space consists of points $ \big[f,H\big] $ satisfying that for every $ \epsilon > 0 $ there exists a topologically finite type subsurface $ E $  of $ \Sigma $ such that the length-spectrum constant $ \L(f_{\restriction_{\Sigma-E}}) $ is less than $ 1+\epsilon $. We consider these two spaces as subspaces of $ \teich_{ls}\entre{H_0} $ and prove that the closure of $ \teich^{fs} \entre{H_0} $ with respect to the length-spectrum metric is $ \teich_{ls}^{0}\entre{H_0} $ (\autoref{theorem:ls_dense}). Using this, we show that the metric space $ \big( \teich_{ls}^{0}\entre{H_0} , d_{ls} \big) $ is complete (\cref{corollary:T_0_complete}).

In \autoref{Sec:FN-coordinates} we provide the Fenchel-Nielsen characterization of the finitely supported Teichmüller space $ \teich^{fs}\entre{H_0} $ in \cref{theorem:FN_T_fs}; that of  $\teich_{ls}^{0}\entre{H_0} $ in \cref{theorem:FN_T_0}. We show asymptotically length-spectrum bounded Teichmüller space $ \asyteich_{ls} \entre{H_0} $ is contractible \cref{corollary:asym_contractible}.

\red{Section 4} contains a study of mapping class groups of surfaces of topologically infinite type. We prove in \cref{theorem:nondiscreteness} that if the base surface admits short curves then the orbit of every finitely supported hyperbolic surface is non-discrete under the action of the finitely supported mapping class group $ \mcg^{fs}\entre{\Sigma} $.

\subsection{Basic definitions, conventions and topology of surfaces of infinite type}
In order to be precise, we first recall some basic definitions. A \textit{surface with boundary} is a second countable, connected Hausdorff space in which every point has a neighborhood homeomorphic to either $ \R^{2} $ or $ \R \times [0, \infty) $. If a point $ p $ has a neighborhood homeomorphic to $ \R^{2} $ it is an \textit{interior point}, whereas if it has a neighborhood homeomorphic to $ \R \times [0, \infty) $ where the homeomorphism sends $ p $ to a point in $ \R \times \{ 0 \} $ it is a \textit{boundary point}. The boundary $ \partial \Sigma $ is the set of boundary points called the \textit{boundary} of $ \Sigma $. The boundary components of a surface consists of a finite or countable collection of connected components which are either lines or circles. If $ \partial \Sigma = \emptyset $ then $ \Sigma $ is a \textit{surface without boundary}. In order to avoid any ambiguity, here we shall note the distinction between "surface boundary" and "topological boundary". If $ Y $ is a subset of a topological space $ X $, the \textit{topological boundary} of $ Y $ in $ X $ is the set of of points in $ X $ which can be approached both by elements in $ Y $ and by elements outside $ Y $. 

By a \textit{subsurface} $ E $ of $ \Sigma $ we mean a connected relatively compact subset whose interior is a surface without boundary with the topology induced from the topology of $ \Sigma $ and whose boundary consists of a finite number of nonintersecting simple closed curves. Throughout this text, a subsurface $ E $ may possibly have finitely many punctures and we assume that the boundary curves of a subsurface $ E $ are not homotopically trivial in $ \Sigma $ and are not homotopic to each other (that is, $ E $ is not an annulus). The latter ensures that the inclusion of the boundary curves of $ E $ into $ \Sigma $ induces a monomorphism of the fundamental group of the subsurface into that of $\Sigma $. 

A surface $ \Sigma $ is called \textit{topologically finite} if its fundamental group is finitely generated. The topological type of such a surface is determined by its orientability class and by three discrete invariants, namely, the genus $ g $, the number of boundary components $ b $, and the number of punctures $ p $. For the classification theorem for the surfaces of finite type see the classical textbook \cite{MR575168} or  Conway's ZIP proof \cite{MR1699257} which does not use a standard form. For orientable surfaces we shall use the notation $ (g,b,p) $ to signify the type of the surface. In the case $ b = 0 $ or $ p = 0 $ we shall omit the corresponding entry to simplify the notation. If the surface is not topologically finite then by definition it is a \textit{surface of topologically infinite type}. There is a classification of surfaces of infinite type. In contrast to the case of surfaces of finite type, to distinguish surfaces of infinite type, in addition to discrete invariants we need continuous invariants namely the space of ends (or ideal boundary), and the space of non-planar ends. This was originally done in the work of Kerékjártó \cite{kerekjarto1923vorlesungen} and later revisited by Richards \cite{MR0143186} in which the classification theorem is obtained for the case $ \partial \Sigma = \emptyset $. The paper \cite{MR2308580} of Prishlyak and Mischenko involves an update for the case of surfaces with boundary. In  \cite{article_1} we review the topological classification of surfaces of infinite topological type and describe a way to construct pair of pants decompositions and ideal triangulations on a given surface.

A \textit{closed curve} on a surface $ \Sigma $ is a continuous map from the unit circle $ \mathbb{S}^{1} $ to $ \Sigma $. It is called \textit{simple} if this map is an injection. A simple closed curve on $ \Sigma $ is \textit{essential} if it is neither homotopic to a point nor a to a puncture. It can possibly be homotopic to a boundary component. We denote the set of homotopy classes of essential simple closed curves on $ \Sigma $ by $ \S = \S (\Sigma)$. Given a hyperbolic structure $ H $ on $ \Sigma $, in each homotopy class $ \alpha $ there is a unique $ H $-geodesic and its length is denoted by $ \ell_{H}(\alpha) $.

For any two elements $ \alpha $ and $ \beta $ in  $ \mathscr{S}(\Sigma) $ the \textit{intersection number} $ i(\alpha, \beta )$ is defined to be the minimum number of intersection points between a representative curve in the homotopy class $ \alpha $ with a representative curve in the homotopy class $ \beta $. 

A \textit{generalized pair of pants} is a surface of type $ (0,b,p) $ where $ b+p $, the number of the \textit{holes}, is equal to 3. The boundary of a pair of pants is a (possibly empty) union of disjoint simple closed curves and its interior is a surface of type $ (0,3)$. A \textit{generalized pair of pants decomposition} of a surface $ \Sigma $ is a collection $ \P = \{ C_{i}\}_{i \in I } $ of pairwise disjoint simple closed curves which decomposes the surface into connected components each homeomorphic to a generalized pair of pants. Such a collection $ \P = \{ C_{i}\}_{i \in I } $ is either finite or countable. It is finite when the surface is of finite topological type. Otherwise, the surface is of infinite topological type and the collection is countably infinite since any surface admits a countable basis for its topology (second countability) and the generalized pair of pants in the decomposition of the surface have pairwise disjoint interiors. We shall also note that every curve in the collection $ \P = \{ C_{i}\}_{i \in I } $ lies in one of the two kinds of subsurfaces of type $ (0,4) $ or $ (1,1) $. To see this let $ C_{j} $ be a curve in $ \P $. If we cut the surface along the curves $ C_{i} $ where $ i \in I $, $ i \neq j $, we get a collection of pairs of pants and a component containing $ C_{j} $. This component is homeomorphic to either a sphere with four holes or a torus with one hole. A surfaces of type $ (0,4) $ or $ (1,1) $ is called a X-type surface. See the \autoref{figure:two_kinds} below. 

\begin{figure}[ht!]
	\centering
	\includegraphics[width=0.6\linewidth]{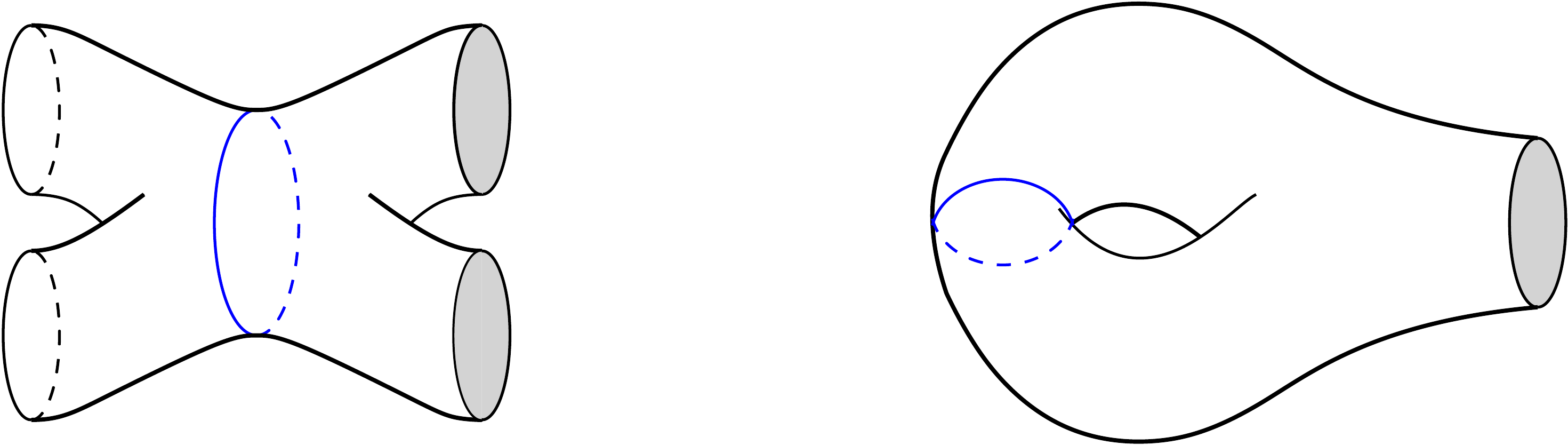}
	\caption{\small {$ \Sigma_{0,4} $ (on the left) and $ \Sigma_{1,1} $ (on the right)}
		\label{figure:two_kinds}}
\end{figure}

We shall open a parenthesis here on the number of curves in a pair of pants decomposition of a finite type surface. The Euler characteristic $ \chi (\Sigma_{g})$ of a compact surface $ \Sigma_{g} $ of genus $ g \geqslant 2 $ is $ 2- 2g $, and that of a pair of pants $ P $ is $ \chi(P)=-1 $. Suppose $ \Sigma_{g} $ is decomposed into $ n $ pairs of pants by $ m $ decomposing simple closed curves. Since the Euler characteristic of the disjoint union of these pairs of pants is the sum of their Euler characteristics, we have $ \chi (\Sigma_{g}) = n \chi(P) $ thus $ n = 2g-2 $. There are $ 3n $ boundary curves in the decomposition and these curves match up in pairs to form simple closed curves in $ \Sigma_{g} $ which shows $ 3n=2m $, hence $ m = 3g-3 $. Using this, we can count the number of curves in a generalized pair of pants decomposition of a surface of type $ (g, p, b) $ with $ b > 0 $. First we can remove each puncture (if there is any) by replacing it by a boundary curve to form a surface with boundary of type $ (g, 0, b + p) $. Such a surface can always be doubled along its boundary curves by taking two copies of the surface and identifying pointwise by the identity map the boundary curves of these surfaces. The resulting surface is a compact surface without boundary of genus $ 2g + b -1 $ which can be decomposed into generalized pairs of pants by $ 6g - 6 +3b $ simple closed curves. By symmetry it follows that a generalized pair of pants decomposition of a surface $ \Sigma $ of type $ (g, p, b) $ consists of $ n = 3g - 3 + b $ curves (without counting the boundary components) decomposing $ \Sigma $ into $ 2g-2 + b + p $ generalized pairs of pants. 

The existence of pair of pants decomposition on a surface of infinite type can be obtained as a consequence of  Zorn's Lemma. It is also possible to construct a pair of pants decomposition using a standard form of the surface, see \cite{article_1} for details. 
It has been proven in \cite{MR2865518} that every pair of pants decomposition on a hyperbolic surface of infinite type can be turned into a geodesic pair of pants if the hyperbolic surface satisfies a criterion called \textit{Nielsen convexity}. We will come back to this discussion in \cref{Sec:FN-coordinates}.

A surface $ \Sigma $ endowed with a hyperbolic structure $ H $ is called \textit{upper bounded} if it admits a geodesic pair of pants decomposition $ \P = \{ C_i \} $ such that there is a uniform upper bound for $ \ell_{H}(C_i) $, in other words,
\begin{align*}
	\sup_{i} \ell_{H} (C_i) < \infty 
\end{align*}
Similarly, it is called \textit{lower bounded} if it admits a geodesic pair of pants decomposition $ \P = \{ C_i \} $ satisfying
\begin{align*}
	\inf_{i} \ell_{H} (C_i) > 0 
\end{align*}
Note that it is clear that if $ H $ is upper bounded and lower bounded then it satisfies Shiga condition (\ref{shiga-condition}), and vice versa. Finally, we say that $ H $ \textit{admits short interior curves} if there is a sequence of homotopy classes of disjoint simple closed curves $ \alpha $ on $ \Sigma $ such that $ \ell_{H}(\alpha_{n}) \rightarrow 0 $.

It is important to point out the distinction between the reduced theory and the non-reduced theory of Teichmüller spaces. In the non-reduced Teichmüller theory the homotopy maps used in the definition of equivalence classes of marked surfaces are required to keep the points of the boundary components pointwise fixed. We will focus on the reduced theory  in which we do not ask the homotopies to preserve pointwise the boundary components of the surface. For simplicity we omit the word \textit{reduced} when we speak of reduced Teichmüller spaces and reduced mapping class groups throughout the text. 

As a convention in this paper, whenever we endow the surface $ \Sigma $ with a hyperbolic structure, the boundary of $ \Sigma $, if it is nonempty, consists of simple closed geodesics. 

\subsection{Hyperbolic trigonometry}\noindent
We record some formulas concerning hyperbolic trigonometry and prove a lemma which we will use in the following section. The propositions below are well-known and their proofs can be found in Fenchel \cite[page 86]{MR1004006}.

\begin{proposition}[Right-angled hexagon]
	Let $ a, b, c $ denote the non-adjacent sides of a right-angled hexagon, and $ a', b', c' $ be the opposite sides respectively. 
	Then 
	\begin{align}\label{hexagon_1}
		\cosh a' = \dfrac{\cosh a +\cosh b \cosh c}{\sinh b \sinh c}
	\end{align}
	and
	\begin{align}
		\dfrac{\sinh a}{\sinh a'} = \dfrac{\sinh b}{\sinh b'} = \dfrac{\sinh c}{\sinh c'}
	\end{align}
	As a corollary, any three non-adjacent sides of such a hexagon determine the remaining ones.
\end{proposition}

\begin{figure}[h!]
	\centering
	\includegraphics[width=0.3\linewidth]{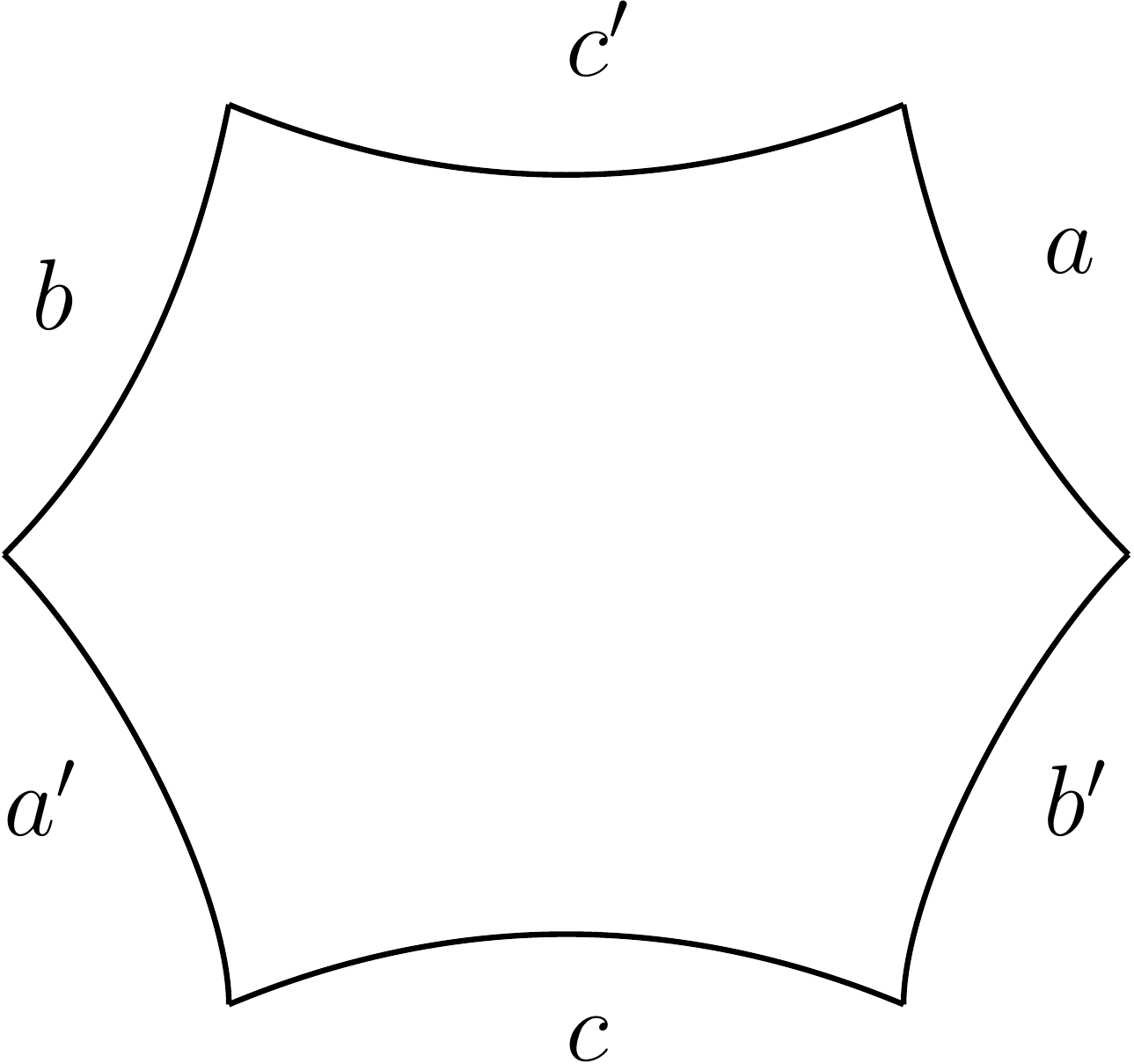}
	\caption{\small {Right-angled hexagon}\label{figure:hexagon_formula}}
\end{figure}

\begin{lemma}\label{lemma:bound_hexagon}
	Let $ \H $ be a right-angled hexagon with non-adjacent sides $ a, b, c $. There exist $ m > 0 $  depending on $ a, b $ and $ c $ such that all the sides of $ \H $ is lower bounded by $ 1/m $ and upper bounded by $ m $. 
\end{lemma}

\begin{proof} 
	Take an $ m' > 0 $ so that 
	\begin{align*}
		\frac{1}{m'} < a, b , c < m'
	\end{align*}
	It suffices to bound $ \cosh a' $. Using (\ref{hexagon_1}) together with $ \cosh x \leq e^x$ and $ x \leq \sinh x $
	\begin{align*}
		\dfrac{\cosh a + \cosh b \cosh c}{\sinh b \sinh c}  
		\leq
		\dfrac{e^{m'} + e^{2m'}}{(1/m')^{2}}
	\end{align*}
	Observe that $ 1+ e^{-2x} \leq \dfrac{\cosh x }{\sinh x} = \dfrac{1+ e^{-2x}}{1- e^{-2x}} $ for $ x > 0 $ and we use this to get
	\begin{align*}
		\dfrac{\cosh a} {\sinh b \sinh c} + \dfrac{\cosh b \cosh c}{\sinh b \sinh c} 
		\geq
		\dfrac{1}{e^{2m'}} + (1+ e^{-2b})(1+ e^{-2c})
		\geq
		1 + e^{-\frac{2}{m'}}
	\end{align*}
	which shows 
	\begin{align*}
		1 + e^{-\frac{2}{m'}} \leq \cosh a' \leq \dfrac{e^{m'} + e^{2m'}}{(1/m')^{2}}
	\end{align*}
	as required.
\end{proof}

\begin{proposition}[quadrilateral with two adjacent right angles]\label{lemma:bound_quadrilateral}
	Let $ a, c', b, c $ denote the sides of a quadrilateral where the angle between $ a $ and $ c' $ and the angle between $ c' $ and  $ b $ are right angles. 
	If the quadrilateral is convex then 
	\begin{align}\label{bound_convex_quadrilateral}
		\cosh c = \cosh a \cosh b \cosh c' - \sinh a \sinh b
	\end{align}
	and
	\begin{align}\label{bound_convex_quadrilateral_2}
		\dfrac{\cosh a}{\cos \alpha} = \dfrac{\cosh a}{\cos \beta} = \dfrac{\sinh c}{\sinh c'}
	\end{align}
	If the quadrilateral is non-convex then
	\begin{align}\label{bound_nonconvex_quadrilateral}
		\cosh c = \cosh a \cosh b \cosh c' + \sinh a \sinh b
	\end{align}
\end{proposition}

\begin{figure}[h!]
	\centering
	\includegraphics[width=0.8\linewidth]{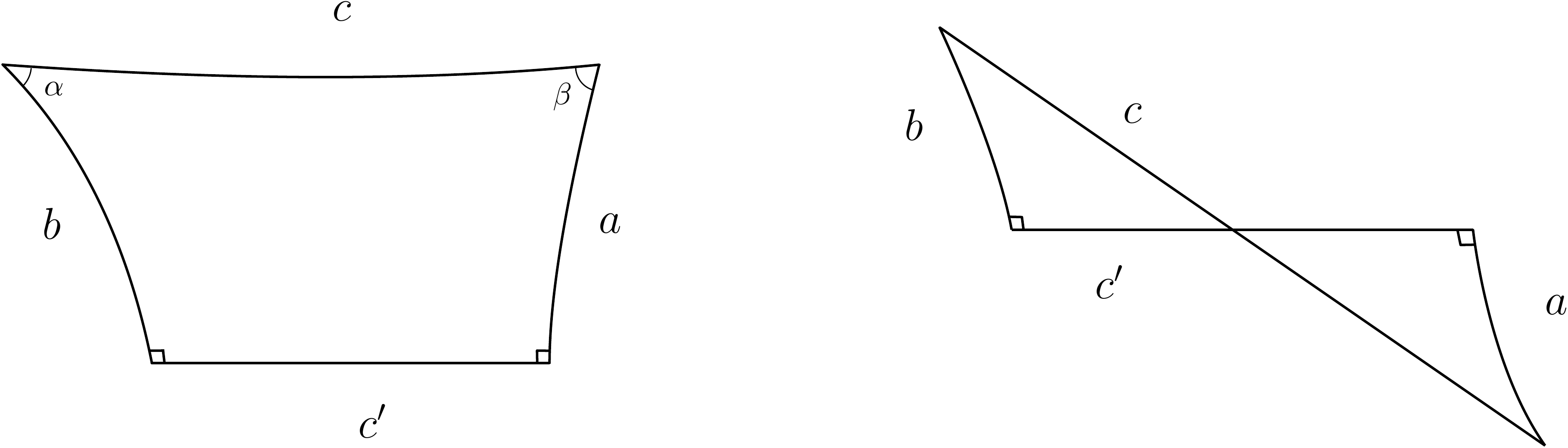}
	\caption{\small {This figure represents a convex (on the left) and a nonconvex quadrilateral (on the right)}\label{figure:quadrilateral_formula}}
\end{figure}

\begin{lemma}[Bound for lengths of geodesic arcs on surfaces of type $ (0,4) $ and $ (1,1) $]\label{lemma:bound_for_X}
	Let $ \X $ be hyperbolic surface. Suppose that we have one of the following configurations for $ \X $.
	\begin{itemize}[nosep]
		\item[i.] $ \X $ is of type $ (0,4) $ with boundary components  $ \partial_1, \partial_2, \partial_3, \partial_4 $ and  $ \beta $ is an interior simple closed geodesic of length $ \ell  $ which decomposes $ \X $ into two pairs of pants $ ( \partial_1, \partial_2, \beta )$ and $ ( \beta, \partial_3, \partial_4)$. Let $ \gamma $ be a simple geodesic path on $ \X $ joining a point  $ p \in \partial_{i} $ to a point $ q \in \partial_{j} $. 
		
		\item[ii.] $ \X $ is of type $ (1,1) $ with boundary component $ \partial $ and $ \beta $ is an interior simple closed geodesic of length $ \ell  $. Let $ \gamma $ be a simple geodesic path on $ \X $ joining a point  $ p \in \partial $ to a point $ q \in \partial $. 
	\end{itemize}
	
	Then for each situation, there is positive real number $ M $ depending on the length $ \ell $ of $ \beta $ and length of the boundary geodesics of $ \X $ such that 
	\begin{align*}
		\dfrac{1}{M} \leq \ell(\upgamma) \leq M 
	\end{align*}
\end{lemma}

\begin{proof}
	Suppose first that $ \X $ is of type $ (0,4) $. To simplify the notation, we shall use the letters $ a, b, c, d $ to denote lengths of $ \ell(\partial_1),  \ell(\partial_2), \ell(\partial_3) $ and $ \ell(\partial_4), $ respectively. The geodesic seam between two boundary components say $ a $ and $ b $ will be denoted by $ \delta_{ab} $ and we denote by $ a_p $ the length of the path from $ p $ to $ \delta_{ab} $; and we use the analogous notations in what follows. 
	
	We can think of $ \X $ as a union of $ 4 $ hexagons by cutting it along $ \beta $ and the geodesic seams of the two resulting pair of pants, then \cref{lemma:bound_hexagon} says that there is an $ m $ so that $ 1/m $ is a lower bound and $ m $ is an upper bound for lengths of geodesic seams in $ \X $ and for $ a, b, c, d $ and $ \ell $.
	We can suppose that the initial point $ p $ of the geodesic path $ \upgamma $ is on $ \partial_1 $, since the situation is symmetric for the other cases. 
	There are several cases to deal with depending on where $ q $ is placed and on whether $ \upgamma $ intersects some geodesic seams.
	
	In the first two cases the geodesic path $ \upgamma $ lies inside exactly one of the two pair of pants. We shall note that the remaining cases can be reduced to these two cases as we will see.
	\begin{figure}[ht!]
		\centering
		\begin{subfigure}{7cm}
			\centering\includegraphics[width=4cm]{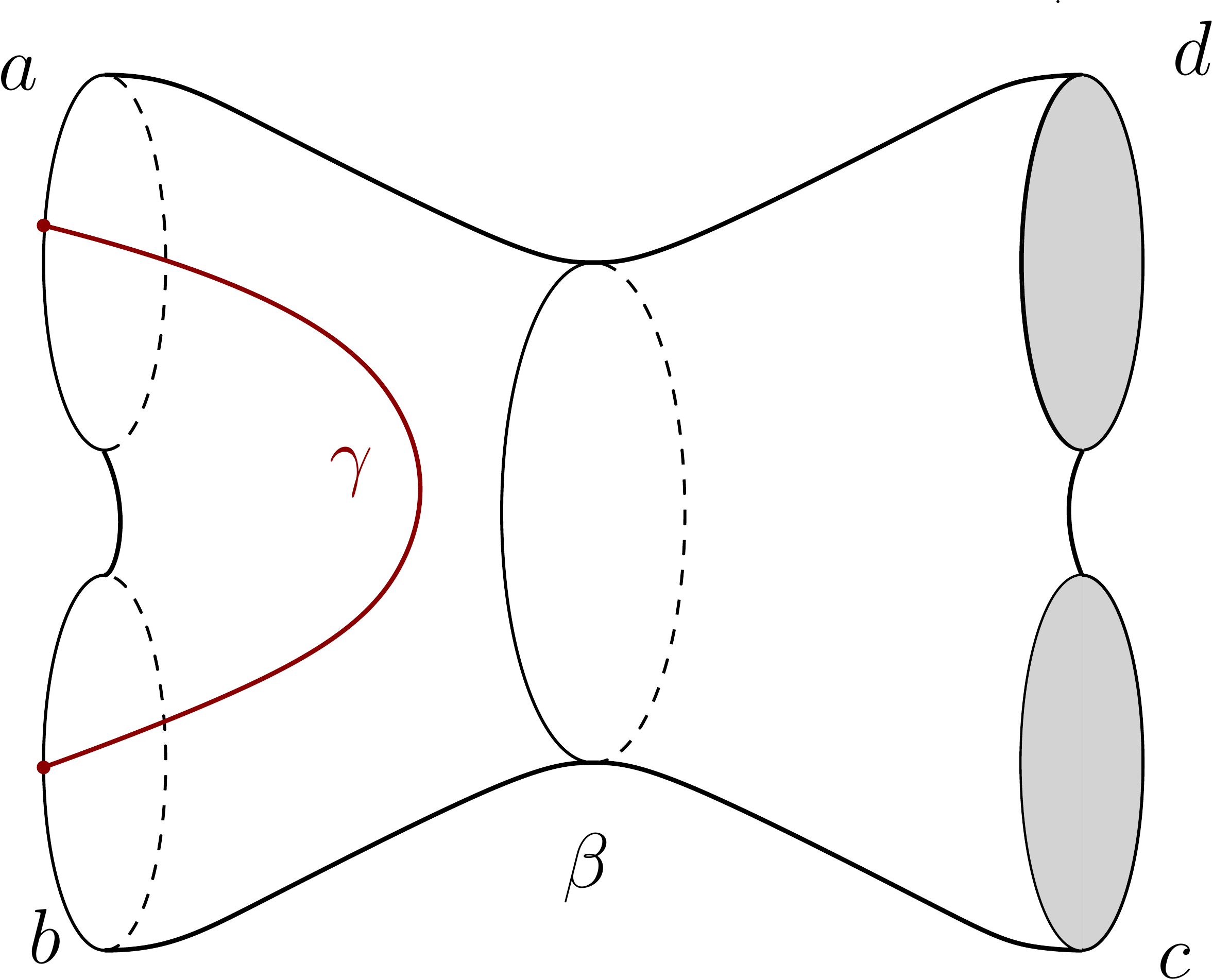}
			\caption{\footnotesize {Case 1: $ q  \in \partial_2 $ and $ \upgamma $ does not cross $ \beta $. \label{fig:lemma_case_1}
			}}
		\end{subfigure}\hspace{0.2cm}
		\begin{subfigure}{7cm}
			\centering\includegraphics[width=4cm]{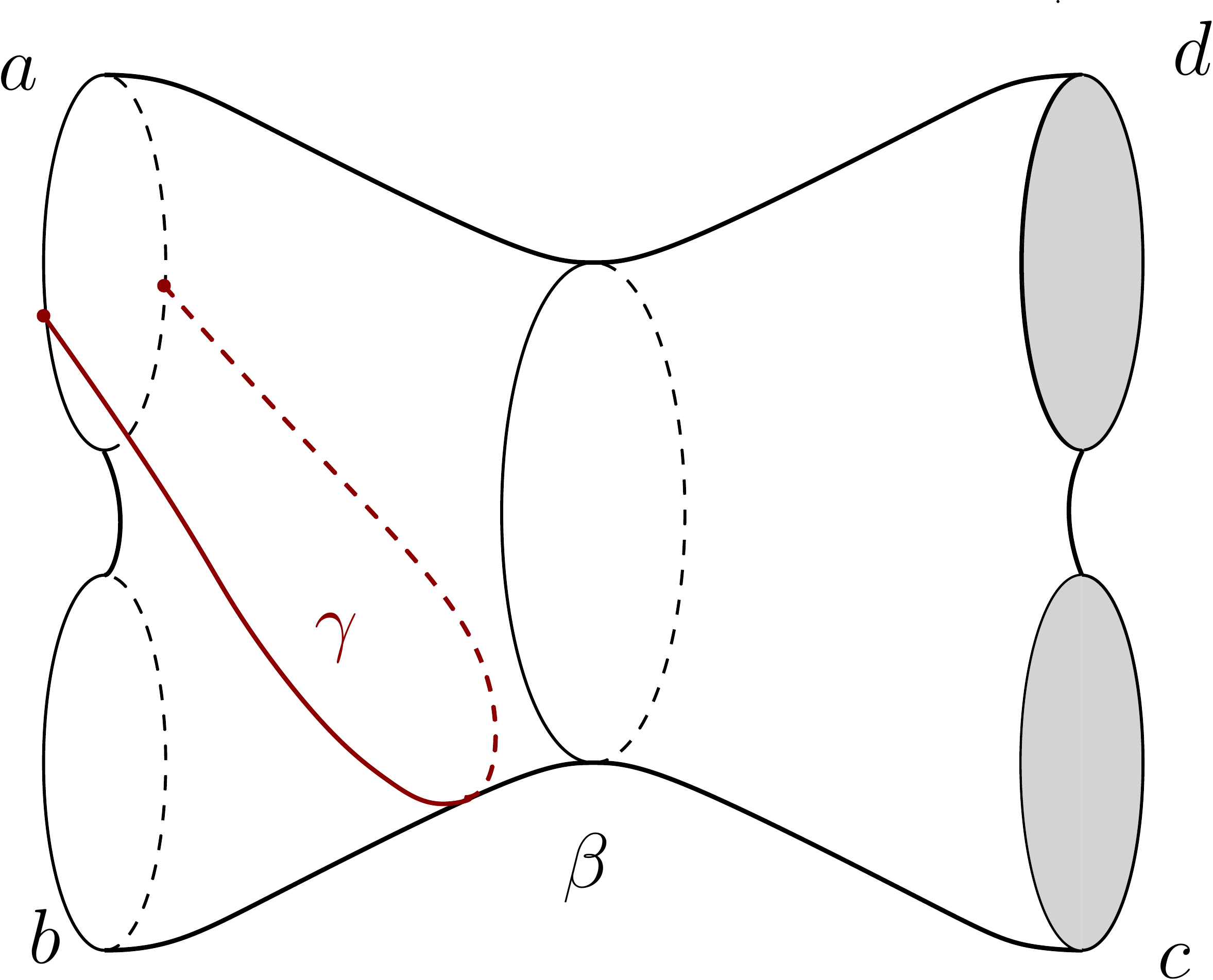}
			\caption{\footnotesize {Case 2: $ q \in \partial_1 $  and $ \upgamma $ does not cross $ \beta $. \label{fig:lemma_case_2}
			}}
		\end{subfigure}	
		\caption{\small{The possibilities for the geodesic path $ \upgamma $}}
		\label{figure:path_lemma_12}
		\vspace{-.8cm}
	\end{figure}
	
	\subsubsection*{Case 1.} Suppose that $ q \in \partial_2 $ and $ \upgamma $ does not cross $ \beta $. The geodesic path connecting $ a $ to $ b $ may intersect with either one or two hexagons on the \autoref{fig:lemma_octagon_case_1}.
	
	We first consider the case where $ \upgamma $ does not intersect the geodesic seam $ \delta_{ab} $ (\autoref{fig:lemma_octagon_case_1}, the polygon on the right) and the convex quadrilateral with two right angles. Using the expression (\ref{bound_convex_quadrilateral}) for convex quadrilaterals, we have
	\begin{align*}
		\cosh \upgamma 
		= 
		\cosh a_p \cosh b_q \cosh \delta_{ab} - \sinh a_p \sinh b_q
	\end{align*}
	Then,
	\begin{align*}
		\cosh \upgamma & \leq  \cosh a \cosh b \cosh \delta_{ab}\\
		& \leq  e^{3m}
	\end{align*}
	The formula (\ref{bound_convex_quadrilateral_2}) implies that
	\begin{align*}
		\sinh \upgamma \sin \theta = \cosh b \cosh \delta_{ab}
	\end{align*}
	from which we can get
	\begin{align*}
		\sinh \upgamma  & \geq \cosh b \cosh \delta_{ab} \\
		& \geq e^{2m} 
	\end{align*}
	For the other situation where $ \upgamma $ intersects the geodesic seam $ \delta_{ab} $, there appears a non-convex quadrilateral with two right angles (\autoref{fig:lemma_octagon_case_1}, the polygon on the left). Here we use (\ref{bound_nonconvex_quadrilateral})
	\begin{align*}
		\cosh \upgamma 
		= 
		\cosh a_p \cosh b_q \cosh \delta_{ab} + \sinh a_p \sinh b_q
	\end{align*}
	It follows that
	\begin{align*}
		\Bigg(1+\dfrac{1}{2m^{2}}\Bigg)^3 + \frac{1}{m^2}
		\leq
		\cosh \upgamma \leq e^{3m} + e^{2m}
	\end{align*}
	which shows $ \upgamma $ is bounded by a number which depends on $ m $.
	
	\begin{figure}[ht!]
		\centering
		\includegraphics[width=0.3\linewidth]{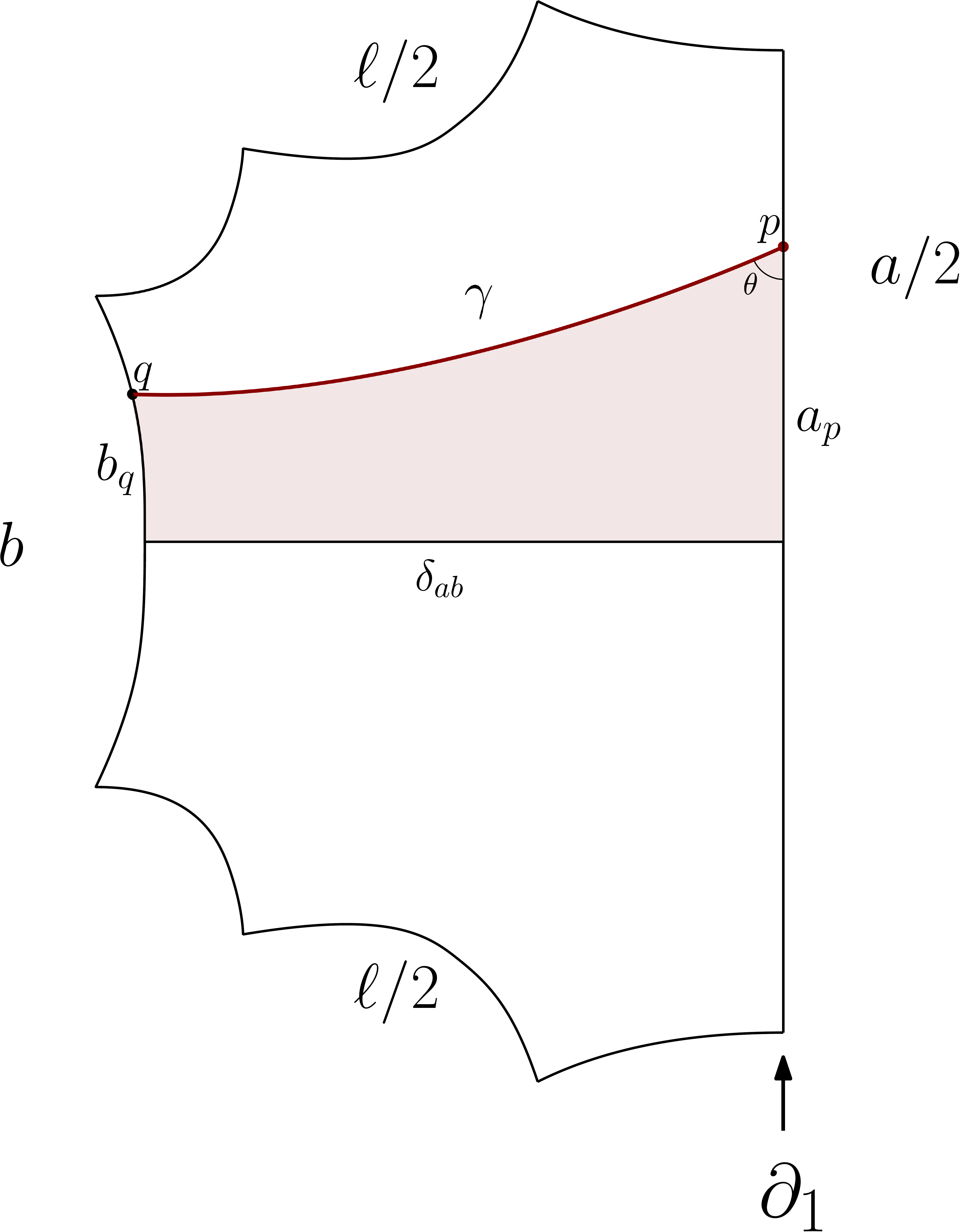}
		\hspace{1cm}
		\includegraphics[width=0.3\linewidth]{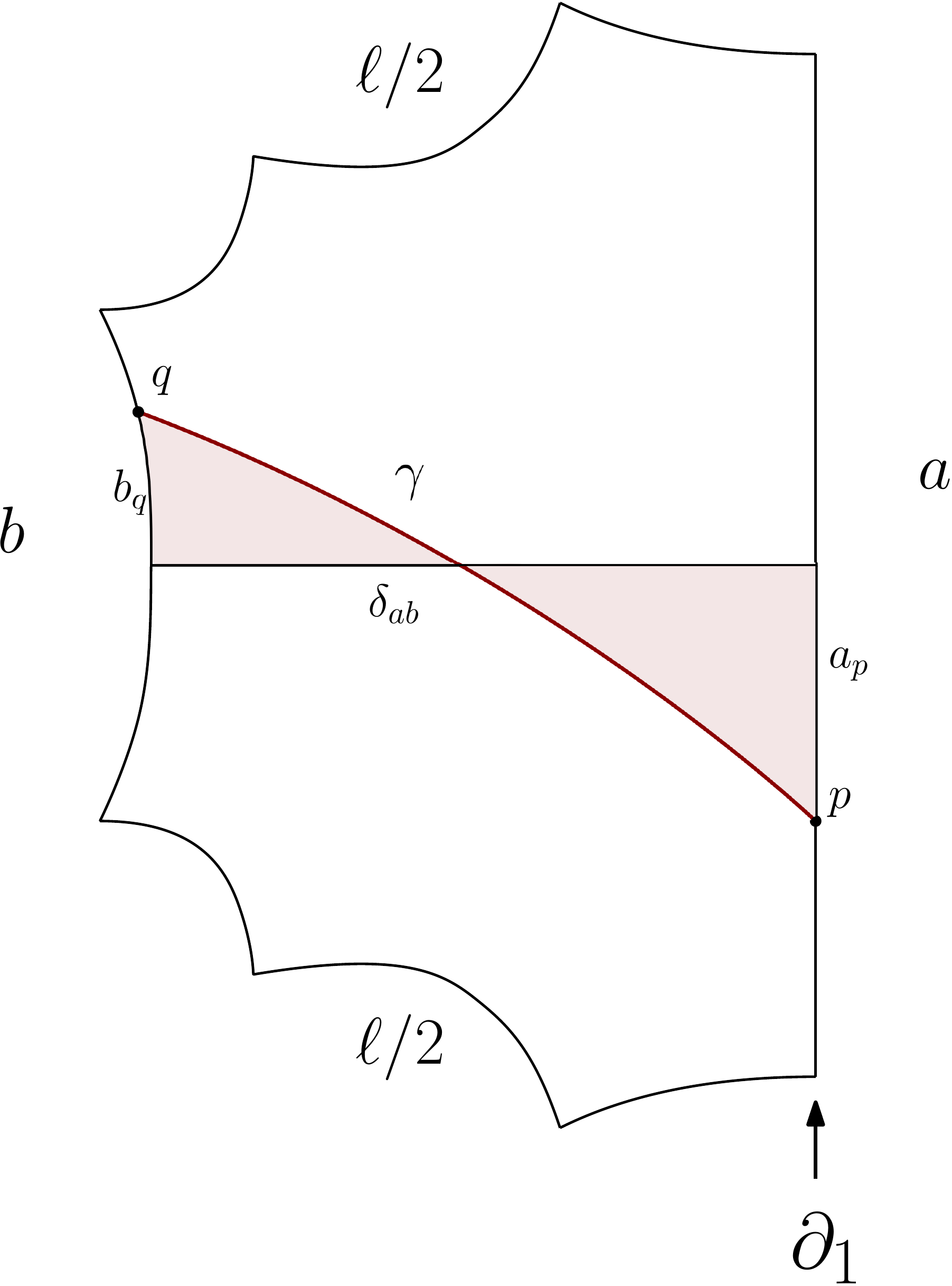}
		\caption{\small {The octagon that we use to find a bound for $ \upgamma $. In the figure on the left $ \upgamma $ does not intersect the geodesic seam $ \delta_{aa} $, whereas the figure on the right depicts the other possibility. }\label{fig:lemma_octagon_case_1}}
	\end{figure} 
	
	\subsubsection*{Case 2.} Suppose that $ q \in \partial_1 $ (that is, $ \upgamma $ joins $ \partial_1 $ to $ \partial_1 $) and $ \upgamma $ does not intersect with $ \beta $. Depending on whether the geodesic path $ \upgamma $ intersects the geodesic seam $ \delta_{ab}$ we have two situation shown in the figure In this case we focus on the \autoref{fig:lemma_octagon_case_2}.
	
	\begin{figure}[ht!]
		\centering
		\includegraphics[width=0.45\linewidth]{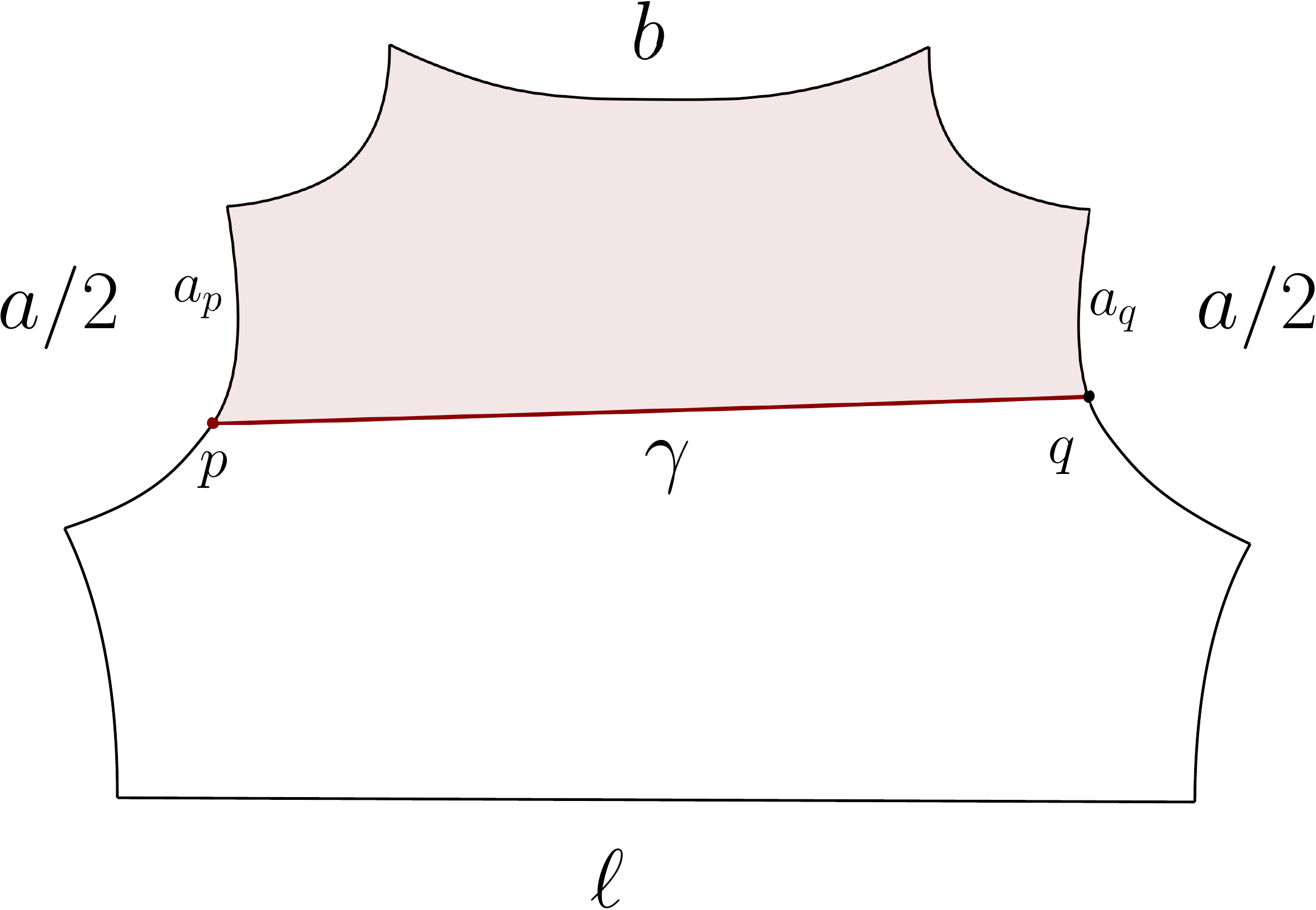}
		\hspace{1cm}
		\includegraphics[width=0.45\linewidth]{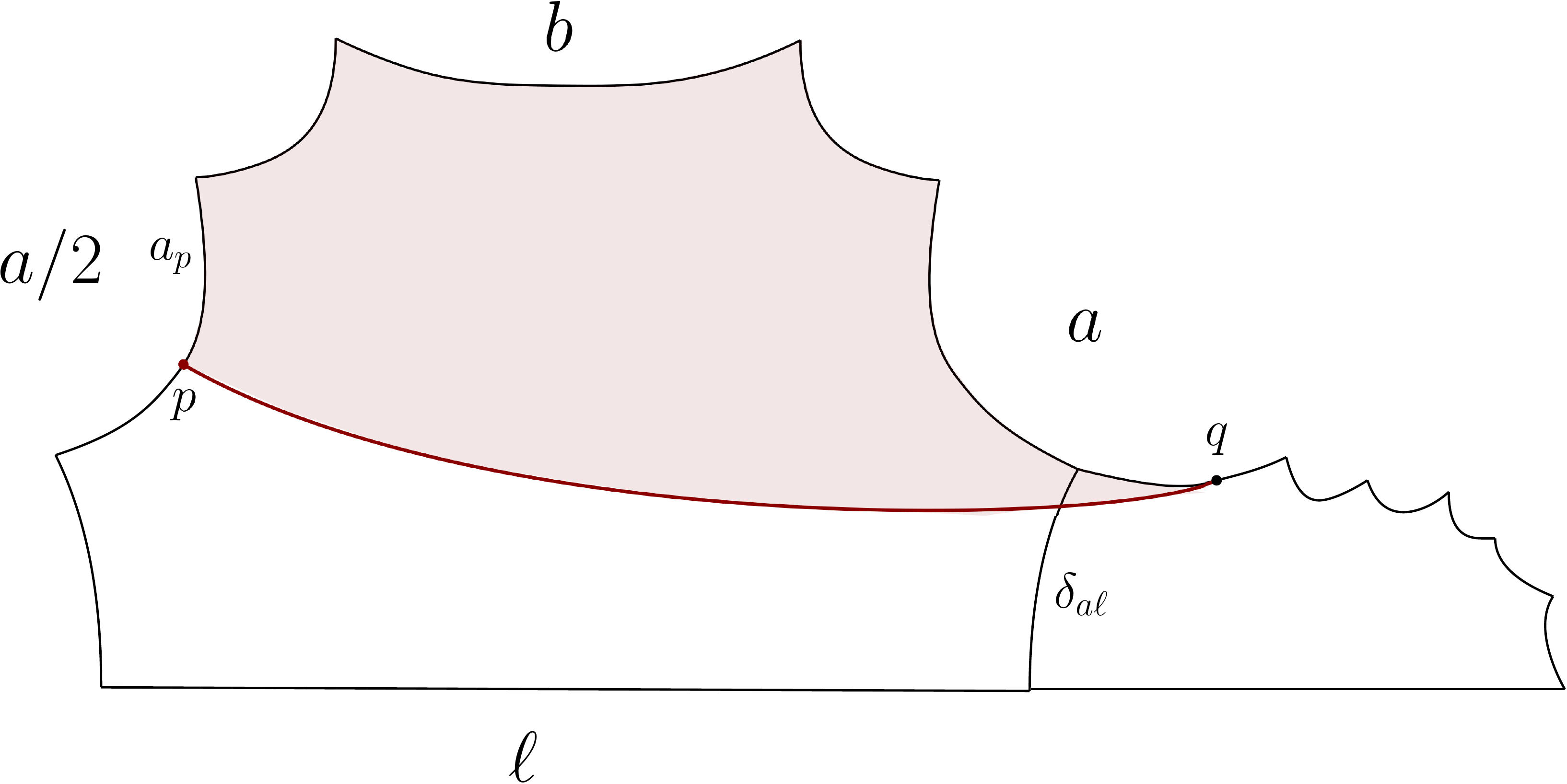}
		\caption{\small {The octagon that we use for the Case 2.}\label{fig:lemma_octagon_case_2}}
	\end{figure} 
	
	We shall first find an upper bound for $ \upgamma $. Notice that in both cases we have an hexagon with $ 4 $ successive right angles. We draw the perpendiculars $ d_p $ and $ d_q $ to the side $ b $ from the point $ p $  and from the point $ q $  respectively. There are $ 3 $ polygons which we denote by $ P_1 $, $ P_2 $ and $ P_3 $, as in the \autoref{fig:lemma_octagon_case_2_polygon}. In this configuration $ b_p $ is the length of the geodesic arc between the geodesic seam $ \delta_{ab} $ and the endpoint of the perpendicular $ d_p $, and $ b_q $ is defined similarly.
	\begin{figure}[ht!]
		\centering
		\includegraphics[width=0.5\linewidth]{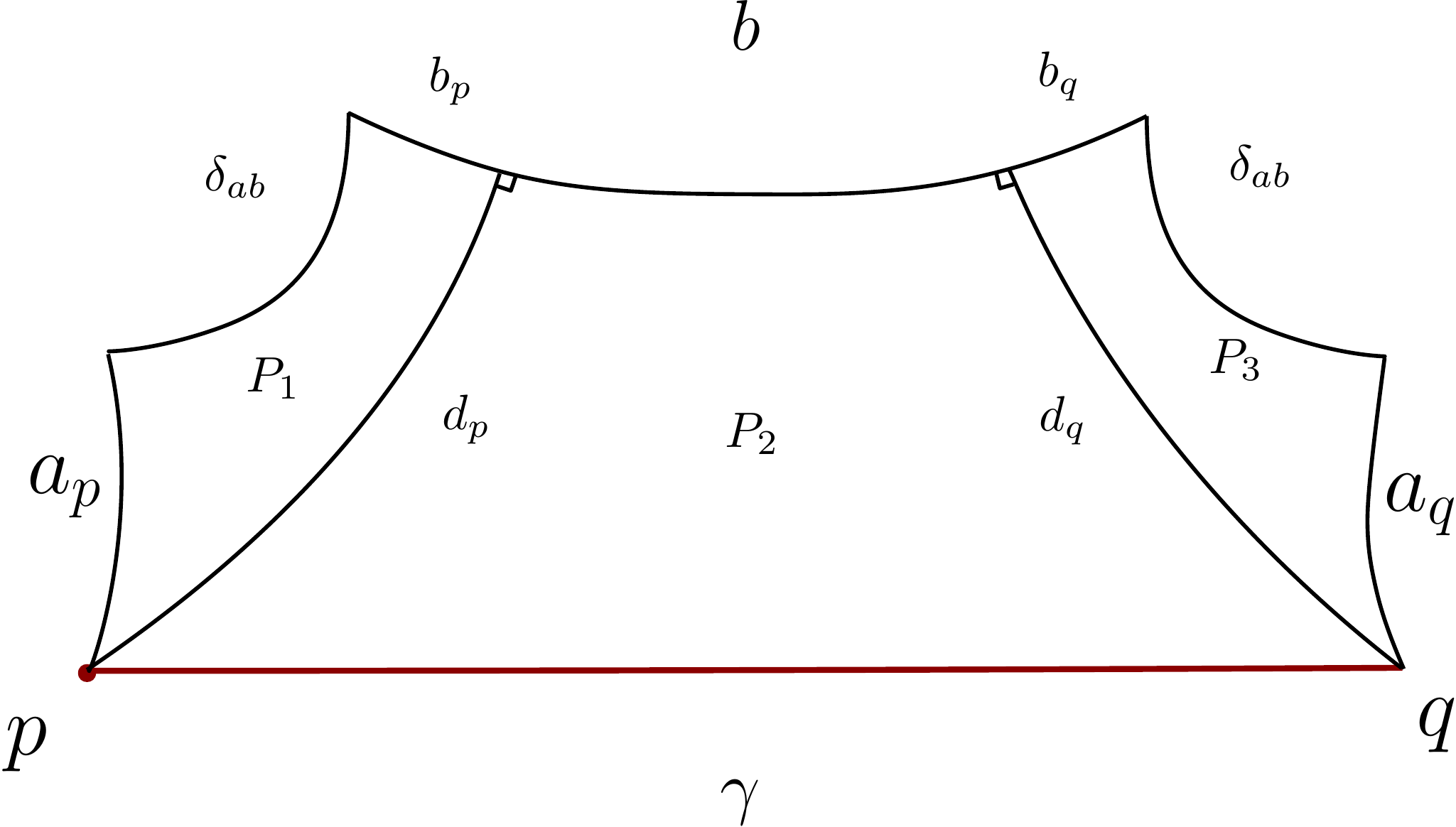}
		\caption{\small {The hexagon used to find an upper bound for $ \upgamma $}\label{fig:lemma_octagon_case_2_polygon}}
	\end{figure} 
	
	From the polygon $ P_1 $ and by the formula  (\ref{bound_convex_quadrilateral_2}) for quadrilaterals we get  
	\begin{align*}
		\sinh d_p = \sinh \delta_{ab} \cosh b_p
	\end{align*}
	then,
	\begin{align*}
		\sinh d_p \leq e^{2m}
	\end{align*}
	Similarly we can get from the polygon $ P_3 $ that 
	\begin{align*}
		\sinh d_q \leq e^{2m}
	\end{align*}
	Now from $ P_2 $ we have 
	\begin{align*}
		\cosh \upgamma 
		= 
		\cosh d_p \cosh d_q \cosh \delta_{ab} - \sinh d_p \sinh d_q
	\end{align*}
	Then,
	\begin{align*}
		\cosh \upgamma & \leq  \cosh d_p \cosh d_q \cosh (b- (b_p+b_q)) \\
		& \leq  e^{3m}
	\end{align*}
	which gives an upper bound for $ \upgamma $. 
	
	We need a lower bound as well. To do so, it suffices to find a lower bound for the shortest path which intersects the geodesic seam $ \delta_{b\ell}$ and joins $ a $ to itself. This path is denoted by $ \delta_{aa}$, see \autoref{fig:lemma_octagon_case_2_polygon_2}.
	
	\begin{figure}[ht!]
		\centering
		\includegraphics[width=0.5\linewidth]{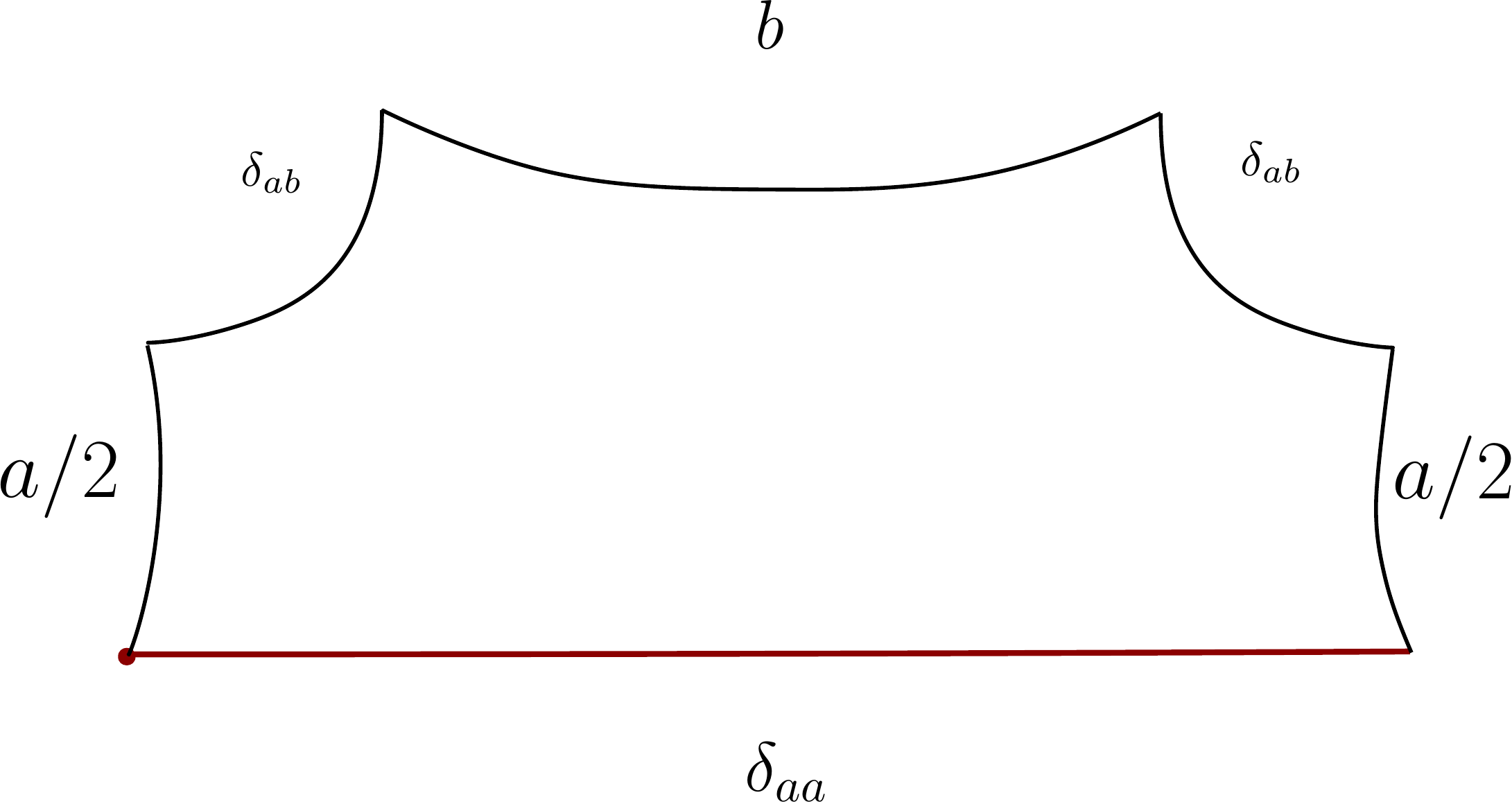}
		\caption{\small {The hexagon used to find a lower bound for $ \upgamma $}\label{fig:lemma_octagon_case_2_polygon_2}}
		\vspace{-.2cm}
	\end{figure} 
	We use the formula (\ref{hexagon_1}) for right-angled hexagons 
	\begin{align*}
		\cosh \delta_{aa} & = \dfrac{\cosh (a/2) +\cosh (a/2) \cosh b}{\sinh (a/2)  \sinh b} \\
		& \geq \dfrac{1}{e^{-2m}} + (1 + e^{-a})(1 + e^{-b})\\
		& \geq 1 + e^{-m}
	\end{align*}
	Since $  \delta_{aa} $ is shorter that $ \upgamma $, this inequality gives a lower bound for $ \upgamma $, as desired. We now turn to the remaining cases, see \autoref{figure:path_lemma_345}.
	
	\subsubsection*{The remaining cases.} Suppose that $ q \in \partial_4 $. The geodesic path is connecting the point $ p \in \partial_1 $ to $ q \in \partial_4 $. Let $ \upgamma_1 $ be the path which connects $ p $ to the intersection point of $\beta $ and $ \upgamma $. Similarly $ \upgamma_1 $ the path which connects the intersection point of $\beta $ and $ \upgamma $ to $ q $. Now $\upgamma $ is decomposed it into two geodesic paths $ \upgamma_1 $ and $ \upgamma_2 $. The calculation in the first case 1 gives a bound for $ \upgamma_1 $ and as well as for $ \upgamma_2 $, hence a bound for $ \upgamma $ since the length of the geodesic path $ \upgamma $ decomposed into $ \upgamma_1 $ and $ \upgamma_2 $ is the sum of the length of these geodesic paths. Notice that if  $ q \in \partial_3 $ (which corresponds to the case 5 in the \autoref{figure:path_lemma_345}(c)) we can decompose $ \upgamma $ in the same manner. Finally for the case 4, $ q \in \partial_1 $ and $ \upgamma $  intersect with $ \beta $. Again $ \upgamma $ can decomposed into two geodesic paths. Each of these geodesic paths are bounded (from below and above) by the calculation which we made for the first two cases. 
	
	\begin{figure}[ht!]
		\centering
		\begin{subfigure}{4cm}
			\centering\includegraphics[width=4cm]{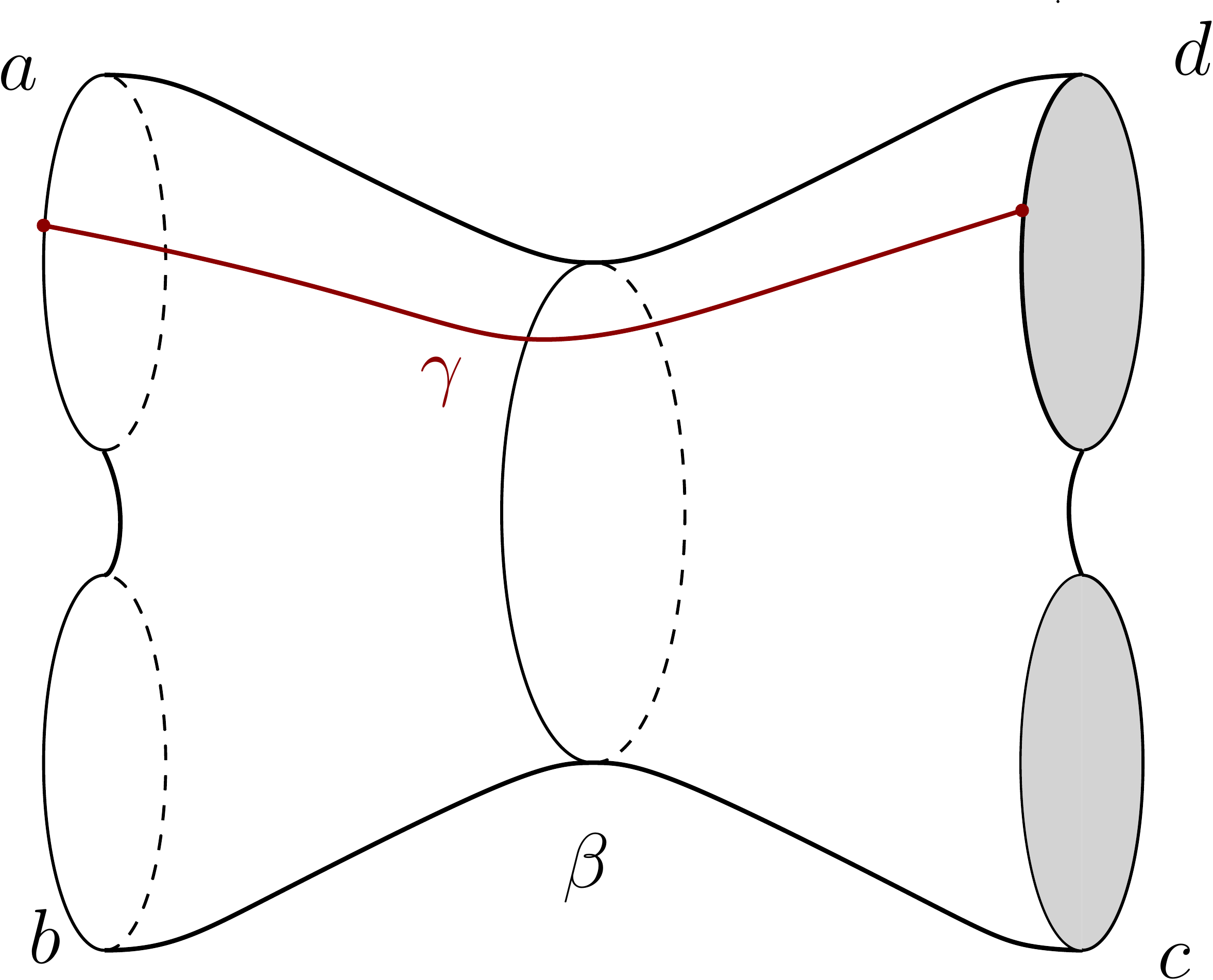}
			\caption{\footnotesize {Case 3: $ q \in \partial_4 $. \label{fig:lemma_case_3}
			}}
		\end{subfigure}\hspace{1.4cm}
		\begin{subfigure}{4cm}
			\centering\includegraphics[width=4cm]{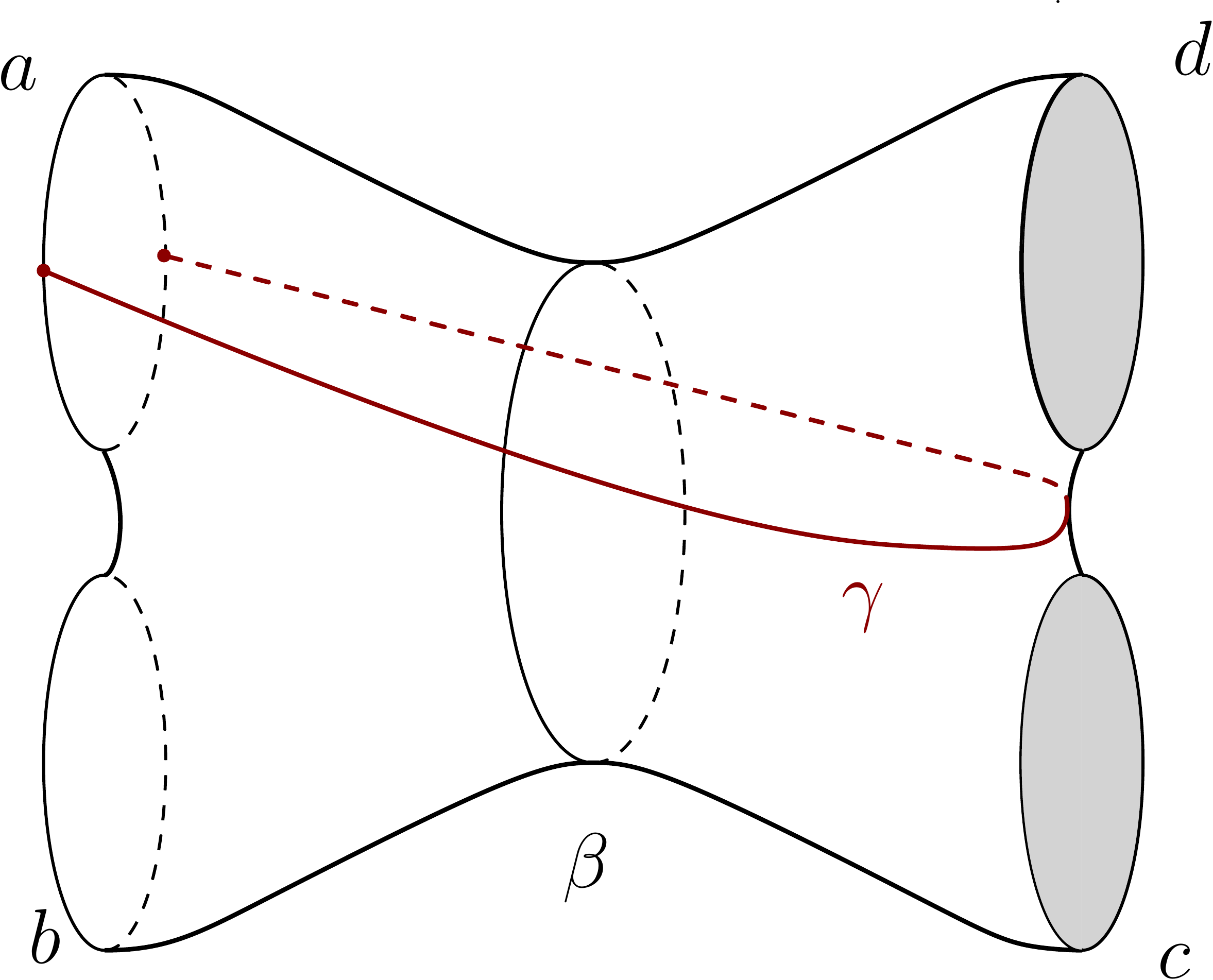}
			\caption{\footnotesize {Case 4: $ q \in \partial_1 $. \label{fig:lemma_case_4}
			}}
		\end{subfigure}\hspace{1.4cm}
		\vspace{.5cm}
		\begin{subfigure}{4cm}
			\centering\includegraphics[width=4cm]{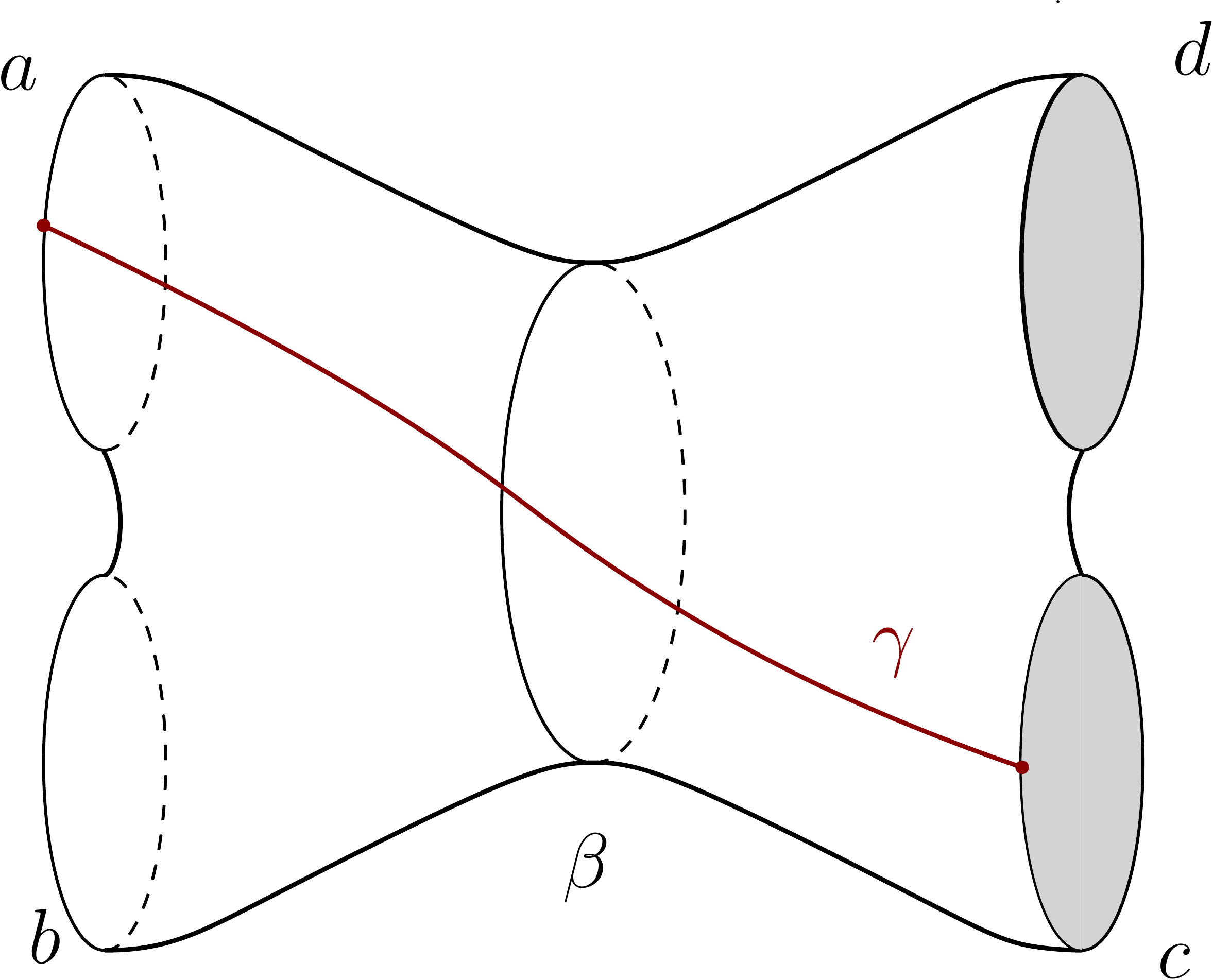}
			\caption{\footnotesize {Case 5: $ q \in \partial_3 $. \label{fig:lemma_case_5}
			}}
		\end{subfigure}
		\caption{\small{} \label{figure:path_lemma_345}}
	\end{figure}
	
	\begin{figure}[ht!]
		\centering
		\begin{subfigure}{4cm}
			\centering\includegraphics[width=4cm]{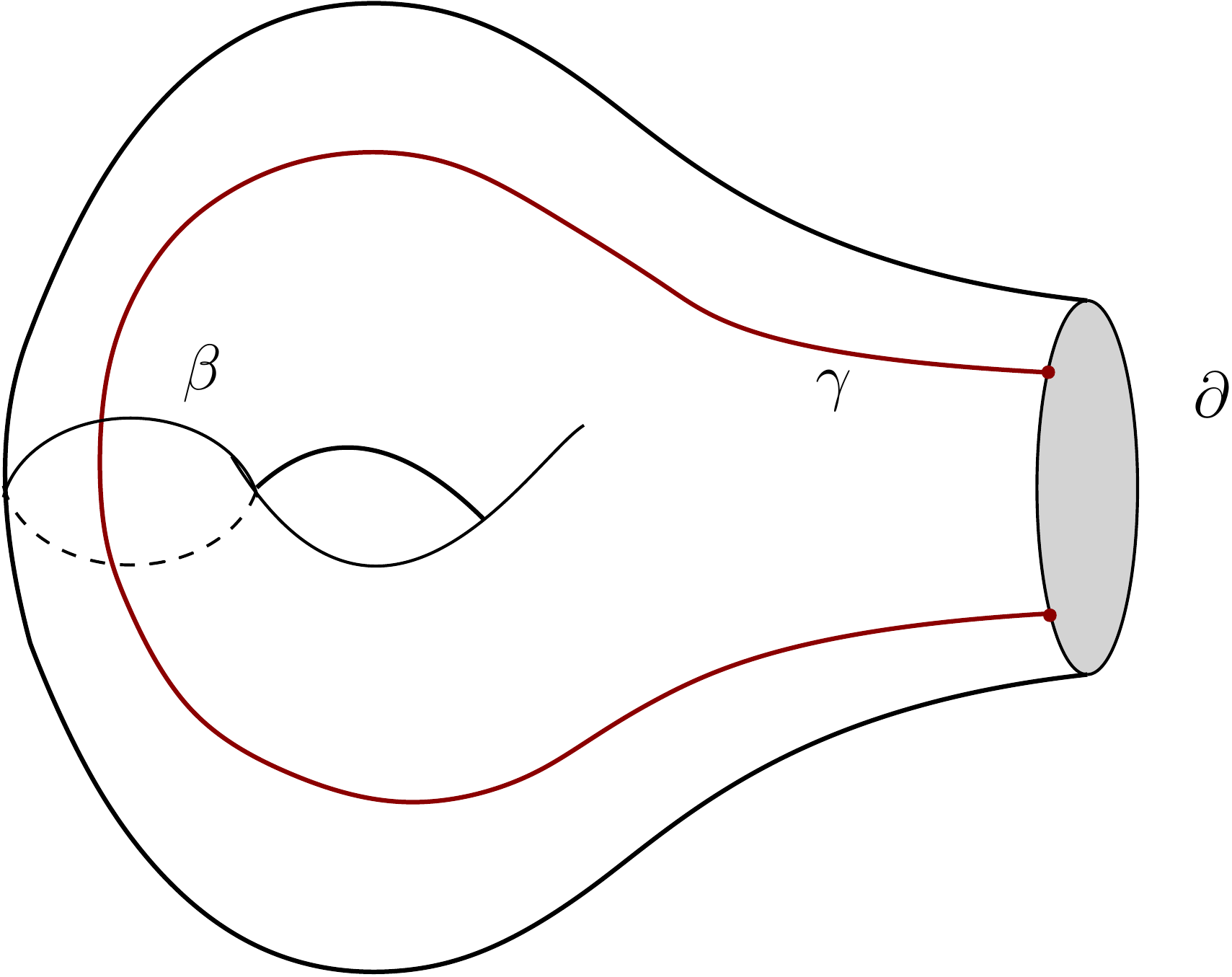}
			\caption{\footnotesize {Case 6: The geodesic path $ \upgamma $ crosses $ \beta $. \label{fig:lemma_case_3}
			}}
		\end{subfigure}\hspace{2cm}
		\begin{subfigure}{4cm}
			\centering\includegraphics[width=4cm]{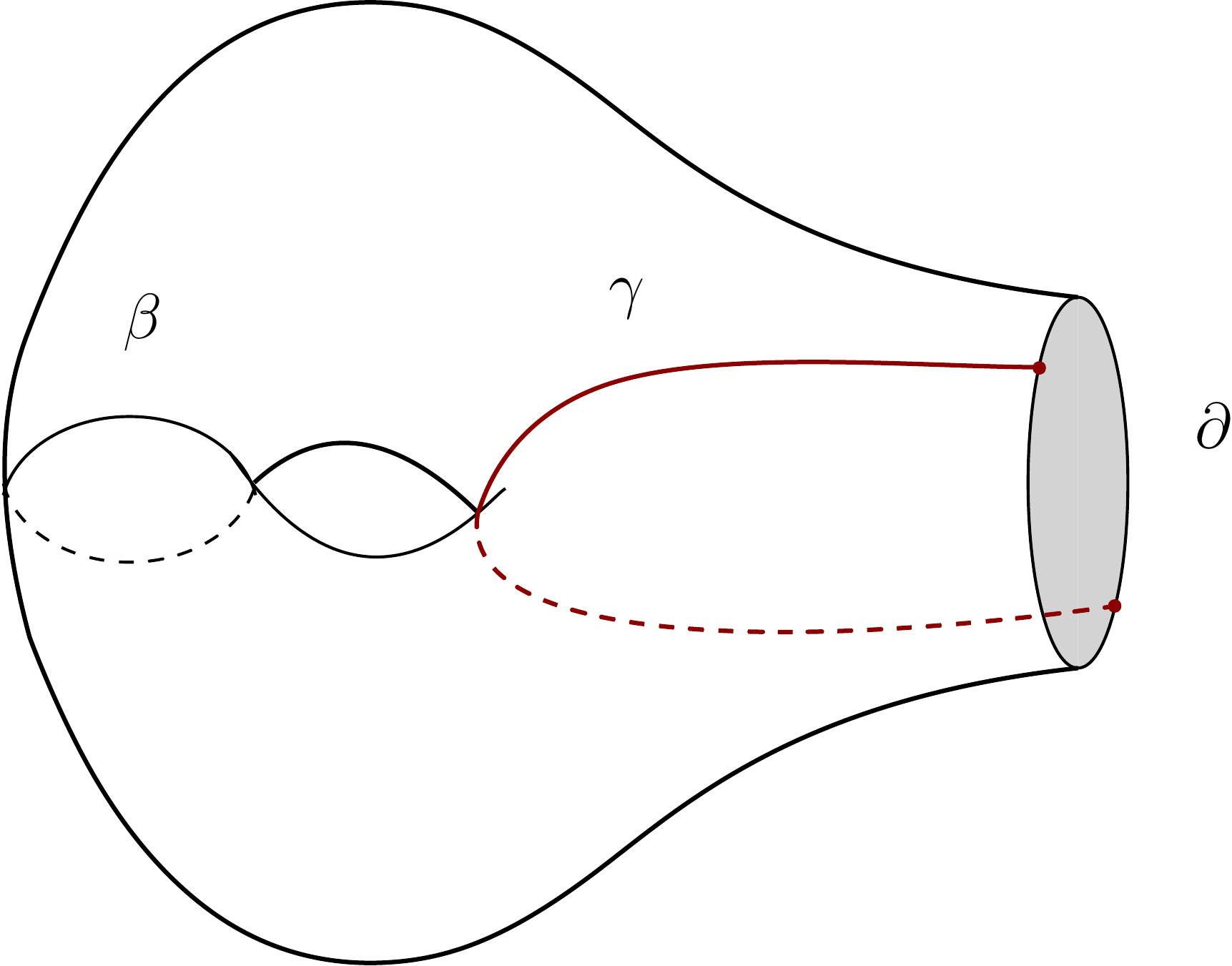}
			\caption{\footnotesize {Case 7: $ \upgamma $ does not cross $ \beta $. \label{fig:lemma_case_4}
			}}
		\end{subfigure}
		\vspace{10pt}
		\caption{\small{} \label{figure:path_lemma_67}}
	\end{figure}

	If $ \X $ is of type $ (1,1) $ there are two configurations. The endpoints $ p $ and $ q $ lie on the boundary curve $ \partial $ of $ \X $. In the case $ 6 $ we can cut the surface along $ \beta  $ so that the geodesic path $ \upgamma $ is decomposed into two pieces. Then the result follows by the case $ 1 $ \autoref{figure:path_lemma_67}(a). The last case follows by the second case \autoref{figure:path_lemma_67}(b).
\end{proof}

\section{Finitely supported hyperbolic structures}\label{Section:cs}
	
It is our goal in this section to introduce the finitely supported Teichmüller space $ \teich^{fs}\entre{H_0}  $ associated to a surface $ \Sigma $, and to show that it is a subspace of the length-spectrum Teichmüller and the quasiconformal Teichmüller space. We will also introduce a closely related Teichmüller space $ \teich^{0}_{ls} \entre{H_0}  $ which consists of asymptotically isometric hyperbolic surfaces and prove that the finitely supported Teichmüller space is dense in $ \teich^{0}_{ls} \entre{H_0} $. This section also contains the analogous results in the context of conformal structures.

Let $ \Sigma $ be a connected oriented surface. A \textit{marked hyperbolic surface} is a pair $ \entre{f,H} $ of a hyperbolic structure $ H $ on a surface $ S $ homeomorphic to $ \Sigma $ and a homeomorphism $f: \Sigma \rightarrow S $. The homeomorphism $ f: \Sigma \rightarrow S $ is called a \textit{marking} of $ \Sigma $. There is a correspondence between the collection of hyperbolic metrics on $ \Sigma $ and the collection of hyperbolic surfaces associated to $ \Sigma $ by homeomorphisms. Any hyperbolic structure on the base surface $ \Sigma $ can be regarded as a \textit{marked} hyperbolic surface with the trivial marking, namely the identity map $\id : \Sigma  \rightarrow \Sigma $. Conversely, each marked hyperbolic surface $ \entre{f,H} $ gives rise to a metric on $ \Sigma $ namely the metric induced by the pullback of  $H$ via $f$. We denote the pullback hyperbolic metric by $ f^{*} \entre{H} $. 

In this section we shall be concerned with the deformations of the base surface $ \Sigma $ performed on its finite type subsurfaces. In what follows, such a deformation will be called a finitely supported deformation. We require that the boundary components of a subsurface are simple closed geodesics in the base hyperbolic surface. Thus the fundamental group of such a subsurface is embedded into the fundamental group of $ \Sigma $. In the conformal setting, we apply the same requirement on $ \Sigma $ where it is endowed with the intrinsic metric corresponding to the base conformal structure.

\begin{definition}
	Let $ H_0 $ be a hyperbolic structure on the surface $ \Sigma $. We say that a marked hyperbolic surface $ \entre{f,H} $ is  \textbf{finitely supported with respect to} $ H_0 $ if there is a finite type subsurface $ E $ of $ \Sigma $ with $ \chi (E) < 0$ such that 
	\begin{align*}
		H_{0 \restriction_{\Sigma \setminus E}} \text{ and }  f^{*}\entre{H}_{\restriction_{\Sigma \setminus E}}
	\end{align*}
	are the same hyperbolic structure on $ \Sigma \setminus E $.
	Two finitely supported marked hyperbolic surfaces $ \entre{f,H}  $ and $ \entre{f',H'}  $ are \textbf{equivalent}  if there is an isometry $ h : H \rightarrow H' $ such that in the following diagram 
	\[ \begin{tikzcd}
		\Sigma  \arrow{r}{f} \arrow[swap]{dr}{f'} & H \arrow{d}{h} \\
		& H'
	\end{tikzcd} \]
	the composition map $ h \circ f $ is homotopic to $ f' $. We denote the equivalence class of a marked hyperbolic surface $ \entre{f,H} $ by $ \big[f,H\big] $. The set consisting of all equivalence classes of marked hyperbolic surfaces $ \big[f,H\big] $ finitely supported with respect to $ \entre{Id, H_0}$ is called \textbf{the finitely supported Teichmüller space}  and it is denoted by $\teich^{fs}\entre{H_0}$.
\end{definition}

We note that on a surface of finite topological type any two hyperbolic structures are vacuously finitely supported with respect to each other. From now on, we shall assume that $ \Sigma $ is of infinite type and it is considered to be endowed with the base hyperbolic structure $ H_0 $. 

We can topologize $\teich^{fs}\entre{H_0}$ using geodesic length functions as follows. Let $ \big[f,H\big] $ be a point in $\teich^{fs}\entre{H_0}$ and $ \alpha $ a homotopy class of essential simple closed curves in the surface $ \Sigma $. There is a unique geodesic on the hyperbolic surface $ H $ in the homotopy class of $ f(\alpha) $.  We assign to each point $ \big[f,H\big] $  the length function $ \widetilde{\ell}_{H}: \mathscr{S} \big(\Sigma\big)  \rightarrow (0,\infty)$ where $\widetilde{\ell}_{H}(\alpha )$ is defined to be the length of the unique geodesic in the homotopy class $ f(\alpha) $, namely $ \ell_{H}(f(\alpha))$. We have the mapping
\begin{align*}
	\teich^{fs}\entre{H_0}  &\longrightarrow (0,\infty)^{\mathscr{S}} \\
	\big[f,H\big]  &\longmapsto \big( \ell_{H}(f(\alpha)) \big)_{\alpha \in \mathscr{S}}
\end{align*}
The product topology on the space $ (0,\infty)^{\mathscr{S}} $ is generated by the base sets formed by taking products $  \prod U_{\alpha}$ of open subsets $ U_{\alpha} $ of $ (0,\infty) $ such that $ U_{\alpha}  = (0,\infty) $ except for finitely many indices $ \alpha $.  On the finitely supported Teichmüller space $ \teich^{fs}\entre{H_0} $ we have the induced topology, that is, the topology generated by the inverses of all these open subsets under the mapping above. 
We recall that the Teichmüller space $ \teich\entre{\Sigma_{g,b,p}} $ of a finite type hyperbolic surface endowed with the topology defined using geodesic length functions is homeomorphic to $ \R^{6g-6+3b+2p} $. In  \autoref{Sec:FN-coordinates} assuming $ H_0 $ is Nielsen-convex we show that $ \big( \teich^{fs} \entre{H_0}, d_{qc}\big)$ is homeomorphic to the space $ \bigoplus_{\N} \R $ consisting of all sequences which have only finitely many nonzero terms, endowed with the supremum metric.

Let $ \beta $ be an oriented simple closed $ H_0 $-geodesic contained in the interior of $ \Sigma $. The \textbf{Fenchel-Nielsen time-\textit{t} twist deformation} about $ \beta $ is a marked hyperbolic surface $ \big(f_t,H_t\big) $ for $ t > 0 $ obtained by cutting the surface along $ \beta $, performing a left-twist of hyperbolic length $ t $ and gluing back the resulting two boundary geodesics of the cut surface. For $ t \leq 0 $ we perform the right-twist similarly (\autoref{fig:twist_deformation}). 
\begin{figure}[ht!]
	\centering
	\hspace{-1.5cm}
	\includegraphics[width=0.6\linewidth]{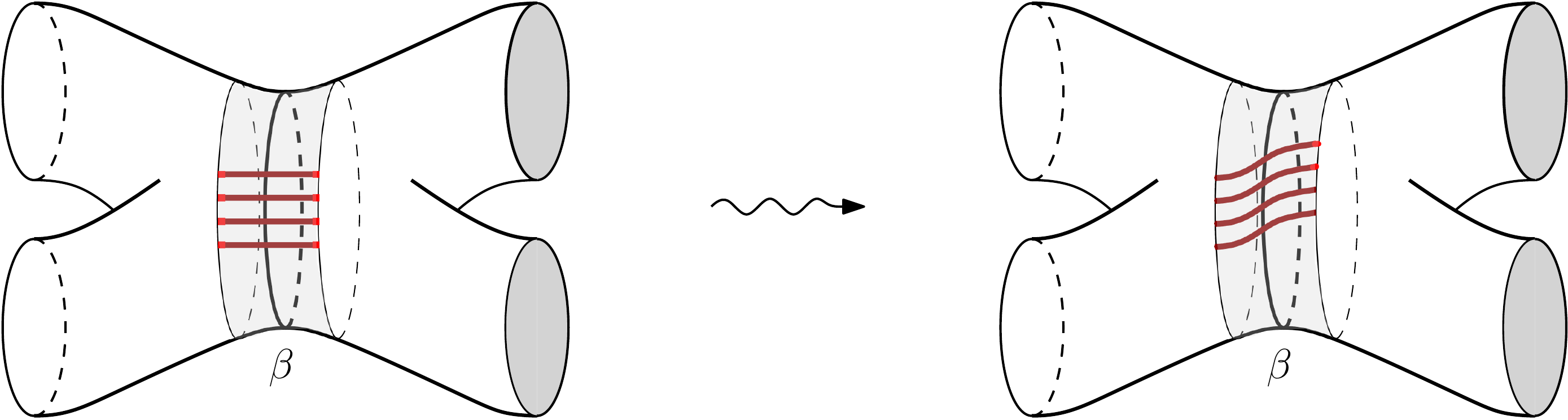}
	\caption{\small {Fenchel-Nielsen twist deformation}\label{fig:twist_deformation}}
	\vspace{-.2cm}
\end{figure} 

We define \textbf{the Fenchel-Nielsen length deformation}. We choose a subsurface $ \X $ of $ \Sigma $ of type $ (0,4) $ which contains a $ \beta $ as an interior geodesic curve. We regard the subsurface $ \X $ as a union of two pairs of pants glued along the geodesic $ \beta $ where the way of gluing is determined by the twist number $ t(\beta) $. On each of these two pairs of pants there are three geodesic seams connecting the boundary components. We label the endpoints of the geodesic seams by $ \zeta_{i,1} $ and $ \zeta_{i,2} $ for $ i = 1,2,3,4 $ on each boundary component $ \partial_i $ of $ \X $. 

Now, let $ r \in \R $ be such that $ \ell + r > 0 $ where $ \ell = \ell_{H_0}(\beta)$. We construct a surface $  \X_{\ell + r}  $ of type $ (0,4) $ by taking two hyperbolic pairs of pants $ P_{1} $ and $ P_{2} $ with boundary components $ (\partial'_1, \partial'_2, \beta^{\prime}) $ and $ (\beta^{\prime}, \partial'_3, \partial'_4) $ such that the hyperbolic lengths of the boundary components coincide, that is, $ \ell(\partial_{i}) = \ell(\partial'_{i})  $ for $ i = 1,2,3,4 $ and $ \ell(\beta^{\prime}) = \ell + r $, then gluing $ P_1 $ and $ P_2 $ along $ \beta' $ by performing a twist of amount $ t (\beta) $. In other words, we glue $ P_1 $ and $ P_2 $ in such a way that the twist number for $ \beta $ and $ \beta^{\prime} $ coincides.  For $ i = 1,2,3,4 $, let $ \zeta'_{i,1} $ and $ \zeta'_{i,2} $ be the endpoints of two geodesic seams on the boundary component of $ \partial'_{i} $.

Now we cut the surface $ \Sigma $ along the boundary of $ \X $ and glue the resulting surface $ \Sigma - \X $ and the surface $ \X_{\ell + r} $ along the boundary components in such a way that the points $ \zeta_{i,1} $ and $ \zeta_{i,2} $ coincide respectively with $ \zeta'_{i,1} $ and $ \zeta'_{i,2} $. Therefore, in this way we can deform any hyperbolic surface $ \Sigma $ by replacing its subsurface $ \X $ containing a $H_0$-geodesic with a surface $ \X_{\ell + r} $ so that the length of the corresponding curve $ \beta^{\prime} $ is $ \ell + r $, see  \autoref{fig:length_deformation}. The resulting marked hyperbolic surface $ \big( f_{\ell + r}, H_{\ell + r} \big) $ is a Fenchel-Nielsen length deformation of $ \Sigma $. It depends on the choice of a subsurface $ \X $ of type $ (0,4) $ and a simple closed geodesic $ \beta $ and a real number $ r $. We note that if $ \X $ were of type $ (1,1) $ the Fenchel-Nielsen length deformation could be defined in a similar way. 

\begin{remark}
	It is clear that the Fenchel-Nielsen time-$ t $ twist deformation and the Fenchel-Nielsen length deformation about a geodesic curve give rise to a marked hyperbolic structure on $ \Sigma $ which is finitely supported with respect to the base hyperbolic surface $ H_0 $.
\end{remark}

\begin{figure}[ht!]
	\centering
	\includegraphics[width=0.7\linewidth]{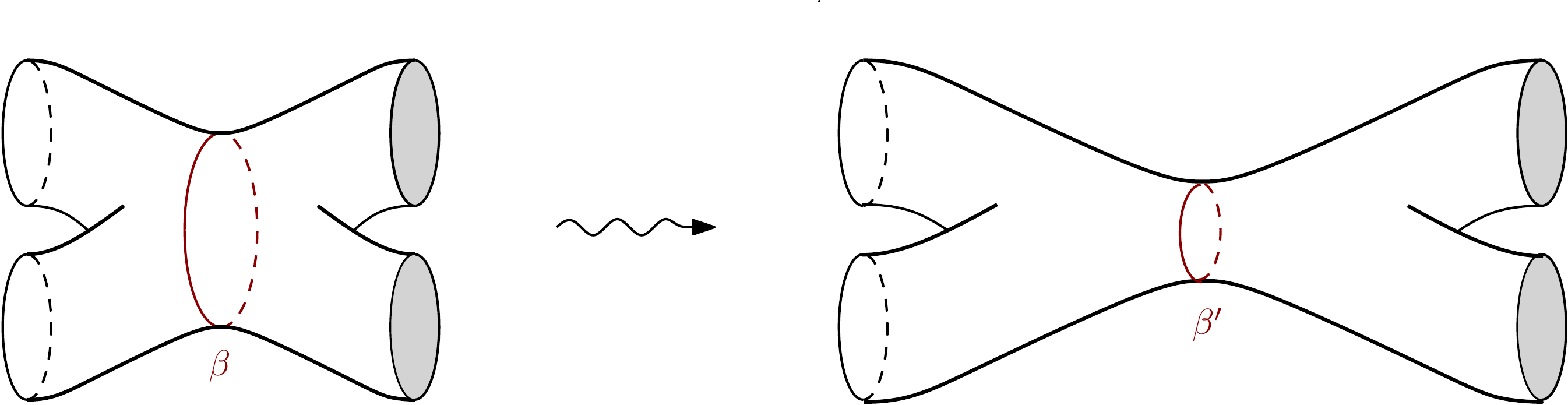}
	\caption{\small {Fenchel-Nielsen length deformation}\label{fig:length_deformation}}
\end{figure} 

Let $ \Sigma $ be a surface endowed with a base structure $ H_0 $ and let $ E $ be a finite type subsurface of $ \Sigma $ together with a pair of pants decomposition $ \P_{E} \cup \B_{E} $ where, as above, $ \P_{E} = \{ \beta_{1}, \ldots, \beta_{n} \} $ is the set of the curves lying in the interior of $ E $ and $ \B_{E} = \{ \beta_{n+1}, \ldots, \beta_{n+b}\} $ is the set of boundary curves. We can assume that all these curves are geodesics with respect to $ H_0 $ by replacing each of them with the simple closed geodesics in its homotopy class. This determines the twist numbers $ t^{*}_1, \ldots t^{*}_n \in \R $ and $ H_0 $-lengths $ \ell^{*}_{1}, \ldots, \ell^{*}_{n} \in (0,\infty) $ for each curve $ \beta_{1}, \ldots, \beta_{n} $. 

Now assume that we are given a collection of numbers $ t_1, \ldots t_n \in \R $ and $ \ell_{1}, \ldots, \ell_{n} \in (0,\infty) $. For each $ i = 1, \ldots, n $ we first perform a Fenchel-Nielsen length deformation $ \big( f_{\ell_{i}}, H_{\ell_{i}} \big) $ so that the length of geodesic in the homotopy class of $ \beta_i $ turns from $ \ell^{*}_{i} $ into $ \ell_{i} $. Then by applying Fenchel-Nielsen time-$t$ deformation $ \big( f_{t_i}, H_{t_i} \big) $ for a suitable $ t $ we can deform the surface so that the twist number of $ \beta_i $ becomes $ t_i $.

In this manner, we can repeatedly apply a sequence of these cut and paste operations to a given hyperbolic surface to get a surface $ \Sigma_{d} $ endowed with a hyperbolic structure $ H_{d} $ together with a marking $ f_d : \Sigma \rightarrow \Sigma_{d} $ where $ d = (t_1, \ldots t_n, \ell_{1}, \ldots, \ell_{n}) $ where $ f_d = f_{t_{n}} \circ \ldots \circ f_{t_{1}} \circ  f_{\ell_{n}} \circ \ldots \circ f_{\ell_{1}} $. 

\begin{notation}
	We denote such a marked hyperbolic surface by $ \big( f_d,H_d \big) $. 
\end{notation}


\begin{lemma}\label{lemma:lemma_FN}
	Let $ \Sigma $ be a surface endowed with a base hyperbolic structure $ H_0 $. Let $ \entre{f,H} $ be a finitely supported marked hyperbolic structure with respect to $ H_0 $. There exist $ t_1, \ldots t_n \in \R $  and $ \ell_{1}, \ldots, \ell_{n} \in (0,\infty) $ and an isometry $ h $ such that for the marked hyperbolic surface $ \entre{f_d,H_d} $ defined as above on $ \Sigma $ in the following diagram
	\[ \begin{tikzcd}
		\Sigma  \arrow{r}{f} \arrow[swap]{dr}{f_d} & H \arrow{d}{h} \\
		& \Sigma_d
	\end{tikzcd} \]
	the composition map $ h \circ f $ is homotopic to $ f_d $.
\end{lemma}

\begin{proof}
	Suppose that $ \entre{f,H} $ is a marked hyperbolic structure on $ \Sigma $ supported on $ E $. We consider a pair of pants decomposition of $ E $ and straighten the curves with respect to the pullback metric $ f^{*}(H) $ on $ \Sigma $. For each $ \beta_i $ which lies in the interior of $ E $ in the resulting geometric pair of pants decomposition, we determine the twist number $ t_i $ and $f^{*}(H)$-length $ \ell_{i} $. Now for $ d = (t_1, \ldots t_n, \ell_{1}, \ldots, \ell_{n}) $ the marked hyperbolic surface $ \big( f_d,H_d \big) $ is in the equivalence class of $ \entre{f,H}$ by construction.
\end{proof}

\subsection{Length-spectrum metric on the finitely supported Teichmüller space}\noindent
We will recall the length-spectrum metric and some of its properties. Our aim is to show that the finitely supported Teichmüller space $ \teich^{fs} \entre{H_0} $ is contained in the length-spectrum Teichmüller space $ \teich_{ls}\entre{H_0} $. This gives rise to a metric on $ \teich^{fs} \entre{H_0} $ namely the metric induced by the length-spectrum metric $ d_{ls} $ of $ \teich_{ls}\entre{H_0} $. If $ f: S_1 \rightarrow S_2 $ is a homeomorphism where $ S_1 $ and $ S_2 $ are surfaces endowed with hyperbolic structures $ H_1 $ and $ H_2 $ respectively, the quantity
\begin{align*}
	\L(f) \text{ }= 
	\sup_{\alpha \in \mathscr{S}(S_1)} 
	\Bigg\{ 
	\dfrac{\ell_{H_2}(f(\alpha))}{\ell_{H_1}(\alpha)}
	\text{ , }
	\dfrac{\ell_{H_1}(\alpha)}{\ell_{H_2}(f(\alpha))}
	\Bigg\} 
\end{align*}
is called \textit{the length spectrum constant of $f$}. It obviously depends only on the homotopy class of $ f $. In the case that $ \L(f) < \infty$ we say that $f$ is \textit{length-spectrum bounded}. In general, such a homeomorphism between surfaces of infinite topological type need not to be length-spectrum bounded. Although in the literature the quantity $ \L(f) $ is defined for connected surfaces, here we also allow $ S_1 $ to be non-connected as well, which we need in the preceding subsection, for instance in the \cref{definition:asymptotic_definition}.

Let $ \Sigma $ be an infinite type surface endowed with a base hyperbolic structure $ H_{0} $ as before. We consider the collection of marked hyperbolic structures $ \entre{f,H} $ on $ \Sigma $ where the marking $ f : H_0  \rightarrow H $ has finite length-spectrum constant. Two length-spectrum bounded marked hyperbolic structures $ \entre{f,H} $ and $ \big(f',H'\big) $ on $ \Sigma $ are considered to be \textit{equivalent} if there is an isometry $ h : H \rightarrow H' $ and a homeomorphism $ \varphi : \Sigma \rightarrow
\Sigma $ homotopic to identity such that $ h \circ f =  f' \circ \varphi $, i.e. the following diagram commutes
\begin{equation*}
	\begin{tikzcd}
		\Sigma \arrow{r}{f} \arrow[swap]{d}{\varphi} & H \arrow{d}{h} \\%
		\Sigma \arrow{r}{f'}& H'
	\end{tikzcd}
\end{equation*}

This is an equivalence relation on the set of all hyperbolic structures $ \entre{f,H} $ where the marking $ f:H_{0} \rightarrow H $ is length-spectrum bounded. 
\begin{definition}
	The space consisting of equivalence classes $ \big[f,H\big] $ of all marked hyperbolic structures which are length-spectrum bounded with respect to the base hyperbolic structure $ [ Id,H_0 ] $ is called $ \textbf{the length-spectrum Teichmüller space} $. It is denoted by $  \teich_{ls}\entre{H_0} $.
\end{definition}
This space depends strongly on the base hyperbolic structure $ H_0 $. For different base points on the same surface $ \Sigma $, the corresponding length-spectrum Teichmüller spaces are either disjoint or coincide as sets. On $\teich_{ls} \entre{H_0}$ we have \textit{the length-spectrum distance} $d_{ls}$ defined by 
\begin{align*}
	d_{ls}\bigg(\big[f_{1}, H_{1}\big], \big[f_{2}, H_{2}\big] \bigg)
	& =
	\dfrac{1}{2} \log  \L(f_{2} \circ f_{1}^{-1})
	\\
	& =
	\dfrac{1}{2} \log
	\sup_{\alpha \in \mathscr{S}(S_1)} 
	\Bigg\{ 
	\dfrac{\ell_{H_2}(f_2 \circ f_1^{-1}(\alpha))}{\ell_{H_1}(\alpha)} 
	\text{ , }
	\dfrac{\ell_{H_1}(\alpha)}{\ell_{H_2}(f_2 \circ f_1^{-1}(\alpha))}
	\Bigg\} 
\end{align*}

This defines a metric on the length-spectrum Teichmüller space $ \teich_{ls}\entre{H_0} $. We refer the reader to the work of Papadopoulos and Liu \cite{MR2792982} for the elementary properties of the length-spectrum metric. Now we will show that hyperbolic structures supported on a finite type subsurface of $ \Sigma $ are length-spectrum bounded with respect to the base hyperbolic structure.

\begin{theorem}\label{theorem:finitely_supported_is_ls}
	Every finitely supported marked hyperbolic structure on $ \Sigma $ is length-spectrum bounded with respect to the base hyperbolic structure $ H_0 $. There is a natural continuous inclusion mapping of the space $ \teich^{fs}\entre{H_0} $ of finitely supported marked hyperbolic surfaces  to the length-spectrum Teichmüller space $ \teich_{ls}\entre{H_0} $.
\end{theorem}

\begin{proof}
	Let $ f: H_0 \rightarrow H $ be a homeomorphism such that $ H_{0} =  f^{*}(H) $ on $ \Sigma \setminus E $ where $ E $ is a finite type subsurface of $ \Sigma $. We fix a pair of pants decomposition of $ E $ and denote the decomposing curves lying in the interior of $ E $ by $ \P_{E} = \{ \beta_{1}, \ldots, \beta_{n}\} $. By \cref{lemma:lemma_FN} there exist $ t_1, \ldots t_n \in \R $ and $ \ell_{1}, \ldots, \ell_{n} \in (0,\infty) $ such that the marked hyperbolic structure $ \big( f_{t_1, \ldots t_n, \ell_{1}, \ldots, \ell_{n}} , H_{t_1, \ldots t_n, \ell_{1}, \ldots, \ell_{n}} \big) $  defined as above on $ \Sigma $ is in the equivalence class of $ \entre{f,H} $. It suffices to show that two types of deformations defined in the previous section are length-spectrum bounded.
	
	\textbf{Step 1.} We first assume that $ f $ is homotopic to a time-$t$ Fenchel-Nielsen twist $ f_{t_i} $ around a curve $ \beta = \beta_i \in \P_{E} $. Let $ \alpha $ be an arbitrary simple closed curve on $ \Sigma $. Without loss of generality, we may assume that $ i (\alpha, \beta ) \neq 0 $ since the hyperbolic length of the homotopy class of a simple closed curve $ \alpha $ which does not intersect any $ \beta $ in $ E $ is invariant under a Fenchel-Nielsen twist. By the collar lemma there exists a positive number $ \omega(\beta) $ such that the $ \omega(\beta) $-tubular neighborhood of $ \beta $ is embedded in $ \Sigma $, and the number $ \omega(\beta) $ depends only on the $H$-hyperbolic length of the curve $ \beta $. We have 
	\begin{align*}
		\ell_{H}(f(\alpha)) \leq \ell_{H_0}(\alpha) + i(\alpha,\beta).t
	\end{align*}
	and by dividing both sides of this inequality by $ \ell_{H_0}(\alpha) $, we get
	\begin{align}
		\dfrac{\ell_{H}(f(\alpha))}{\ell_{H_0}(\alpha)} \leq 1 +\dfrac{ i(\alpha,\beta).t}{\ell_{H_0}(\alpha)} \leq 1 +\dfrac{t}{\omega(\beta)}
	\end{align}
	where the inequality on the right hand side follows from $ \ell_{H_0}(\alpha)  \geq i(\alpha,\beta)\omega(\beta) $ which holds because $ \alpha $ traverses the $ \omega(\beta) $-tubular neighborhood of $ \beta $ exactly $ i(\alpha, \beta )$ times. By exchanging the roles of $ \ell_{H_0}(\alpha) $ and $ \ell_{H}(f(\alpha)) $ we similarly have
	\begin{align*}
		\dfrac{\ell_{H_0}(\alpha)}{\ell_{H}(f(\alpha))} \leq 1 +\dfrac{t}{\omega(\beta)}
	\end{align*}
	and so,
	\begin{align*}
		\log\dfrac{\ell_{H_0}(\alpha)}{\ell_{H}(f(\alpha))} \leq \log \Bigg( 1 +\dfrac{t}{\omega(\beta)} \Bigg) \leq   \dfrac{t}{\omega(\beta)} 
	\end{align*}
	since $\log(1+x) \leq x \text{ for } x>0 $. Together with (1) this gives 
	\begin{align*}
		\sup_{\alpha \in \mathscr{S}(\Sigma)} \bigg\{\log  \dfrac{\ell_{H}(f(\alpha))}{\ell_{H_0}(\alpha)} \text{ , } \log \dfrac{\ell_{H_0}(\alpha)}{\ell_{H}(f(\alpha))} \bigg\} < \infty 
	\end{align*}
	
	\textbf{Step 2.} Now, assume that $ f $ is homotopic to a Fenchel-Nielsen length deformation $ (f_{\ell}, H_{\ell})$ which is performed about a curve $ \beta $ in $ \P_{E} $ lying in a subsurface $ \X $ of $ E $. Let $ \alpha $ be simple closed geodesic on $ \Sigma $. If $ \alpha $ does not intersect $ \X $ then $ \ell_{H}(f(\alpha)) = \ell_{H_{0}}(\alpha) $. Suppose $ \alpha \cap \X $ is nonempty and denote its connected components by $ \upgamma^{*}_{1}, \ldots, \upgamma^{*}_{k} $. These are geodesic paths whose endpoints $ \zeta_{i,1} $ and $ \zeta_{i,2} $ are on the boundary curves of $ \X $. 
	
	For each $ i= 1, \ldots, k $ let $ \upgamma_{i,H_{0}} $ and $ \upgamma_{i,H} $ denote the $H_{0}$-geodesic arc (respectively $H$-geodesic arc) homotopic to $ \upgamma^{*}_{i}$ relatively to the endpoints.
	We have  
	\begin{align*}
		\ell_{H}(f(\alpha)) \leq \ell_{H_0}(\alpha) + \sum_{i=1}^{n} \ell (\upgamma_{i,H})
	\end{align*}
	and further,
	\begin{align*}
		\dfrac{\ell_{H}(f(\alpha))}{\ell_{H_0}(\alpha)}
		\leq
		1 + \dfrac{ \sum_{i=1}^{n} \ell (\upgamma_{i,H}) }{\ell_{H_0}(\alpha)}
		\leq
		1 + \dfrac{ \sum_{i=1}^{n} \ell (\upgamma_{i,H}) }{\sum_{i=1}^{n} \ell (\upgamma_{i,H_0})}
		\leq
		1 + M_HM_{H_0}
	\end{align*}
	where the last inequality follows by the \cref{lemma:bound_for_X}. Since the quantities $ M_H $ and $ M_{H_0} $ depends only on  length of the geodesics $ \beta $ and $ \beta' $ and length of the boundary curves of $ \X $, by taking logarithms,
	\begin{align*}
		\log \dfrac{\ell_{H}(f(\alpha))}{\ell_{H_0}(\alpha)} \leq \log M_HM_{H_0}
	\end{align*}
	we see that 
	\begin{align*}
		\sup_{\alpha \in \mathscr{S}(\Sigma)} \bigg\{\log  \dfrac{\ell_{H}(f(\alpha))}{\ell_{H_0}(\alpha)} \text{ , } \log \dfrac{\ell_{H_0}(\alpha)}{\ell_{H}(f(\alpha))} \bigg\} < \infty 
	\end{align*}
	
	Since $ f_{t_1, \ldots t_n, \ell_{1}, \ldots, \ell_{n}} $ is a composition of Fenchel-Nielsen twist and length deformations it follows that it is length-spectrum bounded. We have shown that $ \big[f,H\big] \in \teich_{ls}\entre{H_0} $.
	
	It is straightforward to check that the inclusion map is continuous. It suffices to show that if two elements of $ \teich^{fs}\entre{H_0} $ are close to each other then so are the corresponding geodesic length functions.
	We start with a point $ \big[f,H\big] $ in $ \teich^{fs}\entre{H_0} $. Let $ \epsilon > 0 $.  Let $ \big[f',H'\big] $ be an arbitrary point in $ \teich^{fs}\entre{H_0} $ such that $ d_{ls}\big(\big[f,H\big], \big[f',H'\big]\big) < \epsilon $. Let $ \alpha $ be a simple closed curve in $ \Sigma $. From the definition, $ \ell_H' (f'(\alpha)) < \ell_H (f(\alpha)) e^{2\epsilon} $. Now we put $ \delta = \big(e^{2\epsilon} + 1\big) \ell_H (f(\alpha))$ and we have
	\begin{align*}
		\big| \ell_{H'}(f'(\alpha)) - \ell_{H}(f(\alpha)) \big|
		<
		\ell_{H'}(f'(\alpha)) + \ell_{H}(f(\alpha))
		\leq
		\ell_H (f(\alpha)) e^{2\epsilon} +  \ell_{H}(f(\alpha))
		=
		\delta
	\end{align*}
	as required.
\end{proof}

\subsection{The density of the finitely supported Teichmüller space in $ \teich^{0}_{ls} (H) $}
\begin{definition}\label{definition:asymptotic_definition}
	A length-spectrum bounded homeomorphism $ f: S_{1} \rightarrow S_{2} $ where $ S_{1} $ and $ S_{2} $ endowed with hyperbolic structures $ H_{1} $ and $ H_{2} $ respectively are two surfaces homeomorphic to $ \Sigma $ is defined to be \textbf{asymptotic isometry} if for every $\epsilon > 0 $ there exists a topologically finite type subsurface $ E $  of $ S_1 $ such that the length-spectrum constant $ \L(f_{\restriction_{H_{1}-E}}) $ is less than $ 1+\epsilon $. In this case, we say that the marked hyperbolic structures $ (f_1,H_1) $ and $ (f_2, H_2) $ are \textbf{asymptotically isometric}.
\end{definition}

Any two marked hyperbolic surfaces $ (f_1,H_1) $ and $ (f_2, H_2) $ where $ f_1 $ and $ f_2 $ are asymptotic isometries are \textit{equivalent} if there is a homeomorphism $ \varphi : \Sigma \rightarrow \Sigma $ homotopic to the identity and an isometry $ h : H_1 \rightarrow H_2 $ such that following diagram commutes: 
\[
\begin{tikzcd}
	\Sigma  \arrow[swap]{d}{\varphi} \arrow{r}{f_{1}}  & H_1 \arrow{d}{h} \\
	\Sigma \arrow{r}{f_{2}}  & H_2
\end{tikzcd}
\]

\begin{definition}
	Let $ H_0 $ be a base hyperbolic structure on $ \Sigma $. The space consisting of equivalence classes $ \big[f,H\big] $ where $ f : H_0 \rightarrow H $ is an asymptotic isometry is called \textbf{the little Teichmüller space} and denoted by $ \teich^{0}_{ls} \entre{H_0}$ 
\end{definition}

\begin{theorem}\label{theorem:ls_dense}
	The finitely supported Teichmüller space $ \teich^{fs} \entre{H_0} $ is dense in the little Teichmüller space $ \teich^{0}_{ls} \entre{H_0} $ where both spaces are considered to be subspaces of the length-spectrum Teichmüller space $ \teich_{ls} \entre{H_0} $ and endowed with the length-spectrum metric.
\end{theorem}

\begin{proof}
	We shall first show that the closure $ \overline{\teich^{fs} \entre{H_0}} $ is contained in $ \teich^{0}_{ls} \entre{H_0} $. Let $ x = \big[f,H\big] \in  \overline{\teich^{fs} \entre{H_0}} $. We need to show that for all $ \epsilon > 0 $ there is a topologically finite type subsurface $ E \subseteq \Sigma $ such that $ \L (f_{\restriction_{\Sigma-E}}) < 1 + \epsilon $. By definition, the intersection $ B (x,\epsilon) \cap \teich^{fs} \entre{H_0} $ is nonempty. Then there is a $ y = \big[f_1, H_1\big] $ satisfying
	\begin{itemize}
		\item $ \L (f \circ f_{1}^{-1}) < 1 + \epsilon $
		\item there is a subsurface $ E_1 \subseteq \Sigma $ such that $ f^{*}_{1} H_1 = H_0 $ on $ \Sigma \setminus E_1 $.  
	\end{itemize}
	Now we put $ E = E_1 $. Using $ \L (f_{1 \restriction_{\Sigma-E}}) = 1 $ we have,
	\begin{align*}
		\L (f_{\restriction_{\Sigma-E}}) 
		= 
		\L (((f \circ f_{1}^{-1} ) \circ f_1 )_{\restriction_{\Sigma-E}})
		\leq
		\L (f \circ f_{1}^{-1} ) \L (f_{1 \restriction_{\Sigma-E}})
		\leq
		1 + \epsilon
	\end{align*}
	as required.
	
	To show the reverse inclusion, let $ x = [ f,H ] $ be in $ \teich^{0}_{ls}\entre{H_0} $. Let  $ \epsilon > 0 $. We will show that the intersection $ B(x,\epsilon) \cap \teich^{fs}\entre{H_0} $ is nonempty. We know that there exists a subsurface $ E \subseteq \Sigma $ such that $ \L (f_{\restriction_{\Sigma-E}}) < 1 + \epsilon $. We take a pair of pants decomposition of the subsurface $ E $ and complete it into a pair of pants decomposition of the whole surface $ \Sigma $. Note that each of the curves in the pair of pants decomposition of $ E $ lies in a subsurface of type $ (0,4) $ or $ (1,1) $. Let $ E' $ be the union of these subsurfaces. Clearly, $ E \subseteq E' $ and therefore, $ \L (f_{\restriction_{\Sigma-E'}}) < 1 + \epsilon $. Moreover, $ E' $ contains the curves in the pair of pants decomposition of $ E $ as interior curves, and together with its boundary components we have a pair of pants decomposition of 
	$ \P_E' \cup \B_E' $ of $ E' $ where $ \P_E' = \{ \beta_{1}, \ldots, \beta_{i} \} $ is the set of interior decomposing curves and $ \B_E = \beta_{i+1}, \ldots, \beta_{n} $ is the set of boundary components of $ E' $. 
	
	Let us consider the pullback $ f^{*}H $ and straighten the curves in the pair of pants decomposition of $ E' $ with respect to it. This specifies a $ d = (t_1, \ldots, t_n, \ell_{1}, \ldots, \ell_{n}) $ and allows us to define a finitely supported hyperbolic surface $ z = \big[f_d,H_d\big] $. By construction, $ d_{ls}(z,x) < \epsilon $. This finishes the proof.
\end{proof}

\begin{theorem}[Alessandrini, Liu, Papadopoulos, Su \cite{MR2846324}; Theorem 4.5]\label{theorem:ls_complete}
	For any hyperbolic structure $ H_0 $ on $ \Sigma $ the metric space $ \big( \teich_{ls} \entre{H_0} , d_{ls}\big) $ is complete.
\end{theorem}

\begin{corollary}\label{corollary:T_0_complete}
	The space $ \teich^{0}_{ls} \entre{H_0} $ of asymptotically isometric hyperbolic surfaces is complete with respect to the length-spectrum metric $ d_{ls} $.
\end{corollary}

\begin{proof}
	By \autoref{theorem:ls_complete} for any base hyperbolic surface $ H_0 $ the length-spectrum Teichmüller space $ \teich_{ls} \entre{H_0} $ is complete. In \autoref{theorem:ls_dense} we have shown that $ \teich^{0}_{ls} \entre{H_0} $ is the closure of $ \teich^{fs}\entre{H_0} $ with respect to the $d_{ls}$, therefore $ \teich^{0}_{ls} \entre{H_0} $ is closed in $ \teich_{ls} \entre{H_0} $, hence it is complete.
\end{proof}

We note that Papadopoulos, Alessandrini, Liu and Su proved in \cite{MR2846324} that the space $ (\teich_{ls}\entre{H_0}, d_{ls})$ is complete as a metric space regardless of the choice of the base hyperbolic surface. The quasiconformal Teichmüller space of a Riemann surface of conformally infinite type can be also endowed with the length-spectrum metric. In the article \cite{MR1996441} Shiga showed that the space $ (\teich_{qc}\entre{R_0}, d_{ls})$ is complete if its base point satisfies the condition (\ref{shiga-condition}).

\subsection{Teichmüller metric on the finitely supported Teichmüller space}
We shall turn our attention to the conformal structures on a surface $ \Sigma $ of infinite topological type. In this section we first provide the indispensable background for quasiconformal maps and then our main objective will be to show that the finitely supported Teichmüller space is contained in the quasiconformal Teichmüller space and that $ \teich^{fs} \entre{H_0} $ is endowed with the Teichmüller metric. 
\begin{wrapfigure}{r}{5.5cm}
	\centering
	\includegraphics[width=5.5cm]{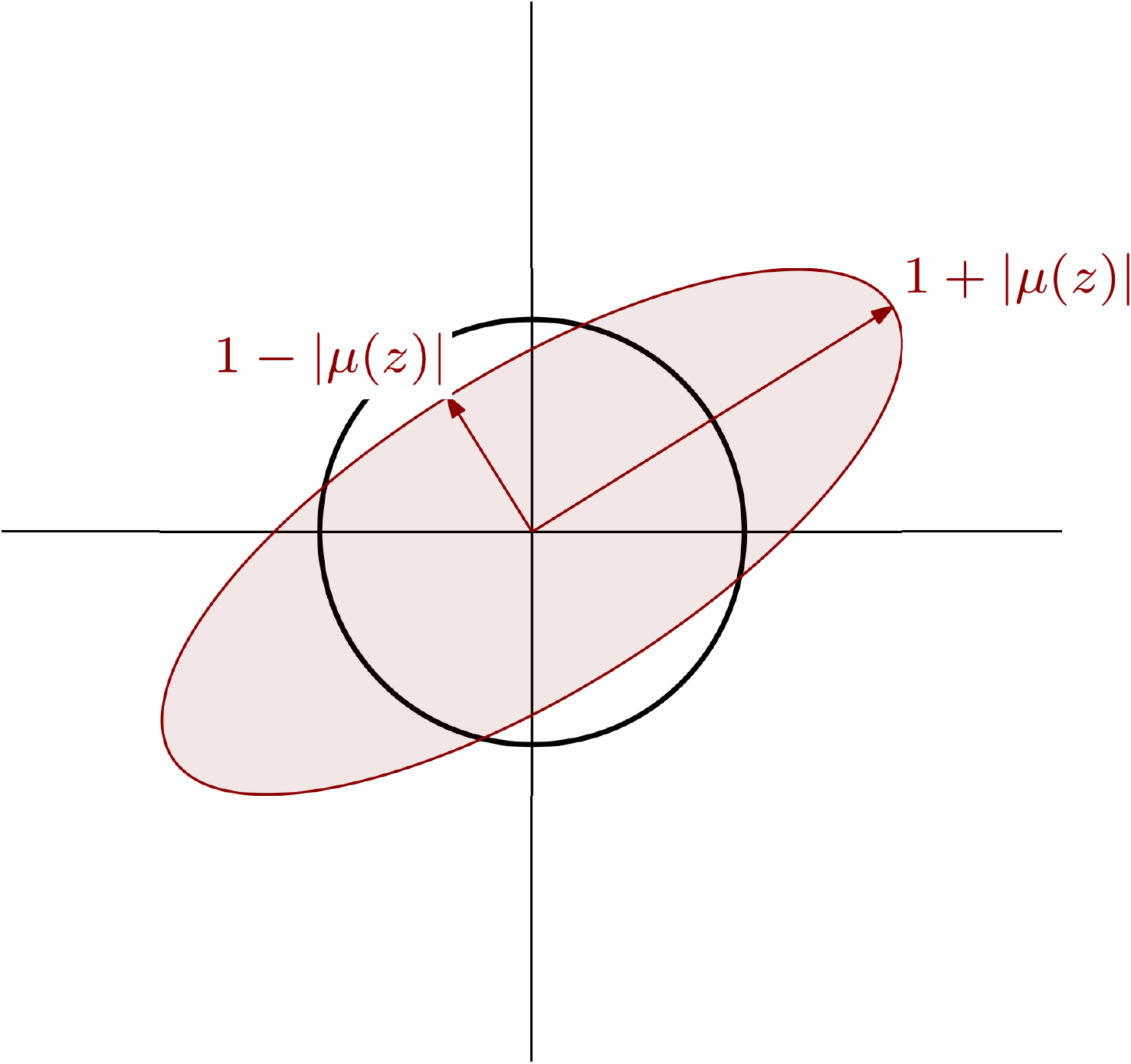}
	\caption{\small { quasiconformal distortion}}\label{fig:circle_distortion}
	\vspace{-3cm}
\end{wrapfigure}

Let $ \Omega $ be a domain in $ \C $ and $ f : \Omega \rightarrow \C $ be an orientation preserving $ C^{1} $-homeomorphism onto $ f(\Omega)$. The ordinary partial derivatives of $ f $ are defined by
\begin{center}
	$ f_z = \dfrac{1}{2}\big( f_x - i f_y \big) $ \text{ and }  $f_{\overline{z}} = \dfrac{1}{2}\big( f_x + i f_y \big) $
\end{center}
The mapping $ f $ takes the infinitesimal circles centered at a point $ z $ to the infinitesimal ellipses centered at $ f(z) $. The dilatation
\begin{center}
	$ \K(z) = \dfrac{\big| f_{z}(z) \big| + \big| f_{\overline{z}}(z) \big|}{\big| f_{z}(z) \big| - \big| f_{\overline{z}}(z) \big| } $
\end{center}
at the point $ z $ is the ratio of the major axis of the ellipse to the minor axis. \autoref{fig:circle_distortion} shows the effect of $ f $ on circles where $ \mu (z) = f_{\overline{z}}(z)/f_z(z) $ called the \textit{Beltrami coefficient} of $ f $. Such a mapping $ f : \Omega \rightarrow \C $ is called \textit{quasiconformal} if $ \K (z) $ is uniformly bounded on $ \Omega $, that is to say, the \textit{complex dilatation} $ \K(f) $ \textit{of} $ f $ defined as 
\begin{center}
	$ \K(f) = \underset{z\in \Omega}{\sup}\K(z) $
\end{center}
is finite. This is clearly equivalent to say that there exists some real number $ K \geq 1 $ such that the inequality
\begin{align}\label{quasi}
	\big| f_{\overline{z}} \big| \leqslant \dfrac{K-1}{K+1}  \big| f_{z} \big|
\end{align}
holds on $ \Omega $, and $ f $ is called $ K $-quasiconformal. Notice that $ f $ is quasiconformal if and only if it is $ K $-quasiconformal for some $ K $. A $ 1 $-quasiconformal mapping is conformal.

As in the classical theory of quasiconformal mappings, the condition of being continuously differentiable is restrictive in the definition above, and most of the mappings do not possess this property including the ones we use in the following sections. Instead, a more general definition involves the \textit{generalized} partial derivatives of $ f $ which are required to be locally square-integrable and (\ref{quasi}) is satisfied almost everywhere on $ \Omega $. 
Let $ f(x,y)$, $ g(x,y) $ and $ h(x,y) $ be measurable functions defined in a domain $ \Omega $ in $ \C $ which are square-integrable over every compact subset of $ \Omega $. One says that $ g $ and $ h $ are the \textit{generalized partial derivatives} of $ f $ with respect to $ x $ and $ y $ respectively and writes $ g :=f_x $ and $ h := f_y $ if for every $ \omega \in C^{1} $ which vanishes outside a compact subset of $ \Omega $ one has
\begin{align*}
	\iint_{\Omega} (f\omega_x + g\omega) \text{dx}\text{dy} = \iint_{\Omega} (f\omega_y + hw) \text{dx}\text{dy} = 0
\end{align*}
We shall record the analytic definition of quasiconformal mappings and their basic properties. The classical references on the definition are \cite{MR2241787} and \cite{MR0114898}. 
\begin{definition}
	Let $ \Omega $ be a domain in $ \C $. Let $ K \geq 1 $. A homeomorphism $ f : \Omega \rightarrow \C $ onto $ f(\Omega) $ is called \textbf{$\textbf{K}$-quasiconformal} if its generalized partial derivatives are in $ \mathcal{L}_{loc}^{2}(\Omega) $ (i.e. locally square-integrable on $ \Omega $) and satisfy 
	\begin{align}\label{quasiconformal}
		\big| f_{z} \big| \leqslant \dfrac{K-1}{K+1}  \big| f_{\overline{z}} \big|   
	\end{align}
	almost everywhere for some $ K $. A homeomorphism is \textbf{quasiconformal} if it is $K$-quasiconformal for some $ K \geq 1 $.. 
	We set $ k = \frac{K-1}{K+1}$ so that $ 0 \leq k < 1 $. For the smallest $ k $ satisfying (\ref{quasiconformal}) we define $ \K(f) = \frac{1+k}{1-k}$ and call this quantity the \textbf{quasiconformal dilatation} of $ f $. 
\end{definition}

A \textit{quadrilateral} $ Q = Q (z_1, z_2, z_3, z_4) $ is a topological image of the closed square with vertices $ (0,0), (1,0) $, $ (1,1) $ and $ (0,1) $ which are sent respectively to the vertices of $ z_1, z_2, z_3, z_4 $ of $ Q $. Two quadrilaterals $ Q (z_1, z_2, z_3, z_4) $ and $ Q^{\prime} (z_1^{\prime}, z^{\prime}_2, z^{\prime}_3, z^{\prime}_4) $ are conformally equivalent if there is a conformal map from $ Q $ to $ Q^{\prime} $ sending $ z_i $ to $ z^{\prime}_i $ for $ i = 1,2,3,4 $. It follows from Riemann mapping theorem that every quadrilateral $ Q (z_1, z_2, z_3, z_4) $ can be conformally mapped onto an Euclidean rectangle $ Q (0, a, a+ib, ib) $ for some $ a, b \in \R $. The uniquely determined quantity $ a / b $ is the ratio of the width of the rectangle to its height and is called the \textit{modulus} of $ Q $ and one writes
\[ \bmod \big( Q (z_1, z_2, z_3, z_4) \big) = \dfrac{a}{b}  \]
The modulus of a quadrilateral is a conformal invariant on rectangles by Grötzsch inequality. It is easy to see that $ \bmod \big( Q (z_2, z_3, z_4,z_1) \big) = b/a $. 
An orientation-preserving homeomorphism of a domain $\Omega $ onto its image is $ K $-quasiconformal if \[ \bmod \big( f(Q) \big) \leq K \bmod (Q) \]
for every quadrilateral $ Q $ in $ \Omega $. This is known as \textit{the geometric definition} of quasiconformality and it is equivalent to the analytic definition, for a proof see Bers \cite{MR156969}. The inverse of a $ K $-quasiconformal map is also $ K $-quasiconformal. If $ f $ is a $K_1$-quasiconformal on a domain $ \Omega $ and $ g $ is a $K_2$-quasiconformal on $ f(\Omega) $ then the composition map $ g \circ f $ is $K_1K_2$-quasiconformal. 

The notion of quasiconformality can be extended to the mappings between Riemann surfaces. Let $ f : R_1 \rightarrow R_2 $ be a homeomorphism between Riemann surfaces $ R_1 $ and $ R_2 $, $ p $ be a point in $ R_1 $ and $ q = f(p) $. Let $ z_{\alpha} $ and $ z_{\beta} $ be local coordinates for $ p $ and $ q $ respectively (i.e, these are conformal homeomorphisms of neighborhoods of $ p $ and $ q $ into the complex plane.)  The dilatation of $ f $ at a point $ p $ is defined to be the dilatation of the mapping $ z_{\beta} \circ f \circ z^{-1}_{\alpha} $ at the point $ z_{\alpha}(p) $ if it is finite, otherwise it is defined to be $ \infty $. It follows from the fact that the composition of a $K$-quasiconformal map with a conformal map is $K$-quasiconformal that the definition is independent of the choice of the local coordinates. 

We note an important difference between the Teichmüller spaces of Riemann surfaces of finite conformal type and the Teichmüller spaces of Riemann surfaces of infinite conformal type. Recall that a Riemann surface is of \textit{finite conformal type} if it obtained from a compact Riemann surface by removing a finite number of points. These points correspond to punctures where a neighborhood of each such puncture is conformally equivalent to a punctured disk. Any homeomorphism $ f: R \rightarrow R' $ between two Riemann surfaces $ R $ and $ R' $ of finite conformal type is homotopic to a quasiconformal homeomorphism between them, see for instance Bers \cite{MR132175}. It follows that if $ R_0 $ and $ R_0' $ are two base points with the trivial marking, and $ g : R_0 \rightarrow R_0' $ is a quasiconformal homeomorphism the mapping $ (f,R) \mapsto (f \circ g, R) $ on the marked Riemann surface induces a mapping from $ \teich_{qc}\entre{R_0} $ onto $ \teich_{qc}\big(R_0'\big) $, so we have $ \teich_{qc}\entre{R_0} =  \teich_{qc}\big(R_0'\big) $. However, for Riemann surfaces of infinite conformal type there exist homeomorphisms between such surfaces which are not homotopic to quasiconformal ones. In order to emphasize on this distinction, we say that two marked Riemann surfaces $ \big(f_1, R_1\big) $ and $ \big(f_2, R_2\big) $ are \textit{quasiconformally equivalent} if there is a quasiconformal homeomorphism $ h : R_1 \rightarrow R_2 $ such that $ h \circ f_1  $ is homotopic to $ f_2 $.

From now on, we fix a Riemann surface structure $ R_0 $ on $ \Sigma $ and consider the set of marked Riemann surfaces $ (f,R) $ where $ R $ is a Riemann surface and $ f : R_0 \rightarrow R $ is a quasiconformal homeomorphism called the marking of $ \Sigma $. Two such marked Riemann surfaces $ \big(f_1, R_1\big) $ and $ \big(f_2, R_2\big) $ are \textit{conformally equivalent} if there is a conformal homeomorphism $ h : R_1 \rightarrow R_2 $ such that $ h \circ f_1  $ is homotopic to $ f_2 $. 

\begin{definition}
	The space of all equivalence classes $ \big[f,R\big] $ of marked Riemann surfaces which are quasiconformally equivalent to $ \big[Id,R_{0}\big] $ is \textbf{the quasiconformal Teichmüller space} $ \teich_{qc}\entre{R_0} $. The point $ [Id,R_{0}] $ is the \textit{base point} of the space $\teich_{qc}\entre{R_0} $. 
\end{definition}

Therefore, the Teichmüller space $ \teich_{qc}\entre{R_0} $ consists of the equivalence classes of marked Riemann surfaces which are quasiconformally equivalent to the base point. It can be endowed with the \textit{Teichmüller metric (or quasiconformal metric)} defined as 
\begin{align*}
	d_{qc}\bigg(\big[f_1,R_1\big], \big[f_2,R_2\big]\bigg) = \dfrac{1}{2} \inf_{f} \text{ } \log \K(f) 
\end{align*}
where the infimum is taken over all quasiconformal homeomorphisms $ f : R_1 \rightarrow R_2 $ homotopic to $ f_2 \circ f_{1}^{-1}$. As indicated above, there exist homeomorphic Riemann surfaces $ R_0 $ and $ R_0' $ for which $ \teich_{qc}\entre{R_0} \neq \teich_{qc}\big(R_0'\big) $. 



The uniformization theorem says that $ R_0 $ is a quotient of the unit disk by a group of biholomorphic automorphisms acting freely and properly discontinuously on the unit disk. Since biholomorphic automorphisms of the unit disk preserves the hyperbolic metric on the unit disk, we naturally obtain an induced hyperbolic structure $ H_0 $ on $ R_0 $. We consider the hyperbolic structure $ H_0 $ on $ \Sigma $ which represents the Riemann surface structure $ R_0 $, and denote the quasiconformal Teichmüller space $ \teich_{qc}\entre{R_0} $ of $ R_0 $ by $ \teich_{qc}\entre{H_0} $. 

We have seen in \autoref{theorem:finitely_supported_is_ls} that
shrinking or expanding a simple closed geodesic give rise to a hyperbolic surface which is finitely supported with respect to the base surface. Now we aim to describe that the Fenchel-Nielsen length deformation of a surface can be realized by quasiconformal maps, where the method is essentially presented in \cite{MR1954866} which also involves a careful estimate of the complex dilatation of the deformation which can be performed so that it satisfies all the requirement that we will need in \autoref{theorem:inclusion_qc}. 

We shall briefly present the idea here. 
If $ T_1 $ and $ T_2 $ are two hyperbolic triangles with nonzero angles there is a quasiconformal map between them \cite[Lemma 3.1]{MR1954866}. Further, by dividing hyperbolic polygons into hyperbolic triangles one can glue the quasiconformal mappings between the
\begin{wrapfigure}{r}{6.5cm}
	\centering
	\includegraphics[width=6cm]{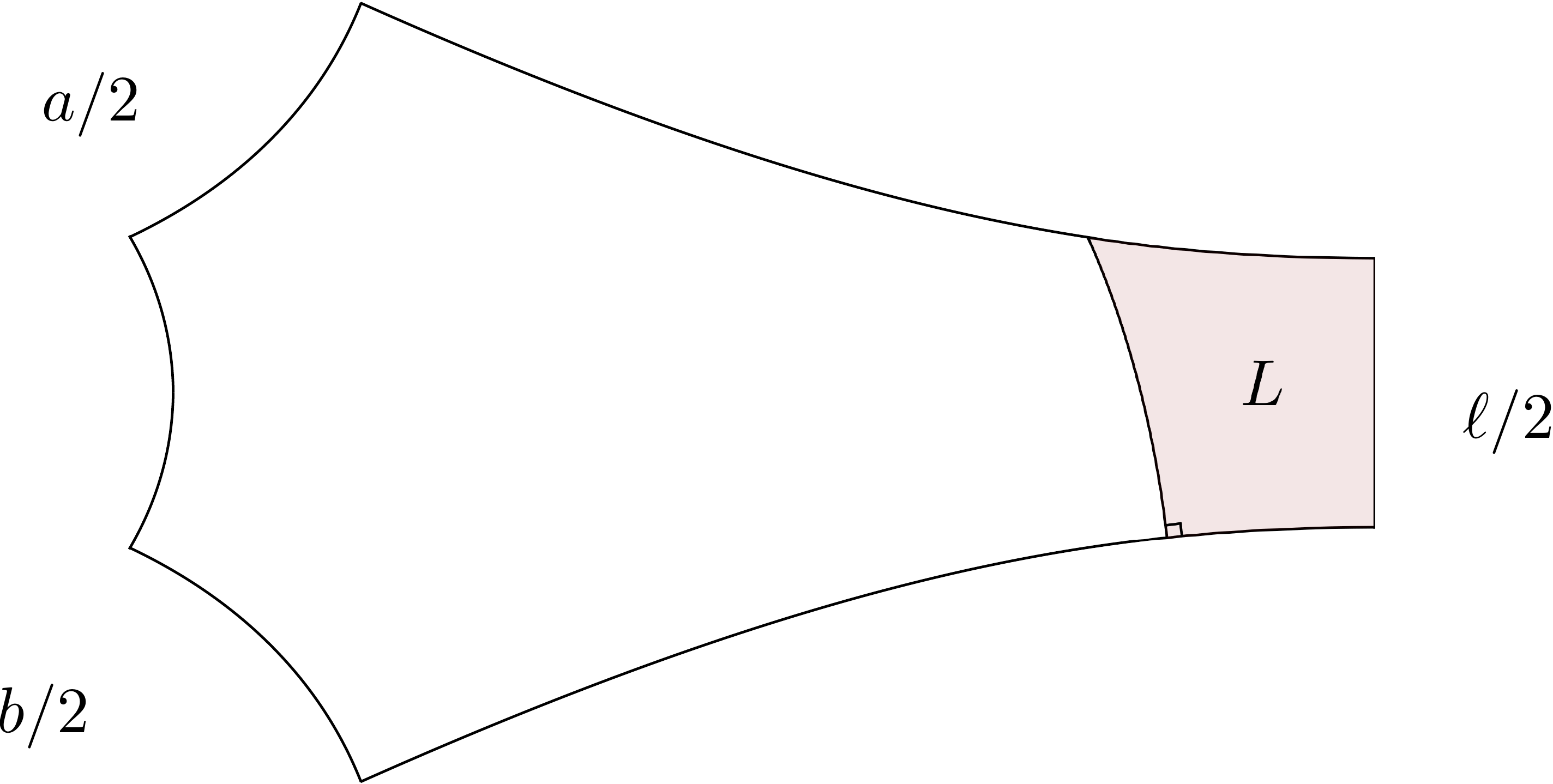}
	\caption{\small {divide the hexagon into triangles and quadrilaterals}\label{fig:bishop}}
\end{wrapfigure}
triangles to produce a quasiconformal mapping between two given hyperbolic polygons \cite[Corollary 3.3]{MR1954866}.
We take two hyperbolic right-angled hexagons $ \H $ and $ \H' $, with edges labeled as $ (a/2,b/2,\ell/2) $ and $ (a/2,b/2,\ell'/2) $. If a side of $ \H $ is too short then some other sides can be long. In order to find a uniform bound for the quasiconformal constant of the mappings between triangles, the hexagon $ \H $ is regarded as a union of  bounded polygon and at most three quadrilaterals with three right angles \cite[Lemma 5.1]{MR1954866}. For each short side of the hexagon $ \H $ we consider a quadrilateral $ L $ chosen in a certain way whose base is $ \ell $ and a quadrilateral $ L' $ whose base is $ \ell' $ on $ \H' $ can  so that we can construct a quasiconformal map from $ L $ to $ L' $, where we can control the quasiconformal constant. Now, divide the hexagon $ \H \setminus L $ into hyperbolic triangles. Each triangle can be mapped quasiconformally to the corresponding hyperbolic triangle in $ \H' \setminus L' $. All quasiconformal mapping are affine on each side of triangles and quadrilaterals, hence gluing these maps we get a quasiconformal mapping $ \H \rightarrow \H' $. All these maps have a uniform bound by a quasiconformal constant. Since any pair of pants can be realized as a union of two symmetric hyperbolic hexagons with alternate sides identified, quasiconformal mappings between hexagons can be used to construct mappings between pairs of pants. The continuity of the map is guaranteed by symmetry. 

\begin{theorem}[Bishop \cite{MR1954866}]\label{theorem:bishop}
	Let $ N $ be a real number such that if $ \X $ and $ \X' $ are two X-type hyperbolic surface with geodesic boundary components $ \partial_1, \partial_2, \partial_3, \partial_4 $ and $ \partial'_{1}, \partial'_{2}, \partial'_{3}, \partial'_{4} $ respectively of lengths $ \ell(\partial_{i}) = \ell(\partial'_{i})  $ for $ i = 1,2,3,4 $ all smaller than $ N $ and the central curves $ \beta $ and $ \beta' $ of lengths $ \ell(\beta) $ and $ \ell(\beta') $ both smaller than $ N $, there exists a $ C(N) $-quasiconformal map $ f : \X \rightarrow \X' $ which
	\begin{enumerate}[label=\roman*., leftmargin=\parindent]
		\item is an isometry on the boundary components of  $ \X $, 
		\item is affine on $ \beta $ and maps $ \beta $ onto $ \beta' $,
		\item has dilatation satisfying $ \log \K (f) \leq C(N) \Bigg| \log \dfrac{\ell(\beta)}{\ell(\beta')} \Bigg| $,
		\item respects the twist number for $ \beta $, that is, $ \tau_{\X}(\beta) =  \tau_{\X'}(\beta') $
	\end{enumerate}
\end{theorem}
\begin{proof}
	Using \cite[Theorem 1.1]{MR1954866} a quasiconformal homeomorphism between two pairs of pants $ P= (\partial_1, \partial_2, \beta) $ and $ P'=(\partial'_1, \partial'_2, \beta) $ can be constructed for the case $  \log \big| \ell(\beta) / \ell(\beta') \big| \leq 2 $ and has the dilatation satisfying the bound in $ (iii) $. Note that for the case where $  \log \big| \ell(\beta) / \ell(\beta') \big| > 2 $ we can take a finite sequence of pairs of pants $ P_i = (\partial_1, \partial_2, \beta_i) $ for which the curves $ $ $  \log \big| \ell(\beta_i) / \ell(\beta_{i+1}) \big| \leq 2 $ where $ P=P_1$ and $ P' = P_n $ and construct a quasiconformal homeomorphism between two successive pairs of pants. The composition of these mappings is also quasiconformal with dilatation satisfying the inequality in $(iii)$. Similarly we construct a quasiconformal homeomorphism from $ Q= (\beta, \partial_3, \partial_4) $ to $ Q' = (\beta', \partial_3', \partial'_4) $. 
	
	Now we glue two mappings to get a homeomorphism from $ \X $ to $ \X' $. The quasiconformal constant at the points on the curve $ \beta $ can be neglected since they form a set of measure zero. So the mapping is quasiconformal and still satisfies  $ (iii) $. The properties $ (i) $ and $ (ii) $ follows from the construction. The resulting quasiconformal mapping sends the intersection point of one boundary component with a geodesic seam to the corresponding point on the corresponding boundary component in the target pair of pants guaranteeing $ (iv) $.
\end{proof} 

Using the construction in this theorem, we can choose an $ X $-type subsurface $ X_1 $, cut the surface along the boundary curves of this subsurface and replace it by another $ X$-type surface $ X_2 $. On $ \Sigma \setminus \X $ the conformal structure remains the same. In this way we get a deformation of $ R_0 $ which we shall denote by $ R_{\ell + r } $. We use this to prove the following theorem which is analogous to \autoref{theorem:finitely_supported_is_ls}.

\begin{theorem}\label{theorem:inclusion_qc}
	Let $ \Sigma $ be a surface endowed with a Riemann surface structure $ R_0 $. Every finitely supported marked Riemann surface structure on $ \Sigma $ is quasiconformally equivalent to the base conformal structure. There is a natural continuous inclusion mapping the space $ \teich^{fs}\entre{H_0} $ of finitely supported marked Riemann surfaces to the quasiconformal Teichmüller space $ \teich_{qc}\entre{H_0} $.
	\begin{proof}
		As above we consider the surface $ \Sigma $ with a base conformal structure $ R_0 $ and with the corresponding hyperbolic structure $ H_0 $. Let $ [f,H] $ be a marked hyperbolic structure supported on a subsurface $ K $ of  $ \Sigma $. We will construct a deformation $ \big(f_d,H_d\big) $ of the base Riemann structure which lies in the equivalence class $ \big[f,H\big] $ where $ f_d : H_0 \rightarrow H_d $ is a quasiconformal homeomorphism. Let $ \beta $ be a simple closed $ H_0 $-geodesic lying in the interior of the surface $ \Sigma $. 
		
		In what follows, there exist quasiconformal homeomorphisms realizing Fenchel-Nielsen twist and Fenchel-Nielsen length deformation along a simple closed geodesic $ \beta $ in two steps. Then the general case will be derived from them.
		
		\textbf{Step 1.} The time-$ t $ Fenchel-Nielsen twist along the curve $ \beta $ is performed by taking a collar neighborhood of $ \beta $, cutting the surface along $ \beta $ and gluing it back along the boundary components after twisting. Following \cite{MR657237}, we will construct a quasiconformal homeomorphism which represents the time-$ t $ Fenchel-Nielsen twist along the curve $ \beta $ and which is equal to the identity map $ \id $ outside a collar so that it is clearly finitely supported. We start by taking the $ \omega $-tubular neighborhood of $ \beta $ 
		\begin{align*}
			W_{\beta}  = \big\{ p \in \Sigma \mid d_{H_0}(p, \beta ) \leq \omega(\beta) \big\} 
		\end{align*}
		where $ d_{H_0} $ is the hyperbolic distance on $ R_0 $ and   
		\begin{align*}
			\omega(\beta) = \sinh^{-1} \Bigg( \dfrac{1}{\sinh \big( \frac{1}{2} \ell_{H} (\beta) \big) } \Bigg)
		\end{align*}
		Then $ W_{\beta} $ is a collar embedded in $ \Sigma $. 
		We identify the universal covers of $ R $ and $ R_{\beta, t} $ with the upper half-plane model of the hyperbolic plane $ \half $, and write $ \Sigma = \half / G $ where $ G $ is a Fuchsian group.
		By taking a conjugate of the group $ G $ by an isometry of $ \half $, we may assume that the hyperbolic transformation $ A(z) = \lambda z $, where $ \lambda = e^{\ell (\beta) } $, belongs to  $ G $ and that the axis of $ A $ covers $ \beta $.
		
		Let $ \widetilde{W_{\beta}}$ be a lift of   $ W_{\beta} $ on $ \half $ containing the axis of $ A $. As shown in Figure \ref{wp}, we write 
		\begin{align*}
			\widetilde{W_{\beta}} = \Big\{ z \in \half \mid \frac{\pi}{2} - \theta_{0}  < \arg z < \frac{\pi}{2} + \theta_{0}  \Big\}
		\end{align*}
		for some suitable $ \theta_{0} $ with $ 0 < \theta_{0} < \frac{\pi}{2}$. 
		
		Now for $ t \in \R $ we define a quasiconformal mapping $ w^{t} $of $ \half $ onto itself by
		\begin{align*}
			w^{t}(z) = 
			\begin{cases} 
				z & \text{if } 0 < \arg z < \frac{\pi}{2} - \theta_{0}  \\
				z \exp \Bigg( t \bigg( \dfrac{\theta - \frac{\pi}{2} + \theta_{0}} {2\theta_{0}}\bigg) \Bigg) & \text{if } \frac{\pi}{2} - \theta_{0} \leq \arg z \leq \frac{\pi}{2} + \theta_{0} \\
				z \exp (t) &  \text{if } \frac{\pi}{2} + \theta_{0} < \arg z < \pi
			\end{cases}
		\end{align*}
		The complex dilatation $ \mu^{t}(z) $ of the mapping $ w^{t}(z) $ is 
		\begin{align*}
			\mu^{t}(z)
			=
			\dfrac{i\dfrac{t}{2 \theta_0}}{2- i\dfrac{t}{2 \theta_0}}
			\chi_{I} (\theta_0) \dfrac{z}{\overline{z}}
		\end{align*}
		for $ z \in \half $ where $ \chi $ denotes the characteristic function of the interval $ I = [ \frac{\phi}{2} - \theta_0 , \frac{\phi}{2} + \theta_0] $ on $ \R $. Now let $ G_{\beta} $ be the set of all elements of $ G $ which cover $ \beta $. We have
		\begin{align*}
			G_{\beta} = \bigg\{ B \circ A \circ B^{-1} : B \in G \bigg\}
		\end{align*}
		
		The effect of the Fenchel-Nielsen twist deformation along the curve $ \beta $ corresponds to a family of self-homeomorphisms of the universal cover $ \half $ which give surgeries along the axis of each element in $ G_{\beta} $. This allows us to construct a family of quasiconformal self-homeomorphisms of $ \half $ along the lifts of the curve $ \beta $ that induces Fenchel-Nielsen twist deformation on $ \Sigma $.
		Next for $ z \in \half $we set 
		\begin{align*}
			\mu^{t}_{G} = \sum_{B \in \langle A \rangle \backslash G} (\mu^{t} \circ B) \dfrac{\overline{B'}}{B'}
		\end{align*}
		where $ \langle A \rangle \backslash G $ denotes the set of right cosets of the group $ G $ with respect to $ \langle A \rangle $, the cyclic group generated by $ A $. From definition it is easy to see that $ \mu^{t}_{G} $ is a $ G $-invariant Beltrami differential. The corresponding $ \mu^{t}_{G} $-quasiconformal self-homeomorphism of $ \half $ induces a quasiconformal self-homeomorphism of $ \Sigma $ which represents the Fenchel-Nielsen time-$t$ twist deformation along $ \beta $.
		\begin{figure}[ht!]
			\centering
			\includegraphics[width=0.9\linewidth]{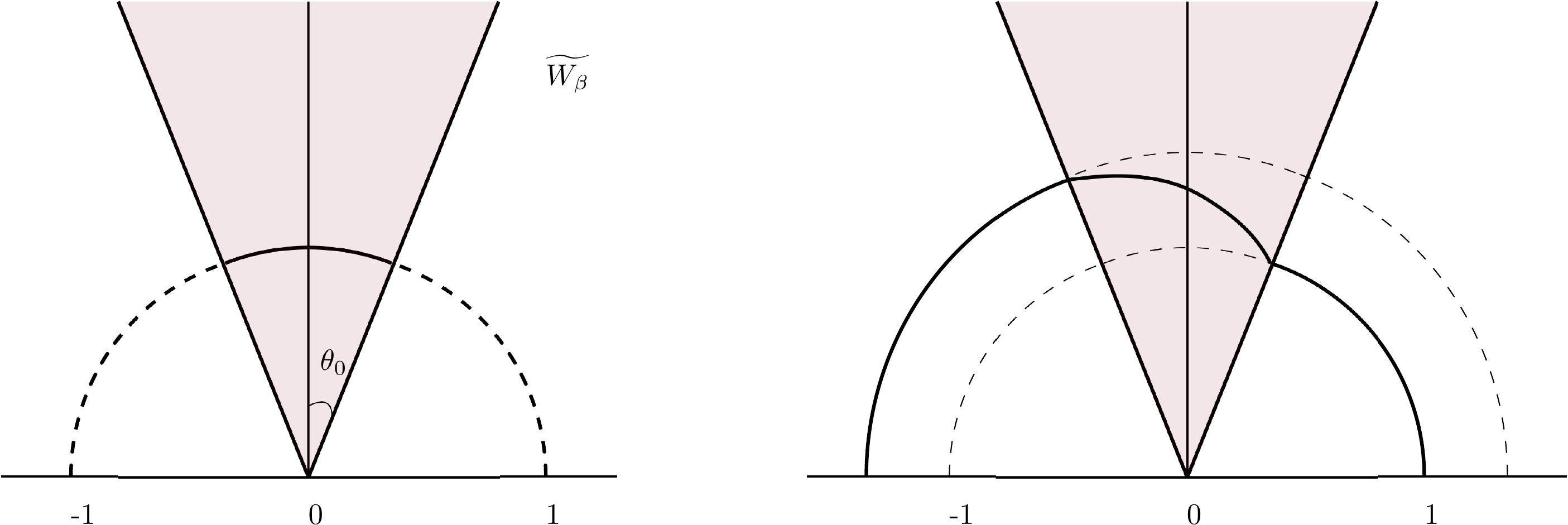}
			\caption{\small {the effect of the map $ w^{t}(z) $ }\label{wp}}
		\end{figure} 
		
		\textbf{Step 2.} Now we construct a quasiconformal homeomorphism which realizes the Fenchel-Nielsen length deformation. Let $ \X $ be an X-type subsurface of $ \Sigma $ containing $ \beta $ with boundary curves $ (\partial_1,\partial_3,\partial_3,\partial_4) $. Let $ r > -\ell $ be a real number where $ \ell $ denotes the length of $ \beta $. By \autoref{theorem:bishop}, it is possible to construct a quasiconformal map from $ \X $ to a X-type surface $ \X' $ with a central curve $ \beta' $ of length $ \ell + r $ with boundary curves $ (\partial'_1,\partial'_3,\partial'_3,\partial'_4) $ of the same hyperbolic lengths as those of the boundary components of $ \X $. Now we take the quasiconformal homeomorphism to be the mapping on $ \X $ and to be identity on $ \Sigma \setminus \X $.
		
		
		\textbf{Step 3.} Now, choose a pair of pants decomposition of $ E $ and $ \P_{E} = \{ \beta_{1}, \ldots, \beta_{n} \} $ be the decomposing curves. The pullback $ f^{*}(H) $ on $ \Sigma $ coincides with the base conformal structure outside $ E $. We straighten the curves in the pair of pants decomposition to get a geodesic pair of pants decomposition $ E $ with respect to the pullback $ f^{*}(H) $. For every interior curve $\beta_i $ we have the twist parameter $ t_i $ and the length parameter $ \ell_i $. This gives an $ n $-tuple of  twist parameters $ t_1, \ldots t_n $ and length parameters $ \ell_{1}, \ldots, \ell_{n} $.
		\begin{center}
			\[
			\begin{tikzcd}
				\Sigma  \arrow{r}{f} \arrow[swap]{dr}{f_{(t_1, \ldots t_n, \ell_{1}, \ldots, \ell_{n})} } & H \arrow{d}{i} \\
				& H_{(t_1, \ldots t_n, \ell_{1}, \ldots, \ell_{n})}
			\end{tikzcd}
			\]
		\end{center}
		
		Since $ f_{t_1, \ldots t_n, \ell_{1}, \ldots, \ell_{n}} $ is a composition of Fenchel-Nielsen twist and length deformations which are quasiconformal maps, we see that $ \big[f,H\big] \in \teich_{qc}\entre{H_0} $. 
		
		The continuity of the inclusion map $ \teich^{fs}\entre{H_0} \hookrightarrow \teich_{qc}\entre{H_0} $ follows easily by Wolpert's lemma and the continuity of the inclusion map $ \teich^{fs}\entre{H_0} \hookrightarrow \teich_{ls}\entre{H_0} $, see the proof of \cref{theorem:finitely_supported_is_ls}.
		
	\end{proof}
\end{theorem} 

\subsection{The density of the finitely supported Teichmüller space in $ \teich^{0}_{qc} (H) $}
A homeomorphism $ f: S_{1} \rightarrow S_{2} $ where $ S_{1} $ and $ S_{2} $ are endowed with Riemann surface structures $ R_{1} $ and $ R_{2} $ respectively are two surfaces homeomorphic to $ \Sigma $ is said to be \textit{asymptotically conformal} if for every $\epsilon > 0 $ there exists a compact subsurface $K$ of $ S_1 $ such that the quasiconformal constant $ \K(f_{\restriction_{S_{1}-K}}) $ is less than $ 1+\epsilon $. 

Any two marked Riemann surface $ (f_1,R_1) $ and $ (f_2, R_2) $ where $ f_1 $ and $ f_2 $ are asymptotically conformal are \textit{equivalent} if there is a homeomorphism $ \varphi : \Sigma \rightarrow \Sigma $ homotopic to identity and a conformal mapping $ h : R_1 \rightarrow R_2 $ such that following diagram commutes: 
\[
\begin{tikzcd}
	\Sigma  \arrow[swap]{d}{\varphi} \arrow{r}{f_{1}}  & H_1 \arrow{d}{h} \\
	\Sigma \arrow{r}{f_{2}}  & H_2
\end{tikzcd}
\]

\begin{definition}
	Let $ R_0 $ be a base Riemann surface structure on $ \Sigma $. The space consisting of equivalence classes $ \big[f,R\big] $ where $ f : R_0 \rightarrow R $ is asymptotically conformal is called \textbf{the little Teichmüller space} and denoted by $ \teich^{0} \entre{R_0}$.
\end{definition}

\begin{theorem}
	The finitely supported Teichmüller space $ \teich^{fs}\entre{H_0} $ is dense in the space $ \teich^{0}\entre{H_0} $ of asymptotically conformal marked Riemann surfaces where both spaces are considered to be subspaces of the quasiconformal Teichmüller space $ \teich_{qc} \entre{H_0} $ and endowed with the quasiconformal metric.
\end{theorem}

\begin{proof}
	The proof is analogous to that of the \autoref{theorem:ls_dense} and follows by the use of the construction in Bishop's theorem, \autoref{theorem:bishop}.
\end{proof}
It is well-known that the space $ \big( \teich_{qc} \entre{H_0} , d_{qc}\big) $ is complete as a metric space. We shall record the following result (which is already known) as a consequence of the theorem above:
\begin{corollary}
	The space $ \teich^{0}_{qc} \entre{H_0} $ of asymptotically conformal Riemann surfaces is complete with respect to the Teichmüller metric $ d_{qc} $.
\end{corollary}


\section{Fenchel-Nielsen coordinates on finitely supported Teichmüller spaces}\label{Sec:FN-coordinates}

Having enhanced our supply of the finitely supported Teichmüller space $ \teich^{fs} \entre{H_0} $ of a hyperbolic surface $ H_0 $ through its relation with the length-spectrum and quasiconformal Teichmüller spaces, we now turn our attention to its parametrization by Fenchel-Nielsen coordinates. The goal of this section is to describe the finitely supported Teichmüller space $ \teich^{fs} \entre{H_0} $ and the space $ \teich^{0}_{ls} \entre{H_0} $ of asymptotically isometric hyperbolic structures in terms of Fenchel-Nielsen coordinates. We prove that the length-spectrum asymptotic Teichmüller space $ \asyteich_{ls}\entre{H_0} $ is contractible.

Once we endow a surface $ \Sigma $  with a hyperbolic metric, each homotopy class of essential simple closed curves on the surface contains a natural representative, namely the geodesic with respect to this metric. The question is whether the resulting collection of geodesic form a pair of pants decomposition. For surfaces of topologically infinite type, we require that the hyperbolic structure is Nielsen-convex. A hyperbolic surface $ \Sigma $ is \textit{Nielsen-convex} if every point of $ \Sigma $ is contained in a geodesic arc whose endpoints are in simple closed curves in $ \Sigma $. A topological pair of pants decomposition on a surface $ \Sigma $ endowed with a hyperbolic structure $ H $ can be turned into a geodesic pair of pants decomposition if and only if $ H $ is Nielsen-convex \cite[Theorem 4.5]{MR2865518}. Throughout the present section, we assume that all hyperbolic surfaces are Nielsen convex and the base hyperbolic surface $ H_0 $ is upper bounded. 

If a surface $ \Sigma $ endowed with a Nielsen-convex hyperbolic structure and $ \P = \{ C_i \}$ is a topological pair of pants decomposition on $ \Sigma $, for each homotopy class of curves $ C_i $ the \textit{length parameter} of $ C_i $ is defined to be the length $ \ell_{H} (C_i) $ of the unique $ H $-geodesic in the homotopy class of $ C_i $. Note that the length parameter $ \ell_{H} (C_i) $ is a positive real number. We define the twist parameter associated to each homotopy class of curves $ C_i $ only when $ C_i $ is not homotopic to a boundary component. In this situation $ C_i $ is contained in a  subsurface of $ \Sigma $ of type $ (0,4) $ or $ (1,1) $ where the boundary curves of the subsurface are elements of $ \P = \{ C_i \} $.
We define the twist parameter associated to the homotopy class of a curve $ C_i $ as we have defined twist number in the beginning of \autoref{Section:cs}. The details can be found in Chapter 4.6 of Thurston \cite{MR1435975}. The twist parameter $ \tau_{i} (C_i) $ is a real number. To sum up, for each Nielsen-convex hyperbolic structure on $ \Sigma $ we can associate a collection of pairs $ \big( \ell_{H}(C_i) , \tau_{H} (C_i ) \big) $ called the \textit{Fenchel-Nielsen parameters} of $ H $ which consists of two parameters for each curve $ C_i $ in the interior of $ \Sigma $ and one single parameter, namely the length parameter, for each boundary curve of the surface $ \Sigma $.

Recall that complex structure on a surface $ \Sigma $ gives rise to a natural hyperbolic metric, called the intrinsic metric.  The intrinsic metric is always Nielsen-convex. In fact, a hyperbolic metric on a surface $ \Sigma $ is Nielsen-convex if it corresponds to the intrinsic metric in a conformal class of a Riemann surface structure on $ \Sigma $. Therefore, in this case there always exists a pair of pants decomposition $ \P = \{ C_i \} $ of $ \Sigma $ in which each curve $ C_i $  is a geodesic with respect to the intrinsic metric.

Let $ H_1 $ and $ H_2  $ be two hyperbolic structures on $ \Sigma $ and $ f: H_1 \rightarrow H_2 $ is a homeomorphism homotopic to identity. If  the quantity 
\begin{align}\label{FN_distance}
	d_{FN}(H_1,H_2)
	=
	\sup_{i \in \N}
	\Bigg\{
	\Bigg| \dfrac{\log\ell_{H_1} (C_i)}{\log \ell_{H_2} (C_i)} \Bigg|
	\text{ , }
	\big| \tau_{H_1}(C_i) - \tau_{H_2}(C_i) \big|
	\Bigg\}
\end{align}
is finite then we say that $ f $ is \textit{Fenchel-Nielsen bounded}.

For a surface $ \Sigma $ endowed with a base hyperbolic structure $ H_0 $, we consider the collection of all marked hyperbolic surfaces $ \entre{f,H} $ where the marking $ f : H_0 \rightarrow H $ is Fenchel-Nielsen bounded with respect to $ \P $.  Two marked hyperbolic surfaces $ \entre{f,H} $ and $ \big(f',H'\big) $ are considered to be equivalent if there exists an isometry $ i : H \rightarrow H' $ homotopic to the map $ f' \circ f^{-1} $. This is an equivalence relation on the collection of the Fenchel-Nielsen bounded marked hyperbolic surfaces. The equivalence class of $ \entre{f,H} $ will be denoted by $ [f,H] $. 

\begin{definition}
	The space of all equivalence classes $ \big[f,H\big] $ of marked hyperbolic surfaces which are Fenchel-Nielsen bounded with respect to $ H_0 $ and $ \P $ is the \textbf{Fenchel-Nielsen Teichmüller space $ \teich_{FN}\entre{H_0} $}. The point $ [Id, H_0] $ is the basepoint of the space $ \teich_{FN}\entre{H_0} $. The Fenchel-Nielsen Teichmüller space $ \teich_{FN}\entre{H_0} $ is endowed with the \textit{Fenchel-Nielsen distance} which is defined to be the quantity $ d_{FN}(H_1, H_2) $ in (\ref{FN_distance}).
\end{definition}

The Fenchel-Nielsen Teichmüller space $ \teich_{FN}\entre{H_0} $ of a hyperbolic surface $ H_0 $ is defined and studied in \cite{MR2865518}.

The following mapping from $ \teich_{ls} \entre{H_0} $ to the infinite-dimensional space $ \ell^{\infty} $ of bounded sequences in $ \R $
\begin{align}\label{map_to_l}
	H \longmapsto \Bigg( \log \dfrac{\ell_{H}(C_i)}{\ell_{H_0}(C_i)} , \tau_{H}(C_i) - \tau_{H_0}(C_i) \Bigg)_i
\end{align}
is an isometric bijection.

Assuming that the base hyperbolic surface is upper bounded, a parametrization of the quasiconformal Teichmüller space by Fenchel-Nielsen coordinates is given in \cite[Theorem 6.2]{MR3449399}. A marked hyperbolic structure $ \big[f, H\big] $ with Fenchel-Nielsen coordinates $ \big( \ell_{H}(C_i) , \tau_{H} (C_i ) \big) $ is in $ \teich_{ls}\entre{H_0}$ if and only if there is an $ N > 0 $ such that
\begin{align}\label{LS_ell}
	\bigg| \log \dfrac{\ell_{H}(C_i)}{\ell_{H_{0}}(C_i)} \bigg| \leq N  
\end{align}
and
\begin{align}\label{LS_twist}
	\dfrac{\big|  \tau_{H_0}(C_i) - \tau_{H}(C_i) \big|}{\max \big\{ \big| 1, \log \ell_{H} (C_i) \big| \big\} } \leq N 
\end{align}
for all $ i $. We shall also note that in \cite[Theorem 8.10]{MR2865518} it is proven that a marked hyperbolic structure $ \big[f, H\big] $ with Fenchel-Nielsen coordinates $ \big( \ell_{H}(C_i) , \tau_{H} (C_i ) \big) $ is in $ \teich_{qc}\entre{H_0}$ if and only if there is an $ N > 0 $ such that
\begin{align*}
	\bigg| \log \dfrac{\ell_{H}(C_i)}{\ell_{H_{0}}(C_i)} \bigg| \leq N  
\end{align*}
and
\begin{align*}
	\big| \tau_{H_0}(C_i) - \tau_{H}(C_i) \big| \leq N 
\end{align*}
for all $ i $.

Moreover it is also proven in the same paper that the map ($ \ref{map_to_l}$) is a locally bi-Lipschitz homeomorphism between the length-spectrum Teichmüller space $ \teich_{ls}\entre{H_0} $ and the space $ \ell^{\infty} $ of bounded sequences (under the condition that $ H_0 $ is upper bounded). Now, recall that we have shown that $ \teich^{fs} \entre{H_0} $ is a proper subspace of $ \teich_{ls} \entre{H_0} $, therefore, in order to find a parametrization of the finitely supported Teichmüller space $ \teich^{fs} \entre{H_0} $ we shall find the image of the restriction of the isometry  between $ \teich_{ls}\entre{H_0} $ and $ \ell^{\infty} $ to the subspace $ \teich^{fs} \entre{H_0} $.

\subsection{Fenchel-Nielsen parametrization of $ \teich^{fs} \entre{H_0} $}

\begin{theorem}\label{theorem:FN_T_fs}
	Suppose that the base hyperbolic surface $ H_0 $ is upper bounded. A marked hyperbolic structure $ \big[f, H\big] $ with Fenchel-Nielsen coordinates $ \big( \ell_{H}(C_i) , \tau_{H} (C_i ) \big) $ is in $ \teich^{fs}\entre{H_0} $ if and only if there is an $ i_{0} \geq 0 $ such that $ \ell_{H}(C_i) = \ell_{H_{0}}(C_i) $ and $ \tau_{H_0}(C_i) = \tau_{H}(C_i) $ for all $ i \geq i_0 $. 
	
	As a corollary, there is a bijective isometry between $ \teich^{fs}\entre{H_0} $ and the infinite-dimensional space $ \bigoplus_{\N} \R $ consisting of all sequences which have only finitely many nonzero terms, endowed with the supremum metric.
\end{theorem}

\begin{proof}
	We begin with the map
	\begin{align*}
		\teich^{fs}\entre{H_0} \ni \big[f,H\big] \longmapsto \Bigg( \log \dfrac{\ell_{H}(C_i)}{\ell_{H_0}(C_i)} , \tau_{H}(C_i) - \tau_{H_0}(C_i) \Bigg)_i \in  \bigoplus_{\N} \R 
	\end{align*}
	and show that it is well-defined. To achieve this, it suffices to verify the following
	
	\textit{Claim:} Let $ H_0 $ be a Nielsen-convex hyperbolic structure on a surface $ \Sigma $ of infinite type and let $ \P = \{ C_{i}\}_{i \in I} $ hyperbolic pair of pants decomposition on $ \Sigma $. Then for any finite type subsurface $ E $ of $ \Sigma $ we have $ E \cap C_{i} = \emptyset $ for all but finitely many $ i \in I $. 
	
	Let us prove this claim. For each $ i \in I $ let $ \X_{i} $ be the subsurface of type $ (0,4) $ (or $ (1,1) $) whose boundary components are curves from the pair of pants decomposition $ \{ C_{i}\}_{i \in I} $ and that contains the curve $ C_{i} $. (We allow that $\X_i$ has punctures.) Clearly, the collection of all subsurfaces $ \X_{i} $ for $ i \in I $ covers the surface $ \Sigma $. For each component $ \gamma $ of $\partial \X_{i} $, using the collar lemma, consider the collar neighborhood $ \omega( \gamma ) $. For each $ i \in I $, the set $ \U_{i} = \X_{i} \setminus \bigcup_{\gamma} \omega(\gamma) $ where the union is taken on the set of the boundary components $ \gamma $ of $ \X_{i} $ is open. 
	\begin{figure}[ht!]
		\centering
		\hspace{-1.5cm}
		\includegraphics[width=0.25\linewidth]{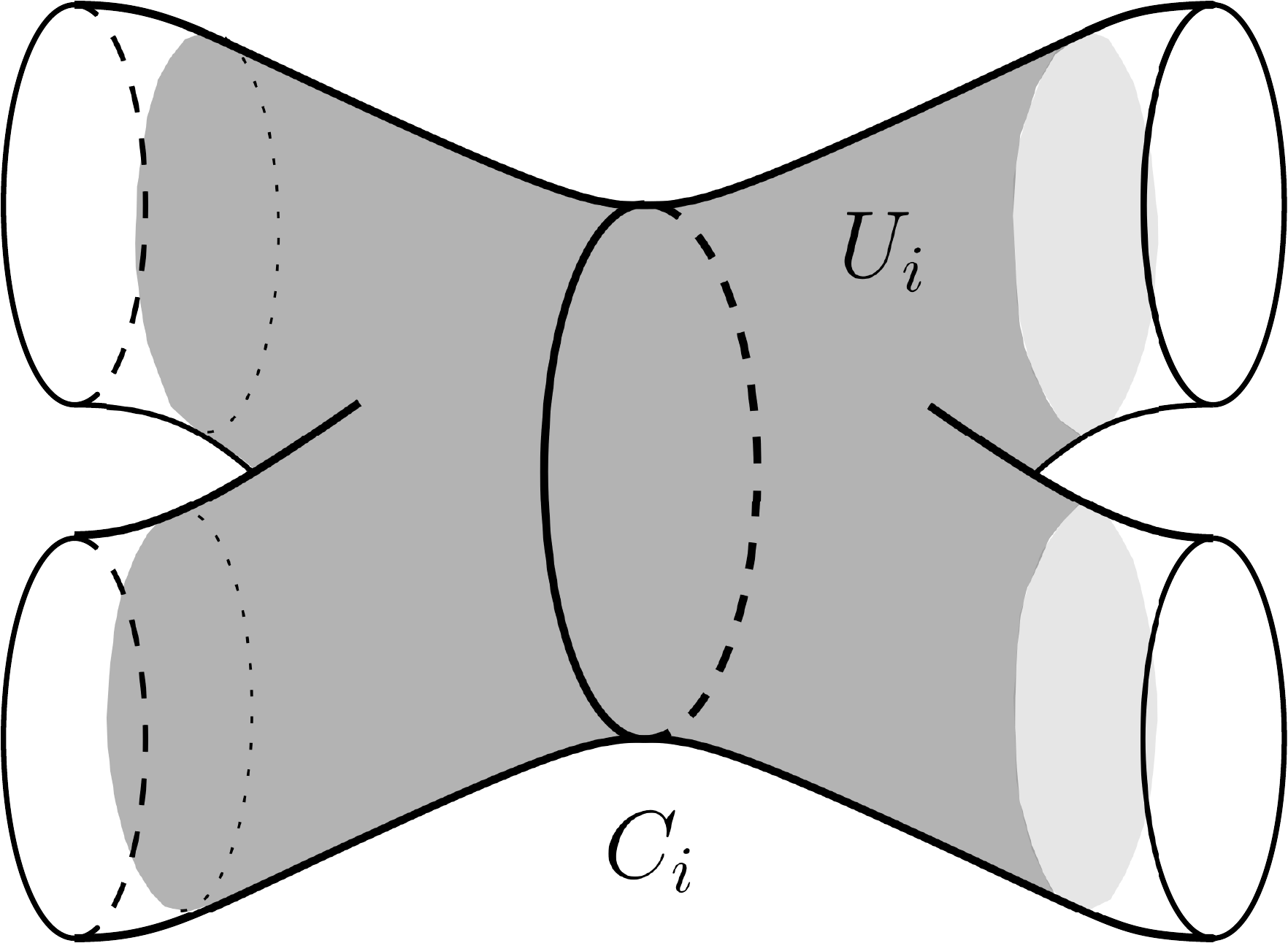}
		\caption{\small {}\label{fig:U_i}}
		\vspace{-.2cm}
	\end{figure} 
	It is clear that $ \U = \{ \U_i\}_{i \in I}$ is cover of the surface $ \Sigma $. Since the collars for different curves are setwise distinct, each $ \U_{i} $ contains exactly one element from $ \P $, namely $ C_{i} $.
	Now let $ E $ be a finite type subsurface of $ \Sigma $. Without loss of generality, we can assume that $ E $ is compact. Consider the subcollection $\U' $ of $ \U $ consisting all the sets $\U_{j}$ for which $ K \cap \U_j $ is nonempty. Since $ E $ is compact, $ \U' $ has a finite subcover. This shows that $ E $ intersects with only a finite number of curves in $ \P = \{ C_{i}\}_{i \in I}$, as claimed.
	
	Now let $ x = \big[f, H\big] $ with Fenchel-Nielsen coordinates $ \big( \ell_{H}(C_i) , \tau_{H} (C_i ) \big) $ be an element of $ \teich^{fs} \entre{H_0} $ supported on a subsurface $ E $ of $ \Sigma $. From the definition and the claim above, we see that only finitely many of $ C_i $ crosses $ E $. So the length and twist parameters remains unchanged for all but finitely many of the curves in $ \P $. 
	
	To show the reverse direction of the statement we need to prove that the map is a surjection. Let  $ \big[f, H\big] $ be an marked hyperbolic surface with Fenchel-Nielsen coordinates $ \big( \ell_{H}(C_i) , \tau_{H} (C_i ) \big) $ satisfying that there is an $ i_0 $ such that $ \ell_{H}(C_i) = \ell_{H_{0}}(C_i) $ and $ \tau_{H_0}(C_i) = \tau_{H}(C_i) $ for all $ i \geq i_0 $. Consider the set of all pair of pants in $ \P $ which has at least one  boundary component $ C_i $ for $ i \leq i_0 $ and let $ E $ be the union of them. Then $ E $ is subsurface of $ \Sigma $ and $ \big[f,H\big] $ is finitely supported on $ \Sigma $.
	The map above is an isometry since it is the restriction of the isometry  between $ \teich_{ls}\entre{H_0} $ and $ \ell^{\infty} $. 
\end{proof}

\subsection{Fenchel-Nielsen parametrization of $ \teich^{0}_{ls} \entre{H_0} $}

The Fenchel-Nielsen parametrization of the space of asymptotically conformal structures is announced by \v{S}ari\'{c} \cite[Theorem 1]{MR3467703}. It states that a marked hyperbolic structure $ \big[f, H\big] $ with Fenchel-Nielsen coordinates $ \big( \ell_{H}(C_i) , \tau_{H} (C_i ) \big) $ is in $ \teich^{0}_{qc}\entre{H_0} $ if and only if we have
\begin{align*}
	\bigg| \log \dfrac{\ell_{H}(C_i)}{\ell_{H_{0}}(C_i)} \bigg| \longrightarrow 0  
\end{align*}
and
\begin{align*}
	\big|  \tau_{H_0}(C_i) - \tau_{H}(C_i) \big| \longrightarrow 0  
\end{align*}
as $ i \rightarrow \infty $.

Now we want to prove a similar result for the space of asymptotically isometric hyperbolic structures. 
\begin{theorem}\label{theorem:FN_T_0}
	Suppose that the base hyperbolic surface $ H_0 $ is upper bounded. A marked hyperbolic structure $ \big[f, H\big] $ with Fenchel-Nielsen coordinates $ \big( \ell_{H}(C_i) , \tau_{H} (C_i ) \big) $ is in $ \teich^{0}_{ls}\entre{H_0} $ if and only if we have
	\begin{align}\label{FN_ell}
		\bigg| \log \dfrac{\ell_{H}(C_i)}{\ell_{H_{0}}(C_i)} \bigg| \longrightarrow 0  
	\end{align}
	and
	\begin{align}\label{FN_twist}
		\dfrac{\big|  \tau_{H_0}(C_i) - \tau_{H}(C_i) \big|}{\max \big\{ \big| 1, \log \ell_{H_0} (C_i) \big| \big\} } \longrightarrow 0  
	\end{align}
	as $ i \rightarrow \infty $. 
\end{theorem}

\begin{proof}
	Suppose that $ \big[f,H\big] $ is a marked hyperbolic structure with Fenchel-Nielsen coordinates $ \big( \ell_{H}(C_i) , \tau_{H} (C_i ) \big) $ satisfying the conditions (\ref{FN_ell}) and (\ref{FN_twist}). We want to show that $ \big[f,H\big] $ is asymptotically isometric to the base surface. 
	
	Since $ H_0 $ is upper bounded we can follow \cite{MR3449399} to construct a $ \big[f',H'\big] $ where $ f $ is length-spectrum bounded with respect to $ H_0 $ and has coordinates  
	\begin{align*}
		\Bigg( \ell_{H}(C_i) , \dfrac{\ell_{H_0}(C_i)}{\ell_{H'}(C_i)}\tau_{H_0} (C_i ) \Bigg) 
	\end{align*}
	because it satisfies (\ref{FN_ell}) (hence (\ref{LS_ell}) as well). We shall note that $ \big[f',H'\big] $ is constructed by gluing pairs of pants so that all the twist parameters remains unchanged. Let us verify that $ \big[f',H'\big] $ is asymptotically isometric to the base surface. Let $ \epsilon > 0 $. From the condition (\ref{FN_ell}), there is a $ j $ such that if $ i > j $ then 
	\begin{align*}
		\bigg| \log \dfrac{\ell_{H}(C_i)}{\ell_{H_{0}}(C_i)} \bigg| < 1 + \epsilon  
	\end{align*}
	Let $ E $ be the subsurface of $ H_0 $ which is the union of all pairs of pants having a $ C_i $ as a boundary curve for some $ i \leq j $. Then $ \L (f'_{\restriction_{\Sigma-E}}) < 1 + \epsilon $. 
	
	We can find a $ f'' : H' \rightarrow H'' $ is length-spectrum bounded and has coordinates
	\begin{align*}
		\Bigg( \ell_{H}(C_i) , \tau_{H} (C_i ) - \dfrac{\ell_{H_0}(C_i)}{\ell_{H'}(C_i)}\tau_{H_0} (C_i ) \Bigg) 
	\end{align*}
	which is obtained by performing a multi-twist on $ H' $. It follows from the conditions \ref{FN_ell} and \ref{FN_twist} that the quantity 
	\begin{align*}
		\bigg| \tau_{H} (C_i ) - \dfrac{\ell_{H_0}(C_i)}{\ell_{H'}(C_i)}\tau_{H_0} (C_i ) \bigg| \longrightarrow 0
	\end{align*}
	as $ i \rightarrow \infty $. Therefore, again by following \cite{MR3449399}, $ f'' $ is length-spectrum bounded. It is an asymptotic isometry as well by an argument similar to the one above. 
	
	Now, $ \big[f'' \circ f' , H''\big] $ has the same Fenchel-Nielsen coordinates as $ \big[f,H\big] $ and the marking $ f'' \circ f' $ is an asymptotic isometry. We conclude that $ \big[f,H\big] \in \teich^{0}_{ls} \entre{H_0} $. 
	
	Let us verify the converse direction. Suppose that $ \big[f,H \big] $ is an asymptotically isometric to the base surface $ H_0 $. Define a sequence $ \big[f_n, H_n\big] $ of finitely supported hyperbolic structures on $ \Sigma $ as follows. Put $ f_0 = \id $. For each $ n \geq 1 $, let $ E_n $ be the subsurface of $ \Sigma $ which is the union of all pairs of pants having $ C_n $ as a boundary curve for some $ i \leq n $. Then $ E_n $ forms an increasing sequence of subsurfaces whose union is $ \Sigma $. Define $ \big[f_n, H_n\big] $ such that it coincides with $ \big[f,H\big] $ on $ E_n $ and with $ [Id, H_0] $ outside $ E_n $. We note that the Fenchel-Nielsen coordinates corresponding to $ \big[f_n, H_n\big] $ satisfies
	\begin{equation*}
		\ell_{H_n}(C_i) = 
		\begin{cases} 
			\ell_{H}(C_i) & \text{if } i < n \\
			\ell_{H_0}(C_i) & \text{if } i \geq n\\
		\end{cases}
		\quad
		\text{ and }
		\quad
		\tau_{H_n}(C_i) = 
		\begin{cases} 
			\tau_{H}(C_i) & \text{if } i < n \\
			\tau_{H_0}(C_i) & \text{if } i \geq n\\
		\end{cases}
	\end{equation*}
	By \autoref{theorem:ls_dense} and from its proof we see that $ \big[f_n, H_n\big] \in \teich^{0}_{ls} \entre{H_0} $ and \begin{align*}
		d_{ls}\bigg( \big[f_n, H_n\big] , \big[f,H\big] \bigg) \rightarrow 0 
	\end{align*}
	It remains to show that each $ \big[f_n, H_n\big] $ satisfies the conditions (\ref{FN_ell}) and (\ref{FN_twist}). Since $ \ell_{H}(C_i) = \ell_{H_0}(C_i) $ for all $ i > n $ we have 
	\begin{align*}
		\dfrac{\ell_{H}(C_i)}{\ell_{H_0}(C_i)} = 1
	\end{align*}
	and taking logarithms,
	\begin{align*}
		\bigg| \log \dfrac{\ell_{H}(C_i)}{\ell_{H_0}(C_i)} \bigg| \longrightarrow 0
	\end{align*}
	which verifies that $ \big[f_n, H_n\big] $ satisfies (\ref{FN_ell}). A similar argument shows also that $ \big[f_n, H_n\big] $ satisfies (\ref{FN_twist}). Then by the continuity of the map $ \teich_{ls} \entre{H_0} $ to $ \ell^{\infty} $, $ \big[f,H\big] $ satisfies the conditions in the statement of the theorem.      
\end{proof}

\begin{remark}
	The proof also shows that the requirement that the markings are length-spectrum bounded is superfluous in the  \autoref{definition:asymptotic_definition} since the conditions (\ref{LS_ell}) and (\ref{LS_twist}) implies to those of the theorem above.
\end{remark}

\subsection{Asymptotic length-spectrum Teichmüller space $ \asyteich_{ls} \entre{H_0} $}
Analogous to the notion of asymptotic Teichmüller space $ \asyteich \entre{H_0} $ which has been studied by several authors, for instance Earle, Markovic, \v{S}ari\'{c} \cite{MR1940165}, Fletcher \cite{MR2584608}, \cite{MR2254550}, Matsuzaki \cite{MR2302578} and \v{S}ari\'{c} \cite{MR3467703} we define
\begin{definition}
	We say that two marked length-spectrum bounded hyperbolic surfaces $ (f_1,H_1) $ and $ (f_2, H_2) $  are \textbf{asymptotically equivalent} if there is a homeomorphism $ \varphi : \Sigma \rightarrow \Sigma $ homotopic to identity and an asymptotic  isometry $ h : H_1 \rightarrow H_2 $ such that following diagram commutes: 
	\[
	\begin{tikzcd}
		\Sigma  \arrow[swap]{d}{\varphi} \arrow{r}{f_{1}}  & H_1 \arrow{d}{h} \\
		\Sigma \arrow{r}{f_{2}}  & H_2
	\end{tikzcd}
	\]This is an equivalence relation on the set of marked  hyperbolic structures where the markings are length-spectrum bounded with respect to the base hyperbolic structure.
	
	The \textbf{asymptotic length-spectrum Teichmüller space} $\asyteich_{ls}\entre{H_0}$ of a hyperbolic surface $H_0$ is the space consisting of all asymptotical equivalence classes $ \big[[f,H]\big] $  of length-spectrum bounded marked hyperbolic structures $f:H_0 \rightarrow H$.
\end{definition}

\begin{remark} 
	It is clear that If $ h: H_1 \rightarrow H_2 $ is an isometry then $ \K(h)=1 $ (the converse is also true, see \cite[Theorem 2.2]{MR2792982}), so it follows that every isometry is asymptotically length-spectrum bounded. In other words, the requirement that asymptotic equivalence on $ h $ in the diagram above is weaker than being isometric, hence in a sense, the space $ \asyteich_{ls}\entre{H_0} $ is smaller compared to $ \teich_{ls} \entre{H_0}$. More precisely, we have a projection map 
	$$\pi : \teich_{ls}\entre{H_0}  \rightarrow \asyteich_{ls}\entre{H_0} $$ defined by $ \big[f,H\big] \mapsto  \big[[f,H]\big]  $. It follows from the definitions that the space $ \teich_{0}^{ls} \entre{H_0} $ is the inverse image of $ [\id , H_0] $ under the projection map $ \pi $, namely $ \pi^{-1}([\id , H_0]) $. It is convenient to identify $ \asyteich_{ls}\entre{H_0} $ with the quotient $ \teich_{ls}\entre{H_0} /  \teich_{0}^{ls} \entre{H_0}  $.
\end{remark}


\begin{corollary}[of the \autoref{theorem:FN_T_0}]\label{corollary:asym_contractible}
	Let $ \Sigma $ be a surface of topologically infinite type endowed with an upper bounded hyperbolic structure $ H_0 $. The asymptotically length-spectrum Teichmüller space $ \asyteich_{ls}\entre{H_0} $ is contractible in the length-spectrum metric.
\end{corollary}

\begin{proof}
	Since $ \teich_{ls}\entre{H_0} $ is homeomorphic to the space $ \ell^{\infty} $ of bounded sequences and $ \teich^{0}_{ls}\entre{H_0} $ is homeomorphic to the space $ c_0 $ of sequences converging to $ 0 $ by the map (\ref{map_to_l}), it follows that $ \asyteich_{ls}\entre{H_0} = \teich_{ls}\entre{H_0} / \teich^{0}_{ls} \entre{H_0}$ is homeomorphic to $ \ell^{\infty} / c_0 $ which is contractible. 
\end{proof}

\section{Finitely supported mapping class groups}\label{{S}ection:mcg}
The present section concerns the mapping class groups associated to a surface of topologically infinite type. We study the action of finitely supported mapping class group $ \mcg^{fs}(\Sigma)$ on the finitely supported Teichmüller space and show that the orbits are non-discrete when $ H_0 $ admits short curves. As we have mentioned in the introduction, we do not require the homotopies to preserve the boundary components pointwise. We deal with the reduced theory where the homotopies preserve boundary components setwise and preserves the punctures pointwise.
In keeping with this convention, throughout this section we shall omit the word \textit{reduced} in our terminology. 

We recall that the orientation-preserving homeomorphisms which are homotopic to the identity map $\text{Id}_{\Sigma} : \Sigma \to \Sigma $ forms a normal subgroup of the group $\textsf{Homeo}\entre{\Sigma} $ of self-homeomorphisms of $ \Sigma $. The quotient group is \textit{the mapping class group} of $ \Sigma $ denoted by $\mcg\entre{\Sigma}$. The \textit{length-spectrum mapping class group} $ \mcg_{ls}\entre{H_0} $ of $ \Sigma $ based at the point $ \big[\id, H_0\big]$ is defined to be the subgroup of $ \mcg\entre{\Sigma}$ consisting of homotopy classes of homeomorphisms which are length-spectrum bounded with respect to $H_0$. 

\begin{definition}
	The set of homotopy classes of homeomorphisms $ \varphi : \Sigma \rightarrow \Sigma $  satisfying  $ \varphi_{\restriction_{\Sigma \setminus E}} = Id_{\restriction_{\Sigma \setminus E}}$ for some finite type subsurface $ E $ of $ \Sigma $ is a group called \textit{\textbf{finitely supported mapping class group}} of $ \Sigma $ and denoted by $\mcg^{fs}\entre{\Sigma}$. 
\end{definition}

It follows directly from the definitions that both the finitely mapping class group $\mcg^{fs}\entre{\Sigma}$ and the length-spectrum mapping class group $ \mcg_{ls}\entre{H_0} $ are subgroups of $ \mcg\entre{\Sigma}$. Furthermore, we have the following set relations which have been proven in \cite{MR2792982} (cf. Propositions 2.6 and 2.12), and here we adapt it to the terminology of finitely supported mapping class groups. The proof can be also seen as a corollary of \autoref{theorem:finitely_supported_is_ls} because any finitely supported mapping class $ \varphi $ when regarded as a marking $ \varphi : H_0 \rightarrow H_0 $ is length-spectrum bounded.

\begin{proposition}
	For the finitely supported mapping class group $ \mcg^{fs}\entre{\Sigma}$; the length-spectrum mapping class group $\mcg_{ls}\entre{H_0} $; and the mapping class group $ \mcg\entre{\Sigma} $ we have 
	\begin{align*}
		\mcg^{fs}\entre{\Sigma}  \subsetneq \mcg_{ls}\entre{H_0} \subsetneq \mcg\entre{\Sigma} 
	\end{align*}
\end{proposition}
\begin{proof}
	Let $ \varphi : \Sigma  \rightarrow \Sigma $ be a mapping class supported on a finite type subsurface $ E $ of the surface $ \Sigma $. Since $ \varphi(E) = E $ and every mapping class of a surface of finite type is a finite composition of Dehn twists,  we can assume $ \varphi $ is a Dehn twist $ \dehn_{\beta} $ about an essential simple closed curve $ \beta $ without losing the generality. Let $ \alpha $ be a homotopy class of essential simple closed curves in  $ \Sigma $. We need to show that 
	\begin{align*}
		\sup_{\alpha \in \mathscr{S}\entre{\Sigma}} \Bigg\{\log  \dfrac{\ell_{H_0}(\dehn_{\beta}(\alpha))}{\ell_{H_0}(\alpha)} \text{ , } \log \dfrac{\ell_{H_0}(\alpha)}{\ell_{H_0}(\dehn_{\beta}(\alpha))} \Bigg\} < \infty 
	\end{align*}	
	If $ i(\alpha, \beta) = 0 $ then $ \dehn_{\beta}(\alpha) = \alpha $ so that $ \ell(\dehn_{\beta}(\alpha)) = \ell(\alpha) $. We assume $ i(\alpha, \beta) \neq 0 $. By the collar lemma there is a positive number $ \omega(\beta) $ such that the $ \omega(\beta) $-tubular neighborhood of $ \beta $ is embedded in $ \Sigma $, and the number $ \omega(\beta) $ depends only on the $H$-hyperbolic length of the curve $ \beta $. We have 
	\begin{align*}
		\ell_{H_0}(\dehn_{\beta}(\alpha)) \leq \ell_{H_0}(\alpha) + i(\alpha,\beta).\ell_{H_0}(\beta)
	\end{align*}
	and by dividing both sides of this inequality by $ \ell_{H_0}(\alpha) $, we get
	\begin{align*}
		\dfrac{\ell_{H_0}(\dehn_{\beta}(\alpha))}{\ell_{H_0}(\alpha)} \leq 1 +\dfrac{ i(\alpha,\beta).\ell_{H_0}(\beta)}{\ell_{H_0}(\alpha)} \leq 1 +\dfrac{\ell_{H_0}(\beta)}{\omega(\beta)}
	\end{align*}
	where the inequality on right hand side follows from $ \ell_{H_0}(\alpha)  \geq i(\alpha,\beta)\omega(\beta) $ which holds because $ \alpha $ traverses the $ \omega(\beta) $-tubular neighborhood of $ \beta $ exactly $ i(\alpha, \beta )$ times. 
	We have, 
	\begin{align}\label{first}
		\log\dfrac{\ell_{H_0}(\dehn_{\beta}(\alpha))}{\ell_{H_0}(\alpha)}
		\leq
		\log \Bigg( 1 +\dfrac{\ell_{H_0}(\beta)}{\omega(\beta)} \Bigg)
		\leq 
		\dfrac{\ell_{H_0}(\beta)}{\omega(\beta)} 
	\end{align}
	where we have used that $\log(1+x) \leq x \text{ for } x>0 $. Furthermore, by exchanging the roles of $ \ell_{H_0}(\alpha) $ and $ \ell_{H_0}(\dehn_{\beta}(\alpha)) $, we similarly have
	\begin{align*}
		\log \dfrac{\ell_{H_0}(\alpha)}{\ell_{H_0}(\dehn_{\beta}(\alpha))} \leq \dfrac{\ell_{H_0}(\beta)}{\omega(\beta)}
	\end{align*}
	which shows, together with (\ref{first}), that $ \dehn_{\beta} $ is length-spectrum bounded.
\end{proof}

\begin{example}
	We give an example that indicates $ \mcg^{fs}\entre{\Sigma}  \neq \mcg_{ls}\entre{H_0} $. Consider the surface $ \Sigma $ with two non-planar ends drawn in \autoref{figure:mcg_twoends} where the curves $ C_i $ for $ i \in \Z $ are the curves represented. We equip $ \Sigma $ with a hyperbolic structure $ H_0 $ by letting $ \ell_{H_0}(C_i) = 1 $ 
	Consider the homeomorphism $ f : \Sigma \rightarrow \Sigma $ defined by $ f(x) = x + a $ where $ a $ is suggested by the figure. 
	
	\begin{figure}[ht!]
		\centering
		\includegraphics[width=0.75\linewidth]{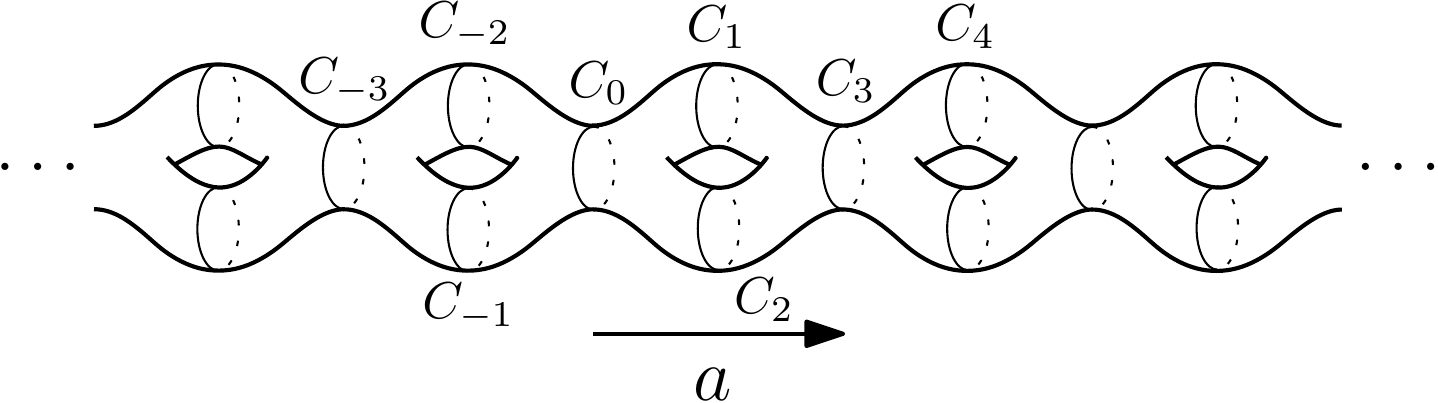}
		\caption{\small {$ f $ shifts $ C_i $ to $ C_{i+3}$}}
		\label{figure:mcg_twoends}
	\end{figure} 
	
	It is clear that $ f \in \mcg_{ls}\entre{H_0} $ because $ \ell_{H}(f(C_{i})) = \ell_{H_{0}}(C_{i+3}) = 1 $ for all $ i $. Note that $ f $ is not homotopic to the identity map since, for instance, it maps $ C_0 $ to $ C_3 $ which are non-homotopic curves. Evidently, $ f $ is not finitely supported.
\end{example}

\begin{lemma}
	The group $ \mcg^{fs}\entre{\Sigma} $ acts on the finitely supported Teichmüller space $\teich^{fs}\entre{H_0}$ as 
	\begin{center}
		$ \varphi . \big[f,H\big] = \big[f \circ \varphi^{-1}, H\big] $
	\end{center}
\end{lemma}
\begin{proof}
	Let $ f: H_0 \rightarrow H $ be a homeomorphism on a hyperbolic surface $ H $ supported on a finite type subsurface $ E $ with respect to $ H_0 $. Let $ \varphi : \Sigma \rightarrow \Sigma $ be a homeomorphism supported on a finitely supported subsurface $ F $. Then $ f \circ \varphi^{-1} $ is supported on the union of $ F $ and $ E=\varphi (E) $ which is clearly of finite type. So the action is well-defined.
\end{proof}

\begin{lemma}\label{lemma:accumulation}
	Suppose that $ H_0 $ is a hyperbolic structure on $ \Sigma $ which admits short curves, that is, there is a sequence of homotopy classes of disjoint simple closed curves $ \alpha_{n}$ on $ H_0 $ such that $\ell_{H_0}(\alpha_{n}) \rightarrow 0 $ as $ n \rightarrow \infty $. Then every element $ \big[f,H\big] $ of $ \teich^{fs} \entre{H_0} $ admits short curves.
\end{lemma}

\begin{proof}
	Let $ \entre{f,H} $ be a finitely supported hyperbolic surface. From the definition, there exists a finite type subsurface $ E $ of $ \Sigma $ such that $H_{0 \restriction_{\Sigma \setminus E}}$ and $ f^{*}(H)_{\restriction_{\Sigma \setminus E}}$ are the same hyperbolic structure. We claim that $ \ell_{H} (f(\alpha_{n})) $ goes to $ 0 $ as $ n \to \infty $. If there is an $ N $ such that $ \alpha_{n} \cap E = \emptyset $ for all $ n > N $ then we have $ \ell_H(f(\alpha_{n})) = \ell_{H_{0}}(\alpha_{n}) $ so we are done. Suppose that $ \alpha_n $ intersects $ E $ for infinitely many $ n $. By passing to a subsequence of it we may assume that $ \alpha_n $ intersects a simple closed curve $ \beta $ in $ E $ for infinitely many $ n $. Consider the collar about each $ \alpha_n $ of width $ \omega(\alpha_n) $. Since $ \ell_{H} (\alpha_{n}) $ approaches to $ 0 $, the width of collars $ \omega(\alpha_{n}) \to \infty $ by collar lemma. This contradicts with that $ \beta $ is a simple closed curve. It turns out that the $ \ell_H(f(\alpha_{n})) \to 0 $. Therefore, we have verified the assertion above.
\end{proof}

\begin{theorem}\label{theorem:nondiscreteness}
	Suppose that for the base hyperbolic structure $ H_0 $ on $ \Sigma $ there is a sequence of homotopy classes of disjoint simple closed curves $ \alpha_{n}$ with $\ell_{H_0}(\alpha_{n}) \rightarrow 0 $ as $ n \rightarrow \infty $. For each natural number $ n \geq 1 $ denote by $ \dehn_{\alpha_n} $ the positive Dehn twist about $ \alpha_n $. Let $ \big[f,H\big] $ be an element of $ \teich^{fs} \entre{H_0} $. Then 
	\begin{align*}
		d_{ls}\bigg( \big[f \circ \dehn_{\alpha_{n}}^{-1},H \big], \big[f,H\big] \bigg) 
		\longrightarrow 0
	\end{align*}
	In particular, under the condition above, the action of the finitely supported mapping class group $ \Gamma_{fs}\entre{H_0} $ on $ \teich^{fs}\entre{H_0} $ is not discrete. 
\end{theorem}

\begin{proof}
	Firstly we note that $ f \circ (f \circ \dehn_{\alpha_n}^{-1})^{-1} = f \circ \dehn_{\alpha_n} \circ f^{-1} $ is homotopic to $ \dehn_{f(\alpha_n)} $ that is, the positive Dehn twist on $ H $ about the simple closed curve $ f(\alpha_n) $.
	We need to show that 
	\begin{align*}
		\sup_{\alpha \in \mathscr{S}(H)} 
		\Bigg\{
		\log
		\dfrac{ \ell_{H} \big(\dehn_{f(\alpha_n)}(\alpha) \big) }{ \ell_{H}(\alpha) }
		\text{ , }
		\log
		\dfrac{\ell_{H}(\alpha)}{\ell_{H} \big(\dehn_{f(\alpha_n)}(\alpha) \big)} 
		\Bigg\} 
		\longrightarrow
		0 
		\text{ as }
		n \rightarrow \infty
	\end{align*}
	Let $ \alpha \in \mathscr{S}(H) $ and $ n $ be a positive integer. If $ i (\alpha, \alpha_{n}) = 0 $ then $ \ell_{H} \big(\dehn_{f(\alpha_n)}(\alpha) \big) = \ell_{H}(\alpha) $. We suppose that $ i (\alpha, \alpha_{n}) \neq 0 $ and denote $ \epsilon_{n} = \ell_{H}\big( \dehn_{f(\alpha_n)} \big) $. By Lemma \ref{lemma:accumulation} we have $ \epsilon_{n} \rightarrow 0 $. By collar lemma, there exists a real number $ \delta_{\alpha} > 0 $ depending on $ \ell_{H}(\alpha) $ and on $ \epsilon_{n} $ such that 
	\begin{align*}
		\ell_{H}(\alpha) = i (\alpha, \alpha_{n}) \big| \log \epsilon_{n} \big| + \delta_{\alpha}
	\end{align*}
	We have from the definition of Dehn twists,
	\begin{align*}
		i (\alpha, \alpha_{n}) \big| \log \epsilon_{n} \big| + \delta_{\alpha} - i (\alpha, \alpha_{n})\epsilon_{n}
		\leq
		\ell_{H} \big(\dehn_{f(\alpha_n)}(\alpha) \big)
		\leq
		i (\alpha, \alpha_{n}) \big| \log \epsilon_{n} \big| + \delta_{\alpha} + i (\alpha, \alpha_{n})\epsilon_{n}
	\end{align*}
	So we get 
	\begin{align*}
		\dfrac{\ell_{H} \big(\dehn_{f(\alpha_n)}(\alpha) \big)}{\ell_{H}(\alpha)} 
		& \leq 
		\dfrac{i (\alpha, \alpha_{n}) \big| \log \epsilon_{n} \big| + \delta_{\alpha} + i (\alpha, \alpha_{n})\epsilon_{n}}{i (\alpha, \alpha_{n}) \big| \log \epsilon_{n} \big| + \delta_{\alpha} }\\
		&  = 1 +
		\dfrac{ i (\alpha, \alpha_{n})\epsilon_{n}}{i (\alpha, \alpha_{n}) \big| \log \epsilon_{n} \big| + \delta_{\alpha} }\\
		& \leq 
		1 +
		\dfrac{ i (\alpha, \alpha_{n})\epsilon_{n}}{i (\alpha, \alpha_{n}) \big| \log \epsilon_{n} \big|}\\
		& = 
		1 +
		\dfrac{ \epsilon_{n}}{\big| \log \epsilon_{n} \big|}
	\end{align*}
	and taking logarithms (and using $\log(1+x) \leq x \text{ for } x>0 $),
	\begin{align*}
		\log \dfrac{\ell_{H} \big(\dehn_{f(\alpha_n)}(\alpha) \big)}{\ell_{H}(\alpha)}  
		\leq
		\log \dfrac{ \epsilon_{n}}{\big| \log \epsilon_{n} \big|}
		\longrightarrow 0
	\end{align*}
	as $ n \rightarrow \infty$. 
	We have also
	\begin{align*}
		\dfrac{\ell_{H} \big(\dehn_{f(\alpha_n)}(\alpha) \big)}{\ell_{H}(\alpha)} 
		& \geq 
		\dfrac{i (\alpha, \alpha_{n}) \big| \log \epsilon_{n} \big| + \delta_{\alpha} - i (\alpha, \alpha_{n})\epsilon_{n}}{i (\alpha, \alpha_{n}) \big| \log \epsilon_{n} \big| + \delta_{\alpha} }\\
		&  = 1 -
		\dfrac{ i (\alpha, \alpha_{n})\epsilon_{n}}{i (\alpha, \alpha_{n}) \big| \log \epsilon_{n} \big| + \delta_{\alpha} }\\
		& \geq 
		1 -
		\dfrac{ i (\alpha, \alpha_{n})\epsilon_{n}}{i (\alpha, \alpha_{n}) \big| \log \epsilon_{n} \big|}\\
		& = 
		1 -
		\dfrac{ \epsilon_{n}}{\big| \log \epsilon_{n} \big|}
	\end{align*}
	It follows that 
	\begin{align*}
		\dfrac{\ell_{H}(\alpha)}{\ell_{H} \big(\dehn_{f(\alpha_n)}(\alpha) \big)}
		& \leq 
		1 -
		\dfrac{ \epsilon_{n}}{\big| \log \epsilon_{n} \big|} 
	\end{align*}
	and then that 
	\begin{align*}
		\log \dfrac{\ell_{H}(\alpha)}{\ell_{H} \big(\dehn_{f(\alpha_n)}(\alpha) \big)}
		\longrightarrow 1
	\end{align*}
\end{proof}

We finalize this section by referring the reader to the work of Aramayona and Vlamis \cite{athanase} for a survey of recent developments on mapping class groups of surfaces of topologically infinite type.

\subsection*{Acknowledgements} 

I owe my deepest gratitude to Athanase Papadopoulos whose comments made enormous contribution to this work. I would also like to thank Norbert A'Campo insightful discussions on Teichmüller theory. This research is supported by \textit{Initiative d'Excellence} (IdEx) of the University of Strasbourg.

\vspace{2cm}

\noindent F{\i}rat Ya\c{s}ar \\
Université de Strasbourg \\
{\href{mailto:yasar.math@gmail.com}{\large \texttt{yasar.math@gmail.com}}}


\newpage

\begin{thebibliography}{10}
	
	\bibitem{MR2241787}
	{\sc L.~V. Ahlfors}, {\em Lectures on quasiconformal mappings}, vol.~38 of
	University Lecture Series, American Mathematical Society, Providence, RI,
	second~ed., 2006.
	\newblock With supplemental chapters by C. J. Earle, I. Kra, M. Shishikura and
	J. H. Hubbard.
	
	\bibitem{MR2846324}
	{\sc D.~Alessandrini, L.~Liu, A.~Papadopoulos, and W.~Su}, {\em On various
		{T}eichm\"{u}ller spaces of a surface of infinite topological type}, Proc.
	Amer. Math. Soc., 140 (2012), pp.~561--574.
	
	\bibitem{MR3449399}
	{\sc D.~Alessandrini, L.~Liu, A.~Papadopoulos, and W.~Su}, {\em On the
		inclusion of the quasiconformal {T}eichm\"{u}ller space into the
		length-spectrum {T}eichm\"{u}ller space}, Monatsh. Math., 179 (2016),
	pp.~165--189.
	
	\bibitem{MR2865518}
	{\sc D.~Alessandrini, L.~Liu, A.~Papadopoulos, W.~Su, and Z.~Sun}, {\em On
		{F}enchel-{N}ielsen coordinates on {T}eichm\"{u}ller spaces of surfaces of
		infinite type}, Ann. Acad. Sci. Fenn. Math., 36 (2011), pp.~621--659.
	
	\bibitem{article_1}
	{\sc D.~Alessandrini, A.~Papadopoulos, and F.~Yasar}, {\em Ideal triangulations
		of surfaces of infinite type and quasiconformal {T}eichm\"{u}ller spaces},
	preprint,  (2021).
	
	\bibitem{athanase}
	{\sc J.~Aramayona and N.~G. Vlamis}, {\em Big mapping class groups: An
		overview}, in In the {T}radition of {T}hurston, {G}eometry and {T}opology,
	Papadopoulos, A. and Ohshika K. (eds.)), Springer, 2020, pp.~459--475.
	
	\bibitem{MR0114898}
	{\sc L.~Bers}, {\em Quasiconformal mappings and {T}eichm\"{u}ller's theorem},
	in Analytic functions, Princeton Univ. Press, Princeton, N.J., 1960,
	pp.~89--119.
	
	\bibitem{MR132175}
	\leavevmode\vrule height 2pt depth -1.6pt width 23pt, {\em Uniformization by
		{B}eltrami equations}, Comm. Pure Appl. Math., 14 (1961), pp.~215--228.
	
	\bibitem{MR156969}
	\leavevmode\vrule height 2pt depth -1.6pt width 23pt, {\em The equivalence of
		two definitions of quasiconformal mappings}, Comment. Math. Helv., 37
	(1962/63), pp.~148--154.
	
	\bibitem{MR1954866}
	{\sc C.~J. Bishop}, {\em Quasiconformal mappings of {$Y$}-pieces}, Rev. Mat.
	Iberoamericana, 18 (2002), pp.~627--652.
	
	\bibitem{MR1940165}
	{\sc C.~J. Earle, V.~Markovic, and D.~Saric}, {\em Barycentric extension and
		the {B}ers embedding for asymptotic {T}eichm\"{u}ller space}, in Complex
	manifolds and hyperbolic geometry ({G}uanajuato, 2001), vol.~311 of Contemp.
	Math., Amer. Math. Soc., Providence, RI, 2002, pp.~87--105.
	
	\bibitem{MR1004006}
	{\sc W.~Fenchel}, {\em Elementary geometry in hyperbolic space}, vol.~11 of De
	Gruyter Studies in Mathematics, Walter de Gruyter \& Co., Berlin, 1989.
	\newblock With an editorial by Heinz Bauer.
	
	\bibitem{MR2254550}
	{\sc A.~Fletcher}, {\em Local rigidity of infinite-dimensional
		{T}eichm\"{u}ller spaces}, J. London Math. Soc. (2), 74 (2006), pp.~26--40.
	
	\bibitem{MR2584608}
	\leavevmode\vrule height 2pt depth -1.6pt width 23pt, {\em On asymptotic
		{T}eichm\"{u}ller space}, Trans. Amer. Math. Soc., 362 (2010),
	pp.~2507--2523.
	
	\bibitem{MR1699257}
	{\sc G.~K. Francis and J.~R. Weeks}, {\em Conway's {ZIP} proof}, Amer. Math.
	Monthly, 106 (1999), pp.~393--399.
	
	\bibitem{kerekjarto1923vorlesungen}
	{\sc B.~Ker{\'e}kj{\'a}rt{\'o}}, {\em Vorlesungen {\"u}ber Topologie}, vol.~8,
	J. Springer, 1923.
	
	\bibitem{MR2792982}
	{\sc L.~Liu and A.~Papadopoulos}, {\em Some metrics on {T}eichm\"{u}ller spaces
		of surfaces of infinite type}, Trans. Amer. Math. Soc., 363 (2011),
	pp.~4109--4134.
	
	\bibitem{MR2302578}
	{\sc K.~Matsuzaki}, {\em A quasiconformal mapping class group acting trivially
		on the asymptotic {T}eichm\"{u}ller space}, Proc. Amer. Math. Soc., 135
	(2007), pp.~2573--2579.
	
	\bibitem{MR2308580}
	{\sc A.~O. Prishlyak and K.~I. Mischenko}, {\em Classification of noncompact
		surfaces with boundary}, Methods Funct. Anal. Topology, 13 (2007),
	pp.~62--66.
	
	\bibitem{MR0143186}
	{\sc I.~Richards}, {\em On the classification of noncompact surfaces}, Trans.
	Amer. Math. Soc., 106 (1963), pp.~259--269.
	
	\bibitem{MR575168}
	{\sc H.~Seifert and W.~Threlfall}, {\em Seifert and {T}hrelfall: a textbook of
		topology}, vol.~89 of Pure and Applied Mathematics, Academic Press, Inc.
	[Harcourt Brace Jovanovich, Publishers], New York-London, 1980.
	\newblock Translated from the German edition of 1934 by Michael A. Goldman,
	With a preface by Joan S. Birman, With ``Topology of $3$-dimensional fibered
	spaces'' by Seifert, Translated from the German by Wolfgang Heil.
	
	\bibitem{MR1996441}
	{\sc H.~Shiga}, {\em On a distance defined by the length spectrum of
		{T}eichm\"{u}ller space}, Ann. Acad. Sci. Fenn. Math., 28 (2003),
	pp.~315--326.
	
	\bibitem{MR1435975}
	{\sc W.~P. Thurston}, {\em Three-dimensional geometry and topology. {V}ol. 1},
	vol.~35 of Princeton Mathematical Series, Princeton University Press,
	Princeton, NJ, 1997.
	\newblock Edited by Silvio Levy.
	
	\bibitem{MR3467703}
	{\sc D.~\v{S}ari\'{c}}, {\em Fenchel-{N}ielsen coordinates for asymptotically
		conformal deformations}, Ann. Acad. Sci. Fenn. Math., 41 (2016),
	pp.~167--176.
	
	\bibitem{MR624836}
	{\sc S.~Wolpert}, {\em The length spectrum as moduli for compact {R}iemann
		surfaces}, 97 (1981), pp.~515--517.
	
	\bibitem{MR657237}
	\leavevmode\vrule height 2pt depth -1.6pt width 23pt, {\em The
		{F}enchel-{N}ielsen deformation}, Ann. of Math. (2), 115 (1982),
	pp.~501--528.
	
\end{thebibliography}
\end{document}